\documentclass[12pt]{amsart}
\usepackage{amsmath,amsfonts,amsthm,amscd,amssymb,stmaryrd}
\usepackage{graphicx}
\newtheorem{theorem}{Theorem}
\newtheorem{claim}[theorem]{Claim}
\newtheorem{proposition}[theorem]{Proposition}
\newtheorem{lemma}[theorem]{Lemma}

\newtheorem{remark}[theorem]{Remark}
\numberwithin{equation}{section}
\numberwithin{theorem}{section}
\numberwithin{figure}{section}
\numberwithin{table}{section}
\newcommand{\ZZ}{\mathbb{Z}}

\newcommand{\RR}{\mathbb{R}}
\newcommand{\RP}{\mathbb{RP}}
\newcommand{\CP}{\mathbb{CP}}
\newcommand{\CC}{\mathbb{C}}

\newcommand{\I}{\sqrt{-1}}
\newcommand{\bint}{\oint_{\partial B}}
\newcommand{\DN}{\frac{\partial}{\partial N}}

\newcommand{\oc}{\overset{\circ}}
\newcommand{\h}{\tilde{h}}
\newcommand{\tlambda}{\tilde{\lambda}}
\newcommand{\tmu}{\tilde{\mu}}
\newcommand{\tnu}{\tilde{\nu}}
\newcommand{\tW}{\widetilde{W}}
\begin{document}
\bibliographystyle{amsalpha}
\title[Critical metrics]{Critical metrics on connected sums of \\Einstein four-manifolds}
\author{Matthew J. Gursky}
\address{Department of Mathematics, University of Notre Dame, Notre Dame, IN 46556}
\email{mgursky@nd.edu}
\author{Jeff A. Viaclovsky}
\address{Department of Mathematics, University of Wisconsin, Madison, WI 53706}
\email{jeffv@math.wisc.edu}
\date{March 4, 2013}
\begin{abstract}
 We develop a gluing procedure designed to obtain
canonical metrics on connected sums of Einstein four-manifolds.
The main application is an existence result,
using two well-known Einstein manifolds
as building blocks:
the Fubini-Study metric on $\CP^2$ and the product metric
on $S^2 \times S^2$. Using these metrics in various gluing configurations,
critical metrics are found on connected sums for a
specific Riemannian functional, which depends on the global geometry of the factors.
Furthermore, using certain quotients of $S^2 \times S^2$ as one of the gluing factors,
critical metrics on several non-simply-connected manifolds are also obtained.
\end{abstract}
\maketitle
\setcounter{tocdepth}{1}
\tableofcontents
\section{Introduction}
A Riemannian manifold $(M^4,g)$ in dimension four is critical for the
Einstein-Hilbert functional
\begin{align}
\mathcal{R}(g) = Vol(g)^{-1/2} \int_M R_g dV_g,
 \end{align}
where $R_g$ is the scalar curvature, if and only if it satisfies
\begin{align}
Ric(g) = \lambda \cdot g,
\end{align}
where $\lambda$ is a constant; such Riemannian manifolds
are called {\em{Einstein manifolds}}.
Non-collapsing limits of Einstein manifolds have been
studied in great depth \cite{Anderson, BKN, Tian}.
In particular, with certain geometric conditions, the limit
space is an orbifold, with asymptotically locally Euclidean
(ALE) spaces bubbling off at the singular points.
A natural question is whether it is possible to reverse this
process: can one start with the limit space, and glue on
a bubble in order to obtain an Einstein metric?
A recent article of Olivier Biquard makes great strides
in the Poincar\'e-Einstein setting \cite{Biquard2011}.
In this work it is shown that a $\ZZ/2\ZZ$-orbifold singularity $p$
of a non-degenerate Poincar\'e-Einstein orbifold $(M,g)$
has a Poinar\'e-Einstein resolution obtained by gluing on an
Eguchi-Hanson metric if and only if the condition
\begin{align}
\det( \mathbf{R}^+ (p)) = 0
\end{align}
is satisfied, where $\mathbf{R}^+(p) : \Lambda^2_+ \rightarrow \Lambda^2_+$
is the purely self-dual part of the curvature operator at $p$.
The self-adjointness of this gluing problem is overcome by
the freedom of changing the boundary data of the Poincar\'e-Einstein
metric.

However, there is not much known about gluing compact
manifolds together in the Einstein case.
In this work, we will replace the Einstein equations with
a generalization of the Einstein condition. Namely, we ask
whether it is possible to glue together Einstein metrics
and produce a critical point of a certain Riemannian functional
generalizing the Einstein-Hilbert functional.
It turns out that there is a family of such functionals; this gives
an extra parameter which will allow us to overcome the
self-adjointness of this problem.
The particular functional will then depend on the global geometry
of the gluing factors.

 To describe the functionals,  let $M$ be a closed manifold of dimension $4$.
We will consider
functionals on the space of Riemannian metrics $\mathcal{M}$ which are
quadratic in the curvature. Such functionals have also been widely studied in
physics under the name ``fourth-order,'' ``critical,'' or ``quadratic'' gravity;
see for example \cite{LuPope, Maldacena, Schmidt, Stelle}.
In previous work, the authors have studied rigidity and stability
properties of Einstein metrics for quadratic curvature functionals
\cite{GV11}; these results will play a crucial r\^ole in this paper.

Using the standard decomposition of the curvature tensor $Rm$ into the Weyl, Ricci and scalar curvature
curvature components (denoted by $W$, $Ric$, and $R$, respectively), 
a basis for the space of quadratic curvature functionals is
\begin{align} \label{FgenN}
\mathcal{W} = \int |W|^2\ dV, \ \ \rho = \int |Ric|^2\ dV,
\ \ \mathcal{S} = \int R^2\ dV,
\end{align}
where we use the tensor norm.
In dimension four, the Chern-Gauss-Bonnet formula
\begin{align}
\label{CGB}
32\pi^2 \chi(M) = \int |W|^2\ dV - 2 \int |Ric|^2\ dV + \frac{2}{3} \int R^2\ dV
\end{align}
implies that $\rho$ can be written as a linear combination of the
other two (plus a topological term).
Consequently, we will be interested in the functional
\begin{align} \label{Ftdef1}
\mathcal{B}_{t}[g] = \int |W|^2\ dV + t \int R^2\ dV
\end{align}
(with $t = \infty$ formally corresponding to $\int R^2 dV$).

 The Euler-Lagrange equations of $\mathcal{B}_{t}$ are given by
\begin{align}
\label{Bt}
B^t \equiv B + t C = 0,
\end{align}
where $B$ is the {\em{Bach tensor}} defined by
\begin{align}
\label{Bintro}
B_{ij} \equiv -4 \Big( \nabla^k \nabla^l W_{ikjl} + \frac{1}{2} R^{kl}W_{ikjl} \Big)= 0,
\end{align}
and $C$ is the tensor defined by
\begin{align} \label{gradformS}
C_{ij} =  2 \nabla_i \nabla_j R - 2 (\Delta R) g_{ij}
- 2 R R_{ij} +   \frac{1}{2} R^2 g_{ij}.
\end{align}
It follows that any Einstein metric
is critical for~$\mathcal{B}_{t}$ \cite{Besse}.
We will refer to such a critical metric as a {\em{$B^t$-flat metric}}.
Note that by taking a trace of \eqref{Bt}, it follows that 
the scalar curvature of a $B^t$-flat metric on a compact manifold 
is necessarily constant. Therefore a $B^t$-flat metric satisfies the equation
\begin{align}
\label{bteqn}
B = 2 t R \cdot E,
\end{align}
where $E$ denotes the traceless Ricci tensor. That is, the Bach
tensor is a constant multiple of the traceless Ricci tensor.

The convergence results described above for Einstein metrics 
were generalized to systems of the form 
\begin{align}
\label{dric}
\Delta Ric = Rm * Ric
\end{align}
(of which \eqref{bteqn} is a special case) in  \cite{TV, TV2, TV3}. 
In particular, with certain geometric conditions,  
non-collapsing sequences of metrics satisfying 
an equation of the form \eqref{dric} have orbifold limits. 
Again, the natural question is whether it is possible to 
reverse this bubbling process. 

 The analogous gluing problem for the anti-self-dual equations $W^+ = 0$ in 
dimension four has been very successful  \cite{DonaldsonFriedman,
Floer, Taubes, KovalevSinger, Acheviaclovsky2}.
However, gluing for the $B^t$-flat equations is much more difficult 
because, as in the Einstein case, this is a self-adjoint problem.
The parameter $t$ is the key to overcoming this difficulty. 

 We point out that the linearization of the $B^t$-flat equation $\eqref{Bt}$ is not elliptic
due to diffeomorphism invariance. It will be necessary to
``gauge'' the equation in order to work with an elliptic
operator. This is analogous to the Bianchi gauge for the Einstein
equations. The details of this gauging process appear in Section~\ref{gauge}.

 The main building blocks in this paper are the Fubini-Study
metric $(\CP^2, g_{FS})$, and $(S^2 \times S^2, g_{S^2 \times S^2})$,
the product of $2$-dimensional spheres with unit Gauss curvature.
Both are Einstein, so are $B^t$-flat for all
$t \in \RR$. A key result used in this paper is rigidity of these metrics
for certain ranges of $t$, which was proved in our previous work~\cite{GV11}.
That is, these metrics admit no non-trivial infinitesimal
$B^t$-flat deformations for certain ranges of $t$ (other than scalings).
These rigidity properties will be discussed in Section~\ref{comkersec}.

\subsection{Green's function metric}

Recall that the conformal Laplacian is the operator
\begin{align}
L u = - 6 \Delta u + Ru,
\end{align}
where our convention is to use the analyst's Laplacian (which has negative
eigenvalues). If $(M,g)$ is compact and has positive scalar
curvature, then for any $p \in M$, there exists a unique positive solution to
the equation
\begin{align}
L G &= 0 \ \ \mathrm{ on } \ M \setminus \{p\}\\
G &= \rho^{-2}(1 + o(1))
\end{align}
as $\rho \rightarrow 0$, where $\rho$ is geodesic distance to the basepoint $p$,
which is called the Green's function.
Denote $N = M \setminus \{p\}$ with metric $g_N = G^2 g_M$.
The metric $g_N$ is scalar-flat and asymptotically flat of order $2$.
Recall the mass of an AF space is defined by
\begin{align}
\label{massdef}
{\mathrm{mass}}(g_N) = \lim_{R \rightarrow \infty} \omega_3^{-1}\int_{S(R)} \sum_{i,j}
( \partial_i g_{ij} - \partial_{j} g_{ii} ) ( \partial_i \ \lrcorner \ dV),
\end{align}
with $\omega_3 = Vol(S^3)$.

 A crucial point is the following:
if $(M,g)$ is Bach-flat, then from conformal invariance
of the Bach tensor, $(N,g_N)$ is also Bach-flat.
Also, since the Green's function is used as the conformal
factor, $g_N$ is scalar-flat. Consequently, $g_N$ is $B^t$-flat
for all $t \in \RR$.

The Green's function metric of the Fubini-Study metric $\hat{g}_{FS}$
is also known as the Burns metric, and is completely explicit,
with mass given by
\begin{align}
\mathrm{mass}(\hat{g}_{FS}) = 2.
\end{align}
However, the Green's function metric $\hat{g}_{S^2 \times S^2}$ of
the product metric does not seem to have a known explicit description.
We will denote
\begin{align}
m_1 = \mathrm{mass}(\hat{g}_{S^2 \times S^2}).
\end{align}
By the positive mass theorem of Schoen-Yau, $m_1 > 0$ \cite{SYI, SYII}.
We note that since $S^2 \times S^2$ is spin, this also follows from Witten's
proof of the positive mass theorem~\cite{Witten}.

\subsection{The gluing procedure}

Let $(Z, g_Z)$ and $(Y, g_Y)$ be Einstein manifolds, and assume that
$g_Y$ has positive scalar curvature. Choose basepoints $z_0 \in Z$ and
$y_0 \in Y$. Convert $(Y, g_Y)$
into an asymptotically flat (AF) metric $(N, g_N)$ using the Green's function for
the conformal Laplacian based at $y_0$.
As pointed out above,  $g_N$ is $B^t$-flat for any $t$.

Let $a > 0$ be small, and consider $Z \setminus B(z_0,a)$. Scale the
compact metric to $(Z, \tilde{g} = a^{-4} g_Z)$.
Attach this metric to the metric $ ( N \setminus B(a^{-1}), g_N)$  using
cutoff functions near the boundary, to obtain a smooth metric
on the connect sum $Z \# \overline{Y}$.  Since both $g_Z$ and $g_N$ are
$B^t$-flat, this metric is an ``approximate'' $B^t$-flat metric,
with vanishing $B^t$ tensor away from the ``damage zone'', where cutoff
functions were used. This construction is described in detail in Section \ref{approx},
and is illustrated in Figure~\ref{bubblefig}.

\begin{figure}[h]
\includegraphics{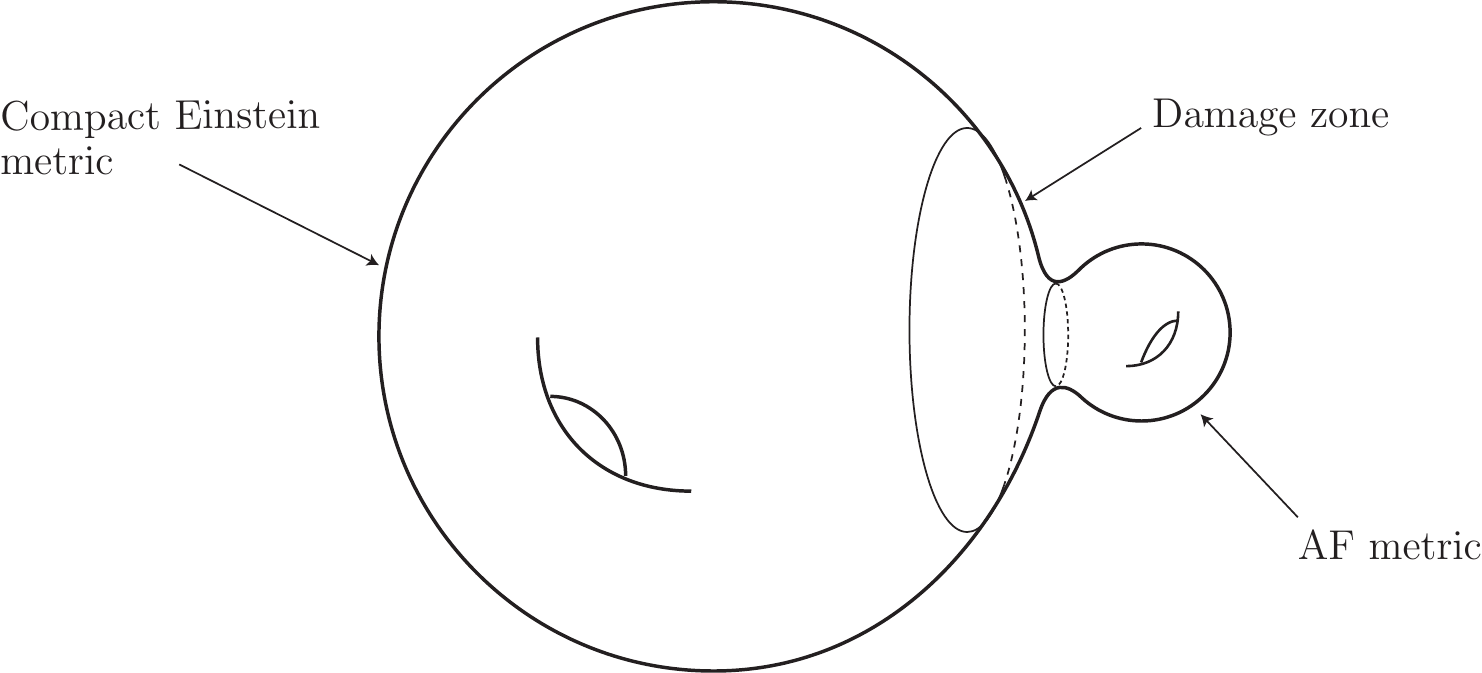}
\caption{The approximate metric.}
\label{bubblefig}
\end{figure}

 This ``na\"ive'' approximate metric is too rough for our purposes -- the
size of the $B^t$ tensor is an order too large in the damage zone.
A refinement of this approximate metric is found by solving linear
equations on each piece to make the metrics match up to highest
order. The $B^t$ tensor of the refined metric is now an order
of magnitude smaller. This step is inspired by the recent work
of Biquard in the Einstein case which was mentioned above \cite{Biquard2011}.
These auxiliary linear equations are solved in Section~\ref{auxiliary},
and the refined approximate metric is constructed in Section~\ref{better}.

Lyapunov-Schmidt reduction is then used to reduce the
problem from an infinite-dimensional problem to a finite-dimensional one.
That is, the problem of
finding a $B^t$-flat metric is reduced
to finding a zero of the Kuranishi map, which is a mapping
between finite-dimensional spaces. This reduction is carried out in Section \ref{Lyap}.

   For the general gluing problem, even if the pieces are rigid,
there can be nonzero infinitesimal kernel elements due to the presence of gluing
parameters. In general, there are infinitesimal kernel elements corresponding
geometrically to freedom of scaling the AF space, rotating the gluing
factor, and moving the base points of the gluing.
The leading term
of the Kuranishi map corresponding to the scaling
parameter, denoted by $\lambda_1(a)$, is given by:
\begin{theorem}
\label{SadjThm} As $a \rightarrow 0$, then for any $\epsilon > 0$,  
\begin{align}
\label{Sadjintro}
\begin{split}
\lambda_1(a)  &= \Big( \frac{2}{3} W(z_0) \circledast W(y_0)
+ 4t R(z_0) \mathrm{mass}(g_N)  \Big) \omega_3 a^4 + O (a^{6 - \epsilon}),
\end{split}
\end{align}
where $\omega_3 = Vol(S^3)$, and the product of the Weyl tensors is given by
\begin{align}
W(z_0) \circledast W(y_0) = \sum_{ijkl} W_{ijkl}(z_0) ( W_{ijkl}(y_0) + W_{ilkj}(y_0)),
\end{align}
where $W_{ijkl}(\cdot)$ denotes the components of the
Weyl tensor in a normal coordinate system at the corresponding point.
\end{theorem}
We note that the product $\circledast$  depends upon the coordinate
systems chosen,
and therefore in general depends upon a rotation parameter, and
obviously on the
base points of the gluing.

\subsection{Simply-connected examples}
In the case either of the factors are $(\CP^2, g_{FS})$ or
$(S^2 \times S^2, g_{S^2 \times S^2})$, Theorem \ref{SadjThm}
implies an existence theorem.
Since these manifolds are toric,  we can use the torus
action plus a certain discrete symmetry, called a diagonal symmetry,
to eliminate all gluing parameters except for the scaling parameter.
Theorem~\ref{SadjThm} will then allow us to
obtain critical metrics on the following manifolds
``near'' the indicated approximate metric:
\begin{itemize}
\item (i)
$\CP^2 \# \overline{\CP}^2$; the Fubini-Study metric with a Burns metric attached
at one fixed point. This case admits a $U(2)$-action.
\item(ii)
$S^2 \times S^2 \# \overline{\CP}^2 = \CP^2 \# 2  \overline{\CP}^2$;
the product metric on $S^2 \times S^2$ with a Burns metric attached at
one fixed point. Alternatively, we can view this as the Fubini-Study
metric on $\CP^2$, with a Green's function $S^2 \times S^2$ metric attached at
one fixed point. For this topology, we will therefore construct two
different critical metrics.
\item(iii)
$2 \# S^2 \times S^2$; the product metric on $S^2 \times S^2$ with
a Green's function $S^2 \times S^2$ metric attached at one fixed point.
\end{itemize}
More precisely, we have
\begin{theorem}
\label{thm2}
In each of the above cases, a $B^t$-flat metric exists for some $t$ near the
critical value of
\begin{align}
\label{t0}
t_0 = \frac{-1}{6  R(z_0) \mathrm{mass}(g_N)} W(z_0) \circledast W(y_0).
\end{align}
Furthermore, this metric is invariant under the indicated action(s).
\end{theorem}
The proof of the theorem appears in Section \ref{Kuranishi},
and the special values of $t_0$ in each case are indicated in Table \ref{table}.
\begin{table}[ht]
\caption{Simply-connected examples with one bubble}
\centering
\begin{tabular}{ll}
\hline\hline
Topology of connected sum& Value(s) of $t_0$ \\
\hline
$\CP^2 \# \overline{\CP}^2$ & $-1/3$ \\
$ S^2 \times S^2 \# \overline{\CP}^2 = \CP^2 \# 2  \overline{\CP}^2$ & $-1/3$, $- (9 m_1)^{-1}$\\
$ 2 \# S^2 \times S^2$ & $-2(9 m_1)^{-1}$ \\
\hline
\end{tabular}\label{table}
\end{table}

With $\CP^2$ as a compact factor, there are three fixed points of
the torus action, and with $S^2 \times S^2$, there are four
fixed points.
Employing various discrete symmetries will also allow us
to obtain critical metrics on connected sums with
more than two factors. Theorem \ref{thm2} extends to the following
cases:
\begin{itemize}
\item(iv)
$3 \# S^2 \times S^2$; the product metric on $S^2 \times S^2$ with
Green's function $S^2 \times S^2$ metrics attached at two fixed points.
In this case, we will impose an additional symmetry called bilateral symmetry.
\item (v) $S^2 \times S^2 \# 2 \overline{\CP}^2
= \CP^2 \# 3  \overline{\CP}^2$; the product metric on $S^2 \times S^2$ with Burns
metrics attached at two fixed points, with bilateral symmetry.
\item (vi) $ \CP^2 \# 3  \overline{\CP}^2$; the Fubini-Study
metric with Burns metrics attached at all fixed points,
with a symmetry called trilateral symmetry.
\item (vii) $ \CP^2 \# 3(S^2 \times S^2) = 4 \CP^2 \# 3 \overline{\CP}^2 $;
the Fubini-Study metric with Green's function $S^2 \times S^2$ metrics
attached at all fixed points, with trilateral symmetry.
\item (viii) $S^2 \times S^2 \# 4 \overline{\CP}^2 =
\CP^2 \# 5 \overline{\CP}^2$;
the product metric on $S^2 \times S^2$ with Burns metrics
attached at all fixed points, with a symmetry called quadrilateral
symmetry.
\item (ix) $5 \# S^2 \times S^2$ viewed as the product metric on
$S^2 \times S^2$ with Greens function $S^2 \times S^2$ metrics
attached at all fixed points, with quadrilateral symmetry.
\end{itemize}
The special values of $t_0$ in each case are indicated in Table \ref{tablesb}.
\begin{table}[ht]
\caption{Simply-connected examples with several bubbles}
\centering
\begin{tabular}{lll}
\hline\hline
Topology of connected sum& Value of $t_0$ & Symmetry\\
\hline
$ 3 \# S^2 \times S^2$ &  $-2(9 m_1)^{-1}$ & bilateral\\
$ S^2 \times S^2 \# 2\overline{\CP}^2 = \CP^2 \# 3  \overline{\CP}^2
 $ & $-1/3$ & bilateral\\
$ \CP^2 \# 3  \overline{\CP}^2$ & $-1/3$  & trilateral \\
$ \CP^2 \# 3(S^2 \times S^2) = 4 \CP^2 \# 3 \overline{\CP}^2 $ &  $- (9 m_1)^{-1}$  & trilateral \\
 $S^2 \times S^2 \# 4 \overline{\CP}^2 =
\CP^2 \# 5 \overline{\CP}^2$ & $-1/3$ & quadrilateral \\
$5 \# S^2 \times S^2$ &  $-2(9 m_1)^{-1}$ & quadrilateral\\
\hline
\end{tabular}\label{tablesb}
\end{table}
\begin{remark}{\em
Since $S^2 \times S^2$ admits an
orientation-reversing diffeomorphism, there is only one
possibility for a connect sum with $S^2 \times S^2$, which is
why $\overline{S^2 \times S^2}$ does not appear in the list of examples.
}
\end{remark}

\subsection{Non-simply-connected examples}
The product metric on $S^2 \times S^2$ admits the Einstein
quotient $S^2 \times S^2 / \ZZ_2$,
where $\ZZ_2$ acts by the antipodal map on both factors,
and  the quotient $\RP^2 \times \RP^2$.
Using one of these metrics as the compact factor or
the Green's function metric of one of these as one of the AF spaces,
we can obtain several non-simply-connected examples. We will denote
\begin{align}
m_2 =   \mathrm{mass}(\hat{g}_{S^2 \times S^2/\ZZ^2}),
\end{align}
and
\begin{align}
m_3 =  \mathrm{mass}(\hat{g}_{\RP^2 \times \RP^2}).
\end{align}
Again, by the positive mass theorem, $m_2 > 0$ and $m_3 > 0$.
Theorem \ref{thm2} holds for these examples as well,
and the special values of $t_0$ in each non-simply-connected
case with one bubble are indicated in Table~\ref{table2}.
We note those without an $\RP^2 \times \RP^2$ factor are
orientable, and those with an $\RP^2 \times \RP^2$ factor are
non-orientable.
Also note that the first, second, fifth and sixth examples
have finite fundamental groups. The others have infinite
fundamental group (in particular, by the Myers Theorem these
manifolds do not admit positive Einstein metrics).
\begin{table}[ht]
\caption{Non-simply-connected examples with one bubble}
\centering
\begin{tabular}{ll}
\hline\hline
Topology of connected sum& Value(s) of $t_0$ \\
\hline
$(S^2 \times S^2 / \ZZ_2) \#  \overline{\CP}^2$ & $-1/3$, $- (9 m_2)^{-1}$ \\
$(S^2 \times S^2 / \ZZ_2) \#  S^2 \times S^2$ & $- 2(9 m_1)^{-1}$,  $- 2(9 m_2)^{-1}$\\
$(S^2 \times S^2 / \ZZ_2) \#(S^2 \times S^2 / \ZZ_2)$ &   $- 2(9 m_2)^{-1}$
\\
$(S^2 \times S^2 / \ZZ_2) \# \RP^2 \times \RP^2$ &  $- 2(9 m_3)^{-1}$ , $- 2(9 m_2)^{-1}$
\\
$\RP^2 \times \RP^2 \# \overline{\CP}^2$ & $-1/3$,  $-(9 m_3)^{-1}$
\\
$\RP^2 \times \RP^2 \# S^2 \times S^2$ & $- 2(9 m_1)^{-1}$, $-2(9 m_3)^{-1}$
\\
$\RP^2 \times \RP^2 \# \RP^2 \times \RP^2$ &  $-2(9 m_3)^{-1}$
\\
\hline
\end{tabular}\label{table2}
\end{table}

As in the simply-connected case, we can take advantage of various
symmetries to obtain non-simply-connected examples with more
than one bubble. For the complete list, see Appendix \ref{appb}.

\subsection{The Bach-flat case}

We remark that Theorem \ref{SadjThm} holds in the
Bach-flat case ($t = 0$), provided one restricts to traceless
tensors throughout the argument (this is necessary due to
conformal invariance of the Bach tensor). This expansion cannot be directly used to
produce Bach-flat metrics,  since the freedom to move
the parameter~$t$ is crucial in the
proof of Theorem~\ref{thm2}.  However, the main argument
does imply the following {\em{non-existence}} result:
\begin{theorem}
\label{nonexist}
Assume that both $(Z, g_Z)$ and $(Y, g_Y)$ are Bach-flat, toric, and
admit a diagonal symmetry. Let $z_0 \in Z$ and $y_0 \in Y$ be fixed
points of the respective torus actions. If
\begin{align}
W(y_0) \circledast W(z_0) \neq 0,
\end{align}
then there is no equivariant Bach-flat metric in a $C^{4, \alpha}$-neighborhood
of the approximate metric.
\end{theorem}
This is applicable to all of the above
examples, so we may conclude that there is no Bach-flat metric near the
metrics found in Theorem~\ref{thm2}. In particular, these metrics are not
Einstein. We remark that this non-existence theorem is true without
the equivariance assumption, but a complete proof of this
adds considerable technical details, so is not included.

 Note that in the case of $\CP^2 \# \CP^2$, it is easy to see
that $W(y_0) \circledast W(z_0) = 0$, since there is an orientation-reversal
required when performing the connected sum. This is not surprising, since
it is well-known that there is a $1$-parameter family of self-dual
metrics (which are Bach-flat) near the approximate metric \cite{Poon, LeBrunJDG, ViaclovskyFourier}.

\subsection{Remarks}
The proof of Theorem \ref{thm2} shows the following dichotomy:
either (i) there is a critical metric at exactly the
critical $t_0$, in which case there would necessarily be a 1-dimensional
moduli space of solutions for this fixed $t_0$
(as pointed out above, this indeed happens for $\CP^2 \# \CP^2$,
in which case there is a $1$-parameter family of self-dual metrics).
The other possibility (ii) is that for each value of the gluing parameter
$a$ sufficiently small, there will be a critical metric for a
corresponding value of $t_0 = t_0(a)$. The dependence of $t_0$
on $a$ will depend on the next term in the expansion of \eqref{Sadjintro}.
For example, if this expansion were improved to
\begin{align}
\label{l1formfin2}
\lambda_1 =  \lambda a^4 + \mu a^8 + O(a^{12 - \epsilon}),
\end{align}
with $\mu \neq 0$, then we would have the dependence
\begin{align}
t_0 = \frac{1}{  4 R(z_0) \mathrm{mass}(g_N)  }
\Big( -  \frac{2}{3} W(y_0) \circledast W(z_0)  - \frac{\mu}{\omega_3} a^4 \Big)
+ O(a^{8 - \epsilon}).
\end{align}
as $a \rightarrow 0$.

It should be possible to extend the methods in this paper to compute $\mu$.
If it turns out that $\mu \neq 0$,
then one may conclude that possibility (ii) definitely happens.
The sign of $\mu$ would then determine if solutions are found for
$t > t_0$ or $t < t_0$.  If $\mu = 0$, this would indicate (but not prove)
that possibility (i) is what actually occurs.
The methods in this paper cannot practically be used to determine
that possibility (i) actually happens, since there would be an infinite sequence
of obstructions to check in this eventuality.

We next make some remarks about some relations between K\"ahler geometry 
and the value $t_0 = -1/3$ appearing in the above tables.
Using the Hirzebruch signature theorem, we can write 
\begin{align} \label{Bredef}
\mathcal{B}_{-1/3}[g] = - 48 \pi^2 \sigma(M^4) + 2 \int \big( | W^{+} |^2 - \frac{1}{6} R^2 \big)\ dV.
\end{align}
An immediate corollary of this formula is that if $(M^4,g)$ is K\"ahler, then
\begin{align} \label{Bkahler}
\mathcal{B}_{-1/3}[g] = - 48 \pi^2 \sigma(M^4).
\end{align}
In addition, a constant scalar curvature K\"ahler metric is necessarily 
critical for the value $t_0 = -1/3$ \cite{Derdzinski}. We note that 
important gluing results for constant scalar curvature K\"ahler metrics were 
proved in \cite{API, APII}. 

For the manifolds $\CP^2 \# k \overline{\CP}^2,$ when $k = 1, 2, 3,$ or $5$,
consider the cases when a Burns metric is used for the bubbles 
(the cases when $t_0 = -1/3$). In these cases, it is known that there are extremal
K\"ahler metrics near the n\"aive approximate metric \cite{APS2011, Gabor}.
These extremal metrics do not have constant scalar curvature, so they are not
the same as the critical metrics found in Theorem~\ref{thm2}.
There might be some other relation between these metrics
(such as conformality), but we are not aware of any such relation.
These manifolds are known to admit Einstein metrics \cite{Besse, CLW, LeBrun2012}.

However, on many of the other manifolds considered in this paper, there does not 
exist {\em{any}} K\"ahler metric (for example $2 \# S^2 \times S^2$),
and the critical metrics found in Theorem \ref{thm2} are the first
known ``canonical'' metrics, to the best of the authors' knowledge.

\subsection{Acknowledgements}
This work was done in part at the Institute for Advanced Study in Princeton
and at the Institut Henri Poincar\'e in Paris.
Both authors have been partially supported by the NSF;
the first author under NSF grant DMS-1206661
and the second author under NSF grant DMS-1105187.
The second author was partially supported by the Simons Foundation
as a Simons Fellow of Mathematics. Both authors thank the above
institutions for their generous support.

The authors would like to thank Olivier Biquard for crucial assistance in
completing this work, and to Simon Brendle, Claude LeBrun, Rafe Mazzeo, Frank Pacard, 
and Karen Uhlenbeck for many helpful remarks.

\section{The building blocks}
\label{bb}

In this section, we will derive metric expansions for the
``building blocks'' of our gluing procedure; the Fubini-Study metric
and product metric on $S^2 \times S^2$. We will also give
metric expansions of the associated scalar-flat asymptotically flat metrics,
arising from the Green's function of the conformal Laplacian.

The general gluing problem has many degrees of freedom.
We will take advantage of various symmetries to reduce
eventually to only one degree of freedom. So in this section,
we will also describe the various group actions which will be
used for an equivariant gluing.

\subsection{The Fubini-Study metric}
\label{fssub}
Recall that $\CP^2$ is the set of complex projective lines through
the origin in $\CC^3$. Equivalently, $\CP^2$ is the set of equivalence
classes $\{\CC^3 \setminus \{0\} \}/ \CC^*$, where the action of $\CC^*$ is
defined by, for $\lambda \in \CC^*$,
\begin{align*}
[u_0, u_1, u_2] \mapsto [ \lambda u_0, \lambda u_1, \lambda u_2].
\end{align*}
Let $U_i = \{ [ u_0, u_1, u_2] | u_i \neq 0 \}$, for $i = 0,1,2$.
The Fubini-Study metric is given in $U_0$ by \cite{KobayashiNomizuII}
\begin{align}
\begin{split}
g_{FS} &= \frac{\I}{2} \partial \overline{\partial}( 1 + |u_1|^2 + |u_2|^2)\\
&  =\frac{  (1 + |u|^2) ( du_1 d \bar{u}_1 +  du_2 d \bar{u}_2)
- ( \bar{u}_1 du_1 + \bar{u}_2 du_2)  ( u_1 d \bar{u}_1 + u_2 d\bar{u}_2)}
{(1 + |u|^2)^2}.
\end{split}
\end{align}
This extends to an Einstein metric on $\CP^2$ with $Ric(g) = 6 g$, and
$PU(3)$, the projective unitary group
(the unitary group $U(3)$ modulo its center),
acts by isometries.

We will consider two sub-actions of this group action.
The first is an action of $U(2)$ fixing the point $[1,0,0]$.
Globally, this action is given by, for $A \in U(2)$,
\begin{align}
[u_0,u_1,u_2] \mapsto [ u_0, A(u_1, u_2)].
\end{align}
The point $[1,0,0]$ is the only fixed point of this action.
In $U_0$, this action is given by the standard action of $U(2)$
acting on $\CC^2$.

The second action is the torus action of the form
\begin{align}
[u_0, u_1, u_2] \mapsto [ u_0, e^{\I \theta_1} u_1, e^{\I \theta_2} u_2]
\end{align}
where $\theta_1,\theta_2 \in [0,2\pi]$.
This action has $3$ fixed points $[1,0,0], [0,1,0]$, and $[0,0,1]$.
In $U_0$, this action is given by
\begin{align}
\label{sda}
(u_1, u_2) \mapsto  (e^{\I \theta_1} u_1, e^{\I \theta_2} u_2).
\end{align}
Next, let  $\{\sigma_1, \sigma_2, \sigma_3\}$ be a left-invariant coframing
on $S^3$ such that $\sigma_3$ is a connection form for the
Hopf fibration $\pi : S^3 \rightarrow S^2 = \CP^1$ defined by
\begin{align}
\pi(u_1, u_2) = [u_1 : u_2],
\end{align}
and such that $\pi^* g_{S^2} = 4(\sigma_1^2 + \sigma_2^2)$.
The Fubini-Study metric can then be written as \cite[page 257]{EGH}
\begin{align}
\label{fscof}
g_{FS} = \frac{1}{(1 + r^2)^2} dr^2 + \frac{r^2}{1 + r^2} \Big(
\sigma_1^2 + \sigma_2^2 + \frac{1}{1 + r^2} \sigma_3^2 \Big)
\end{align}
From this expression, the above action of $U(2)$
is seen here here as an action of $\rm{SO}(3) \times \rm{SO}(2)$
where the first factor acts by rotations of $S^2$, and the
second factor acts by rotations of the fiber of the Hopf fibration.
The above torus action is the restricted action where the first factor acts by
a rotation of $S^2$ fixing the north and south pole.

From \eqref{fscof}, we see that $\rho = \arctan(r)$ is the geodesic distance
from the basepoint, and under this radial change of coordinates the
metric is written as \cite{LeBrunNN}
\begin{align}
\label{cofs}
g_{FS} = d\rho^2 + \sin^2(\rho) \big( \sigma_1^2 + \sigma_2^2 + \cos^2(\rho) \sigma_3^2 \big),
\end{align}
with the restriction that $0 < \rho < \pi/2$. Since the coordinate
change is radial, we note the important fact that in these
coordinates, the above action of $U(2)$ is still the standard linear action.

Finally, we let $\{z^i\}$ be Euclidean normal coordinates,
based at $[1,0,0]$, so that $U(2)$ acts linearly,
and that the above torus action acts by
\begin{align}
(z_1, z_2, z_3, z_4)
\mapsto \big( e^{  \I \theta_1} (z_1 + \I z_2), e^{ \I \theta_2}  (z_3 + \I z_4)\big).
\end{align}
In this coordinate system, we have the expansion
\begin{align}
g_{ij} = \delta_{ij} - \frac{1}{3} R_{ikjl}([1,0,0]) z^k z^l + O^{(4)}(|z|^4)_{ij}
\end{align}
as $|z| \rightarrow 0$.
\begin{remark}{\em
We adopt the convention that for a function (or tensor) $f = f(z)$, $f = O^{(m)}(|z|^{\alpha})$ means
$|\partial^{k} f| = O(|z|^{\alpha - k})$ for all $1 \leq k \leq m$
(as $z$ approaches an indicated limit). }
\end{remark}

This metric is invariant under the diagonal symmetry:
\begin{align}
(z_1, z_2, z_3, z_4) \mapsto (z_3, z_4, z_1, z_2),
\end{align}
which is contained in $U(2)$. In the case
of toric invariance, we will impose this as an extra symmetry for the
equivariant gluing problem.
In both cases, there will therefore be only one fixed point, the point $[1,0,0]$.
These symmetries are illustrated in Figure~\ref{cp2fig}.

\begin{figure}[h]
\includegraphics{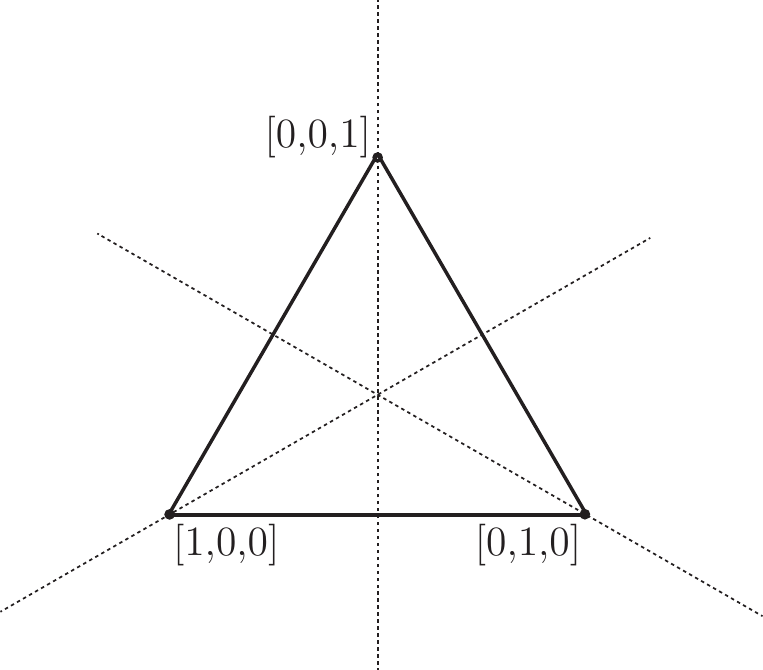}
\caption{Orbit space of the torus action on $\CP^2$. The vertices
of the triangle are fixed points, open edge points are circle orbits,
and interior points are principal orbits. The diagonal
symmetry is a reflection in the dotted diagonal line passing through
[1,0,0]. Invariance under reflection in all dotted diagonal lines will be called
trilateral symmetry (note these reflections correspond to
coordinate flips $u_i \leftrightarrow u_j$ on $\CP^2$, which are
isometries of $g_{FS}$).}
\label{cp2fig}
\end{figure}

\subsection{The Burns metric}
\label{burnssub}
We begin with a general result regarding the Green's function
expansion for a toric Einstein manifold:
\begin{proposition}
\label{greenein}
Let $G$ be the Green's function for the conformal Laplacian at the
point $p \in M$, where $(M,g)$ is an Einstein metric with
positive scalar curvature. If $(M,g)$ admits a non-trivial torus action
fixing the point $p$,
then in a Riemannian normal coordinate system $\{z^i\}$,
we have the following expansion: For any $\epsilon > 0$,
\begin{align}
\label{eing}
G = |z|^{-2} + A + O^{(4)} ( |z|^{2 - \epsilon})
\end{align}
as $|z| \rightarrow 0$, where $A$ is a constant (independent of $\epsilon$).
\end{proposition}
\begin{proof}
A straightforward computation, which we omit, shows that
there is a formal power series solution of the form with leading terms
\begin{align}
\label{formalexpp}
G = |z|^{-2} + A + \dots,
\end{align}
and $A$ is a constant.
Recall that the indicial roots of the Laplacian are
$\ZZ \setminus \{-1\}$. Solutions corresponding to
the indicial root $1$ are linear, and not
invariant under the torus action, so there is no linear term in the
expansion.  It follows from standard
techniques that the formal expansion \eqref{formalexpp}
implies the actual expansion \eqref{eing}. The proof is
identical to \cite[Lemma 6.4]{LeeParker} (using Riemannian
normal coordinates instead of conformal normal coordinates),
so the details are omitted.
\end{proof}
In the case of the Fubini-Study metric, we have the following improved expansion:
\begin{proposition}
Let $G$ be the Green's function for the conformal Laplacian of the
Fubini-Study metric based at $[1,0,0]$, normalized so that $Ric(g) = 6 g$.
Then in the above normal
coordinate system $\{z^i\}$ we have the expansion
\begin{align}
G = |z|^{-2} + \frac{1}{3}+ O^{(4)} ( |z|^2)
\end{align}
as $|z| \rightarrow 0$.
\end{proposition}
\begin{proof}
Since the metric is invariant under $U(2)$, from uniqueness of the
Green's function, $G$ must be radial. Using that $R = 24$, the equation is
\begin{align}
\label{gfe}
\Delta G = 4 G.
\end{align}
We let $\rho = |z|$ denote the radial distance function.
For a radial function, \eqref{gfe} reduces to the ODE
\begin{align}
G_{\rho \rho} + (3 \cot(\rho) - \tan(\rho)) G_{\rho} = 4 G
\end{align}
on the interval $[0,\pi/2]$.
This ODE has the general solution
\begin{align}
G = \frac{C_1}{ \sin^2( \rho)} + C_2 \frac{ \log ( \cos(\rho))}{\sin^2(\rho)}
\end{align}
for constants $C_1$ and $C_2$.
The boundary condition $G = \rho^{-2} (1 + o(1))$ as $\rho \rightarrow 0$
implies that $C_1 = 1$.  For the other boundary condition, in order
to give a smooth global solution, we require that $G_{\rho}(\pi/2) = 0$,
which implies that $C_2 = 0$.
The claimed expansion follows easily from
\begin{align}
(\sin \rho)^{-2} = \rho^{-2} + \frac{1}{3} + \frac{1}{15} \rho^2 +  O^{(4)} ( \rho^{4})
\end{align}
as $\rho \rightarrow 0$.
\end{proof}
Since $g_{FS}$ is Bach-flat (it is self-dual with respect to the
complex orientation), the metric $g_N = G^{2} g_{FS}$ is also
Bach-flat, and scalar-flat. Consequently, $g_N$ is $B^t$-flat
for any $t \in \RR$.
Let $\{ x^i = z^i/|z|^2 \}$ denote
inverted normal coordinates near $[1,0,0]$,and let
\begin{align}
\mathcal{I}(x) = \frac{x}{|x|^2} = z
\end{align}
denote the inversion map. With respect to these
coordinates, we can write the metric $g_N$ in the
complement of a large ball as
\begin{align}
\begin{split}
g_N &= \mathcal{I}^* ( G^2 g_{FS}) \\
& =  (G \circ \mathcal{I})^2 \mathcal{I}^* \Big( \{  \delta_{ij} - \frac{1}{3} R_{ikjl}([1,0,0]) z^k z^l + O^{(4)}(|z|^4)_{ij} \} dz^i dz^j \Big)\\
& = ( |x|^2 + A + O^{(4)}(|x|^{-2}))^2
\big\{  \delta_{ij} - \frac{1}{3} R_{ikjl}([1,0,0]) \frac{x^k x^l}{|x|^4} + O^{(4)}(|x|^{-4})_{ij}\big\}\\
&\cdot \frac{1}{|x|^2} \Big( \delta_{ip} - \frac{2}{|x|^2} x^i x^p \Big) dx^p
\cdot \frac{1}{|x|^2} \Big( \delta_{jq} - \frac{2}{|x|^2} x^j x^q \Big) dx^q,
\end{split}
\end{align}
so we have the expansion
\begin{align}
\begin{split}
\label{burnsexp}
(g_N)_{ij}(x) &= \delta_{ij} - \frac{1}{3} R_{ikjl}([1,0,0]) \frac{x^k x^l}{|x|^4} + 2 A
\frac{1}{|x|^2} \delta_{ij} + O^{(4)}( |x|^{-4})
\end{split}
\end{align}
as $|x| \rightarrow \infty$. Clearly, $g_N$ is asymptotically flat (AF) of order $\gamma = 2$.

Note that this metric is also invariant under the standard linear
action of $U(2)$, now acting in the $\{x\}$-coordinates.

\begin{remark}{\em
As the title of the subsection indicates,
this metric is also known as the Burns metric; it is a K\"ahler scalar-flat
metric on the blow-up of $\CC^2$ at the origin.
By the coordinate change $r = \sin^{-1}(\rho)$, 
and multiplying by $r^4$, one obtains
\begin{align}
\label{glb}
g_{N} = \frac{ dr^2}{ 1 - r^{-2}} +r^2 \Big[ \sigma_1^2 + \sigma_2^2
+ (  1 - r^{-2}) \sigma_3^2 \Big],
\end{align}
which is the expression of the Burns metric obtained in \cite{LeBrunnegative}.
We could instead use this coordinate system for the Burns metric in this paper.
However, since there is not an analogue of this for the
next example, we will remain with the above inverted Riemannian normal coordinates,
 in order to give a unified approach.
}
\end{remark}

We note here the following, which relates the constant $A$ to the
mass of the Green's function metric, and will be used later.
\begin{proposition}
\label{massprop}
Let $(M,g)$ be as in Proposition \ref{greenein}.
Then the mass of the AF metric $g_N = G^2 g$ on $N = M \setminus \{p\}$
is given by
\begin{align}
\mathrm{mass}(g_N) = 12 A - \frac{R(p)}{12}.
\end{align}
\end{proposition}
\begin{proof} This follows from \eqref{massdef} using inverted normal
coordinates; the routine calculation is omitted.
\end{proof}
For the Fubini-Study metric, since $R(p) = 24$, and $A = 1/3$,
this implies that
\begin{align}
\mathrm{mass}({g}_{N}) = 2.
\end{align}
\subsection{The product metric on $S^2 \times S^2$}
\label{s2s2sub}
Next, we consider $S^2 \times S^2$ with metric
$g = g_{S^2} \times  g_{S^2}$ the product of metrics of
constant Gaussian curvature $1$. The torus action we
will consider is just the product of counter-clockwise
$S^1$-rotations fixing the north and south poles. This action has
$4$ fixed points $(n, n), (n,s), (s, n)$, and
$(s, s)$, where $n$ and $s$ are the north and south
poles, respectively.

Taking normal coordinates on each factor
around $(n,n) \in S^2 \times S^2$, yields a normal
coordinate system $(r_1, \theta_1, r_2, \theta_2)$ so that
\begin{align}
g_{S^2 \times S^2} = dr_1^2 + \sin^2(r_1) d\theta_1^2 + dr_2^2 + \sin^2(r_2) d\theta_2^2,
\end{align}
and the radial distance function is given by $\rho = \sqrt{r_1^2 + r_2^2}$.
Finally, we let $\{z^i\}$ be
Euclidean normal coordinates based at $(n,n)$,
so that the above torus action acts by
\begin{align}
(z_1, z_2, z_3, z_4)
\mapsto \big( e^{ \I \theta_1} (z_1 + \I z_2), e^{ \I \theta_2}  (z_3 + \I z_4)\big).
\end{align}
In this coordinate system, we have the expansion
\begin{align}
\label{s2normal}
g_{ij} = \delta_{ij} - \frac{1}{3} R_{ikjl}((n, n)) z^k z^l + O(|z|^4)_{ij},
\end{align}
as $|z| \rightarrow 0$.

In addition to toric invariance, this metric is also invariant under the diagonal
symmetry:
\begin{align}
(z_1, z_2, z_3, z_4) \mapsto (z_3, z_4, z_1, z_2).
\end{align}
We will also impose this as an
extra symmetry for the equivariant gluing problem.
These symmetries, as well as some other symmetries we will
use later, are illustrated in Figure~\ref{s2s2fig}.

\begin{figure}[h]
\includegraphics{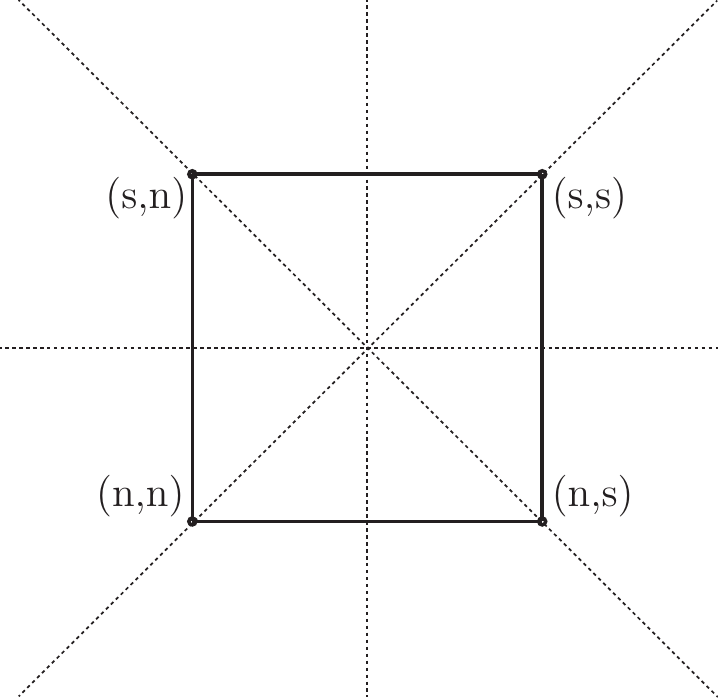}
\caption{Orbit space of the torus action on $S^2 \times S^2$. The vertices
of the square are fixed points, open edge points are circle orbits,
and interior points are principal orbits. The diagonal
symmetry is a reflection in the dotted diagonal line passing
through $(n,n)$. The bilateral symmetry is reflection in the dotted anti-diagonal line.
Reflection in the dotted vertical line is the antipodal map
of the first factor, while reflection in the dotted horizontal line
is the antipodal map of the second factor.
Invariance under all of these reflections
will be called quadrilateral symmetry.}
\label{s2s2fig}
\end{figure}

As mentioned in the introduction, the product metric on $S^2 \times S^2$ admits the Einstein quotient $S^2 \times S^2 / \ZZ_2$,
where $\ZZ_2$ acts by the antipodal map on both factors,
and the quotient $\RP^2 \times \RP^2$.
These quotients are also toric and the
same expansion \eqref{s2normal} holds for these.
The diagonal symmetry also descends to a symmetry of these metrics.

\subsection{Green's function of product metric}
\label{greensub}
Let $G$ be the Green's function for the conformal Laplacian of the
product metric at the point $(n,n)$, normalized so that $R = 4$.
By Proposition \ref{greenein}, in the above normal
coordinate system $\{z^i\}$, for any $\epsilon > 0$, we have the expansion
\begin{align}
\label{s2g}
G = |z|^{-2} + A + O^{(4)} ( |z|^{2 - \epsilon})
\end{align}
as $|z| \rightarrow 0$, where $A$ is a constant (independent of $\epsilon$).

Since $g_{S^2 \times S^2}$ is Bach-flat (it is Einstein),
the metric ${g}_{N} = G^{2} g_{S^2 \times S^2}$ is also
Bach-flat, and scalar-flat. Consequently, ${g}_{N}$ is $B^t$-flat
for any $t \in \RR$.
Letting $\{ x^i = z^i/|z|^2 \}$ denote
inverted normal coordinates, analogous to \eqref{burnsexp},
the metric $g_N$ admits the expansion
\begin{align}
\begin{split}
\label{s2s2exp}
(g_N)_{ij}(x) &= \delta_{ij} - \frac{1}{3} R_{ikjl}((n,n)) \frac{x^k x^l}{|x|^4} + 2 A
\frac{1}{|x|^2} \delta_{ij} + O^{(4)}( |x|^{-4 + \epsilon})
\end{split}
\end{align}
as $|x| \rightarrow \infty$, for any $\epsilon > 0$. Clearly, $g_N$ is AF of order $\gamma = 2$.

This metric is invariant under the above diagonal torus
action, now acting in the $\{x\}$-coordinates, and is also invariant
under the diagonal symmetry
\begin{align}
(x_1, x_2, x_3, x_4) \mapsto (  x_3, x_4, x_1, x_2).
\end{align}
\begin{remark}{\em
Unlike the case of the Burns metric, there is no explicit
description of this metric known (to the best of the authors' knowledge).
Since the metric is invariant under the above torus action,
from uniqueness of the Green's function, $G = G(r_1, r_2)$.
Using that $R = 4$, the equation is
\begin{align}
\Delta G = \frac{2}{3} G.
\end{align}
Since $G = G(r_1, r_2)$, a computation shows that this reduces to the PDE
\begin{align}
G_{r_1 r_1} + \cot(r_1) G_{r_1} + G_{r_2 r_2} + \cot(r_2) G_{r_2}  = \frac{2}{3} G,
\end{align}
on the square $[0,\pi] \times [0, \pi]$. Unlike the case of the Fubini-Study
metric, this does not appear to admit any explicit solution.
}
\end{remark}

\section{The nonlinear map}
\label{gauge}

Let $(M,g)$ be a compact manifold of dimension $4$, and
Let $S^2(T^*M)$ denote the bundle of symmetric
$2$-tensors on $M$.
We recall some important linear operators.
For simplicity of notation, we will treat
the domain and range of an operator as if it were the bundle itself, although
the operator really acts on sections of the bundle.
Let $\delta_g : S^2(T^{*}M) \rightarrow T^{*}M$ denote the divergence operator
\begin{align} \label{deldef}
(\delta_g h)_j = \nabla^i h_{ij},
\end{align}
and $\delta^{*} : T^{*}M  \rightarrow S^2(T^{*}M)$ its $L^2$-adjoint.  Note that
\begin{align} \label{deldual}
\delta^{*} = -\frac{1}{2} \mathcal{L},
\end{align}
where $\mathcal{L}$ is the Killing operator:
\begin{align} \label{Ldef}
(\mathcal{L}_g \omega)_{ij} = \nabla_i \omega_j + \nabla_j \omega_i.
\end{align}
We let $\mathcal{K}_g$ denote the conformal Killing operator, the trace-free part of $\mathcal{L}_g$:
\begin{align} \label{Kef}
(\mathcal{K}_g \omega)_{ij} = \nabla_i \omega_j + \nabla_j \omega_i - \frac{1}{2}(\delta_g \omega)g_{ij}.
\end{align}

Next, for a fixed background metric $g$,
define the nonlinear map $P_g$
\begin{align}
P_g^t : C^{4, \alpha}(S^2(T^*M)) \rightarrow C^{0, \alpha}(S^2(T^*M))
\end{align}
by
\begin{align}
\label{tmap}
P_g^t(\theta) = B^t(g + \theta) + \mathcal{K}_{g + \theta} \delta_g \mathcal{K}_g \delta_g \overset{\circ}{\theta},
\end{align}
where
\begin{align}
\overset{\circ}{\theta} = \theta - \frac{1}{4} tr_g( \theta) g.
\end{align}
\begin{remark}{\em
The domain of $P_g^t$ is not actually the entire space; it is the subset of
$C^{4,\alpha}$ so that $g + \theta$ is a Riemannian metric.
The fact that the image lies in $C^{0,\alpha}$ is a
consequence of $P_g^t$ being analytic as a function of $\theta$ and
its derivatives up to order four.
}
\end{remark}
We let $S_g^t \equiv (P_g^t)'(0)$ denote the linearized operator at $\theta = 0$.
\begin{remark}{\em
When the base metric is clear from the context, we will often omit the subscript in 
the operator $P^t$ and its linearization $S^t$. To further simplify 
notation, we will also often omit the superscript $t$ from both of these 
operators, since it is clear that they depend on $t$.
}
\end{remark}

\begin{proposition}
\label{ellprop}
If $t \neq 0$, then $S^t$ is elliptic.
\end{proposition}

\begin{proof} This is proved in \cite[Theorem 2.7 (i)]{GV11}, although we provide a brief sketch since some of the
formulas will be needed in subsequent sections.
We also note a difference in notation with our previous paper \cite{GV11}.
In that paper we considered the functional
\begin{align}
\mathcal{F}_{\tau} = \int_M |Ric|^2 dV + \tau \int_M R^2 dV.
\end{align}
From \eqref{CGB}, we obtain the relation
\begin{align}
\mathcal{F}_{\tau} = 16 \pi^2 \chi(M) + \frac{1}{2} \mathcal{B}_{2(\tau + \frac{1}{3})}.
\end{align}
Taking gradients, we obtain the relation
\begin{align}
\label{gradrel}
\nabla \mathcal{B}_{t} = 2 \nabla \mathcal{F}_{\frac{t}{2} - \frac{1}{3}}.
\end{align}
It follows from the formula for $P$ that the linearized operator is given by
\begin{align} \label{linop}
S^t h = (B' + t C')h
+ \mathcal{K}_{g} \delta_g \mathcal{K}_g \delta_g \overset{\circ}{h} ,
\end{align}
where $B'$ and $C'$ are the linearizations of $B$ and $C$ respectively.
Using \eqref{gradrel}, from \cite[Equation (2.54)]{GV11}
the leading terms of $B' + tC'$ are
\begin{align} \label{linop1}
\begin{split}
(B' + t C')h_{ij} &= \Delta^2 h_{ij} - \Delta \big[ \nabla_i \delta_j h + \nabla_j \delta_i h \big]  - (2t - \frac{1}{3}) \nabla_i \nabla_j (\Delta tr\ h) \\
 & \hskip.25in + (2 t + \frac{2}{3}) \nabla_i \nabla_j (\delta^2 h) +
(2t - \frac{1}{3}) \big[\Delta^2 (tr\ h) -\Delta (\delta^2 h)\big] g_{ij} + \cdots.
\end{split}
\end{align}
Also, a simple calculation gives
\begin{align*}
(\mathcal{K}_{g} \delta_g \mathcal{K}_g \delta_g \overset{\circ}{h})_{ij}
&=  \Delta \big[ \nabla_i \delta_j h + \nabla_j \delta_i h \big]
 - \frac{3}{4}\nabla_i \nabla_j (\Delta tr\ h)
+ \nabla_i \nabla_j (\delta^2 h) \\
& \hskip.5in  + \frac{3}{16} \Delta^2 (tr\ h) g_{ij} - \frac{3}{4}\Delta( \delta^2 h) g_{ij} + \cdots.
\end{align*}
Consequently,
\begin{align} \label{Sform1} \begin{split}Sh  &= \Delta^2 h - 2\big(t + \frac{5}{24}\big) \nabla^2 (\Delta tr\ h) - 2 \big( t + \frac{5}{24}\big) \Delta (\delta^2 h) g \\
& \hskip.2in + 2\big( t + \frac{5}{6}\big) \nabla^2 (\delta^2h) + 2\big(t - \frac{7}{96}\big) \Delta^2 (tr\ h)g + \cdots.
\end{split}
\end{align}
It follows from \eqref{Sform1} that the symbol of $S$ is
\begin{align} \label{Ssymb} \begin{split}
(\sigma_{\xi}S)h_{ij}  &= |\xi|^4 h - 2\big(t + \frac{5}{24}\big) \xi_i \xi_j |\xi|^2 ( tr\ h) - 2 \big( t + \frac{5}{24}\big) |\xi|^2 h_{k \ell}\xi_k \xi_{\ell} g_{ij} \\
& \hskip.2in + 2\big( t + \frac{5}{6}\big) \xi_i \xi_j h_{k \ell} \xi_k \xi_{\ell}
+ 2\big(t - \frac{7}{96}\big) |\xi|^4 (tr\ h)g_{ij},
\end{split}
\end{align}
which is elliptic for $t \neq 0$, according to \cite[Theorem 2.7 (i)]{GV11}.
\end{proof}
\begin{remark}{\em
For purposes below, it will be useful to rewrite \eqref{Sform1} as
\begin{align} \label{Sform}
Sh = \Delta^2 [h - \frac{1}{4}(tr\ h)g] + \mathcal{K}\big[ d (\mathcal{D}_2(h)) \big] + \frac{3}{2}t \Big[ \Delta^2 (tr\ h) - \Delta (\delta^2 h) \Big]g + \cdots
\end{align}
where $\mathcal{D}_2 : S^2T^{*}M \rightarrow C^{\infty}$ is a second-order operator given by
\begin{align} \label{D2def}
\mathcal{D}_2(h) = \big( t + \frac{5}{6}\big) \delta^2 h - \big( t + \frac{5}{24}\big) \Delta (tr\ h).
\end{align}
}
\end{remark}
The following proposition shows that the zeroes of $P$ are in fact
$B^t$-flat metrics:
\begin{proposition}
\label{smoothprop}
Assume $t \neq 0$.  If $P ( \theta) =0 $ and $\theta \in C^{4,\alpha}$ for some $0 < \alpha < 1$,
then $B^t( g + \theta) = 0$ and $\theta \in C^{\infty}$.
\end{proposition}
\begin{proof}
The equation is
\begin{align}
\label{Peqn}
B^t(g + \theta) + \mathcal{K}_{g + \theta} \delta_g \mathcal{K}_g \delta_g \overset{\circ}{\theta} =0.
\end{align}
We claim that both terms on the left hand side of \eqref{Peqn} vanish.
The proof involves an integration by parts argument, but this presents a difficulty since $\theta \in C^{4, \alpha}$ only implies that
$P_g(\theta)$ is $C^{\alpha}$, and not necessarily differentiable.  To get around this problem we mollify $\theta$; i.e., let $\{ \theta_{\epsilon}\}$ be a family
of smooth tensor fields such that $\theta_{\epsilon} \rightarrow \theta$ in $C^{4,\alpha}$ as $\epsilon \rightarrow 0$, and let $\tilde{g}_{\epsilon} = g + \theta_{\epsilon}$.  From (\ref{Peqn}) and the  continuity of $P$ it follows that
\begin{align} \label{zedPe}
\eta_{\epsilon} =  B^t(g + \theta_{\epsilon}) + \mathcal{K}_{g + \theta_{\epsilon}}[ \Box_{\mathcal{K}_g} \beta_g \theta_{\epsilon}],
\end{align}
where $\eta_{\epsilon} \rightarrow 0$ in $C^{\alpha}$.  Pair both sides of (\ref{zedPe}) with $\mathcal{L}_{\tilde{g_{\epsilon}}}[\Box_{\mathcal{K}_g} \beta_g \theta_{\epsilon}] $ (with respect to the $L^2$-inner product defined by $\tilde{g_{\epsilon}}$), where $\mathcal{L}$ is the
Killing operator defined in (\ref{Ldef}):
\begin{align*}
\langle \mathcal{L}_{\tilde{g}_{\epsilon}}[\Box_{\mathcal{K}_g} \beta_g \theta_{\epsilon}] , \eta_{\epsilon}\rangle_{L^2} &= \big \langle \mathcal{L}_{\tilde{g}_{\epsilon}}[\Box_{\mathcal{K}_g} \beta_g \theta_{\epsilon}] , B^t(\tilde{g}_{\epsilon}) +  \mathcal{K}_{\tilde{g}_{\epsilon}}[ \Box_{\mathcal{K}_g} \beta_g \theta_{\epsilon}]  \big \rangle_{L^2} \\
&= \big \langle \mathcal{L}_{\tilde{g}_{\epsilon}}[\Box_{\mathcal{K}_g} \beta_g \theta_{\epsilon}] ,
B^t (\tilde{g}_{\epsilon}) \big \rangle_{L^2} +  \| \mathcal{K}_{\tilde{g}_{\epsilon}} [\Box_{\mathcal{K}_g} \beta_g \theta_{\epsilon}] \|_{L^2}^2.
\end{align*}
Integrating by parts in the first term on the right-hand side, we get
\begin{align*}
\big \langle \mathcal{L}_{\tilde{g}_{\epsilon}}[\Box_{\mathcal{K}_g} \beta_g \theta_{\epsilon}] , B^t (\tilde{g}_{\epsilon}) \big \rangle_{L^2} &= - 2 \big \langle \Box_{\mathcal{K}_g} \beta_g \theta_{\epsilon} , \delta_{\tilde{g}_{\epsilon}} \big( B^t (\tilde{g}_{\epsilon}) \big) \big \rangle_{L^2} = 0,
\end{align*}
since $\mathcal{L}^{*} = -2 \delta$ and the gradient of a Riemannian functional is always divergence-free (see \cite{Besse}, Proposition 4.11). Therefore,
\begin{align*}
\langle \mathcal{L}_{\tilde{g}_{\epsilon}}[\Box_{\mathcal{K}_g} \beta_g \theta_{\epsilon}] , \eta_{\epsilon}\rangle_{L^2} = \frac{1}{2} \| \mathcal{K}_{\tilde{g}_{\epsilon}} [\Box_{\mathcal{K}_g} \beta_g \theta_{\epsilon}] \|_{L^2}^2.
\end{align*}
Letting $\epsilon \rightarrow 0$, the left-hand side converges to zero, while the right-hand side converges to $\mathcal{K}_{\tilde{g}}[\Box_{\mathcal{K}_g} \beta_g \theta ]$, which
consequently vanishes.  We conclude that
\begin{align}
\label{fteqn3}
B^t (\tilde{g}) = 0
\end{align}
as claimed.

 Next, taking a trace of \eqref{fteqn3}, yields
\begin{align}
\label{reqn}
-6 t\Delta R_{\tilde{g}}  =0,
\end{align}
which implies that the scalar curvature of $\tilde{g}$ is constant.
The equation \eqref{fteqn3} then implies that $\Delta_{\tilde{g}} Ric_{\tilde{g}} \in
C^{2, \alpha}$ (more precisely, around any point $p \in M$, there exists a
coordinate system $\{x^i\}$ such that the components are in $C^{2,\alpha}$),
which implies that $Ric_{\tilde{g}} \in C^{4,\alpha}$.
Since $\tilde{g} \in C^{4, \alpha}$, there exists a harmonic coordinate
system $\{y^i\}$ around $p$ such that the equation
\begin{align}
\frac{1}{2}\tilde{g}^{ij}\partial^{2}_{ij} \tilde{g}_{kl}+Q_{kl}(\partial \tilde{g}, \tilde{g})
=-Ric_{kl}(\tilde{g}) \label{eqsrt12}
\end{align}
holds, where $Q(\partial \tilde{g},\tilde{g})$ is an expression that is quadratic in
$\partial \tilde{g}$,
polynomial in $\tilde{g}$ and has  $\sqrt{|\tilde{g}|}$ in its denominator
\cite{Petersen}. From this we conclude that $\tilde{g}_{ij} \in C^{5, \alpha}$.
A bootstrap argument shows that $\tilde{g}_{ij} \in C^{\ell, \alpha}$
for any $\ell > 0$.

\end{proof}

Later, we will view the nonlinear map in \eqref{tmap} as a mapping from
\begin{align}
P_g : C^{4,\alpha}_{\delta} \rightarrow C^{0,\alpha}_{\delta-4},
\end{align}
where the spaces are certain weighted H\"older spaces with
weight function $w > 0$. Of course, since $w > 0$ and
$M$ is compact, these norms
are equivalent to the usual H\"older norms.
However, in the gluing construction, the weight function
will become large, and these norms will then not be
uniformly equivalent to the usual norms.

Next, we define the weighted norms we will
use. For $\delta \in \RR$, and a positive weight function $w > 0$,
\begin{align}
\Vert h \Vert_{C^0_{\delta}} \equiv \Vert w^{-\delta} h \Vert_{C^0}
= \sup_{x \in M} |w^{-\delta}(x) h(x)|.
\end{align}
For $0 < \alpha < 1$, define the semi-norm
\begin{align}
| h |_{C^{0, \alpha}_{\delta}} \equiv
\sup_{x \in M} \Big( w^{- \delta + \alpha}(x) \sup_{0<4 d(x,y) \leq w(x)}
\frac{|h(x) - h(y)|}{d(x,y)^{\alpha} }\Big).
\end{align}
Finally, define the norm
\begin{align}
\Vert h \Vert_{C^{k, \alpha}_{\delta}} \equiv \sum_{i = 0}^{k} \Vert \nabla^i h \Vert_{C^{0}_{\delta -i}}
+ | \nabla^k h |_{C^{0, \alpha}_{\delta}}.
\end{align}

\begin{remark}{\em
For the remainder of the paper, we fix $\alpha \in \RR$ satisyfing 
$0 < \alpha < 1$.
}
\end{remark}
\subsection{Estimate on the nonlinear terms}
The following proposition regarding the nonlinear structure of the
operator $P_g$ is crucial and will be used througout the paper.
 
\begin{proposition} \label{quadest} Write
\begin{align} \label{Pexp}
P_g(h) = P_g(0) + S_gh + Q_g(h),
\end{align}
where $S_g$ is the linearization of $P$. Then we have the following:  \\

\noindent $(i)$  If $h \in C^{4, \alpha}$ with $\Vert h \Vert_{C^{0}} < s_0$ small, then there exists a constant $C_1 = C_1(s_0)$ so that $Q_g$ satisfies
\begin{align} \label{Qsize} \begin{split}
|Q_g(h)| \leq C_1 &\Big\{ (|\nabla^2 Rm_g| + |Rm_g|^2) |h|^2 + |\nabla Rm_g||h||\nabla^2 h| + |\nabla Rm_g||h||\nabla h|  \\
&  + |Rm_g| |h| |\nabla^2 h|  + |Rm_g| |h| |\nabla h|^2  + |h||\nabla^4 h| \\
& + |\nabla h||\nabla^3 h|
+ |\nabla^2 h |^2 + |\nabla h|^2 |\nabla^2 h|   + |\nabla h|^4  \Big\}.
\end{split}
\end{align}

\noindent $(ii)$  Let $w$ denote a weight function, and assume
\begin{align} \label{wdoes} \begin{split}
w & \geq 1, \\
\delta &< 0.
\end{split}
\end{align}
In addition, assume there is a constant $C_0 > 0$ such that
\begin{align} \label{Rmw} \begin{split}
w^2 |Rm_g| &\leq C_0, \\
w^3 |\nabla_g Rm_g| & \leq C_0, \\
w^4 |\nabla_g^2 Rm_g | & \leq C_0.
\end{split}
\end{align}
Then, for $h_i \in C^{4, \alpha}_{\delta}$ with 
$\|h_i \|_{C^{4,\alpha}_{\delta}} < s_0$ small, 
there exists a constant $C_2 = C_2(s_0)$ so that $Q_g$ satisfies the following estimate:
\begin{align} \label{quadstructure}
\Vert Q_g(h_1) -  Q_g(h_2)\Vert_{C^{0,\alpha}_{\delta - 4}} \leq
C_2(  \Vert h_1 \Vert_{C^{4, \alpha}_{\delta}}
+ \Vert h_2 \Vert_{C^{4, \alpha}_{\delta}}) \cdot
\Vert  h_1 - h_2 \Vert_{C^{4, \alpha}_{\delta}}.
\end{align}
\end{proposition}

\begin{proof}
Since the proof involves a rather lengthy calculation we begin with a brief overview.  The tensor $B + t C$ can be schematically expressed
as
\begin{align} \label{Btscheme}
B_g + t C_g = g* g^{-1} * g^{-1} * \nabla_g^2 Rm_g + g*g^{-1} * g^{-1} * Rm_g * Rm_g,
\end{align}
where $Rm_g$ denotes the curvature tensor of $g$, $g^{-1} * \cdots * g^{-1} * A * B$ denotes any
linear combination of terms involving contractions of the tensor product $A \otimes B$, and $g^{-1} * \cdots * g^{-1} * \nabla_g^k * A$ denotes linear combinations of contractions of the $k$-th iterated covariant derivative of $A$.
Since the mapping $P$ is defined by
\begin{align} \label{Pdef2}
P_g(h) = B_{g + h} + t C_{g+h} + \mathcal{K}_{g+h}\Box_g \beta_g h,
\end{align}
the first step in proving the estimates is to analyze the curvature term
\begin{align} \label{leadexp} \begin{split}
B_{g + h} + t C_{g+h} &= (g+h)*(g+h)^{-1} * (g+h)^{-1} * \nabla_{g+h}^2 Rm_{g+h} \\
& \hskip.25in  + (g+h)*(g+h)^{-1} * (g + h)^{-1} * Rm_{g+h} * Rm_{g+h}.
\end{split}
\end{align}

The starting point is the formula
\begin{align} \label{CSdiff}
\Gamma(g + h)^{k}_{ij} = \Gamma(g)^{k}_{ij} + \frac{1}{2} (g+h)^{km}
\left\{ \nabla_j h_{im} + \nabla_i h_{jm} - \nabla_m h_{ij} \right\},
\end{align}
where $\Gamma(\cdot)$ denotes the Christoffel symbols of a metric.  In the following, any covariant derivative without a subscript
will mean with respect to the fixed metric $g$.  Using this formula and the notation introduced above, we can express the covariant derivative with respect to the metric $g+h$ as
\begin{align} \label{DTId}
\nabla_{g+ h} T = \nabla_g T + (g +h)^{-1} * \nabla_g h * T,
\end{align}
where $T$ is any tensor field. Also, by the standard formula for the $(1,3)$-curvature tensor in terms of the Christoffel symbols we have
\begin{align}
\label{e1rm}
Rm_{g + h} = Rm_g + (g +h)^{-1} * \nabla^2 h +
(g + h)^{-2} * \nabla h * \nabla h.
\end{align}

Taking the covariant derivative $\nabla_{g+h}$ of $Rm_{g+h}$ and repeatedly using (\ref{DTId}), we obtain
\begin{align} \label{DRmh}
\begin{split}
\nabla_{g+h} Rm_{g+h}& = \nabla Rm_g + (g +h)^{-1} * Rm_g * \nabla h + (g +h)^{-1} * \nabla^3 h\\
& + (g + h)^{-2} * (\nabla^2 h * \nabla h)
+ (g + h)^{-3} * ( \nabla h * \nabla h  * \nabla h).
\end{split}
\end{align}
Differentiating again, repeating the above
procedure and collecting terms we have
\begin{align} \label{D2Rmh}
\begin{split}
\nabla^2_{g+h} Rm_{g+h} & = \nabla^2 Rm_g
+ (g + h)^{-1} * \nabla Rm_g *\nabla h + (g+h)^{-1} * Rm_g * \nabla^2 h  \\
& \hskip.2in + (g+h)^{-2} * Rm_g * \nabla h * \nabla h + (g+h)^{-1} * \nabla_g^4 h \\
& \hskip.2in  + (g+h)^{-2} * \nabla^3 h * \nabla h
+ (g+h)^{-2} * \nabla^2 h  * \nabla^2 h  \\
& \hskip.2in + (g+h)^{-3} *  \nabla^2 h * \nabla h * \nabla h + (g + h)^{-4} *   \nabla h * \nabla h  * \nabla h * \nabla h.
\end{split}
\end{align}
Therefore,
\begin{align} \label{divtermscheme}
\begin{split}
(g&+h)*(g+h)^{-2} * \nabla^2_{g+h} Rm_{g+h}  =  \\
& (g+h)* \Big\{ (g+h)^{-2} * \nabla^2 Rm_g
+ (g + h)^{-3} * \nabla Rm_g *\nabla h \\
& + (g+h)^{-3} * Rm_g * \nabla^2 h + (g+h)^{-4} * Rm_g * \nabla h * \nabla h\\
& + (g+h)^{-3} * \nabla^4 h + (g+h)^{-4} * \nabla^3 h * \nabla h\\
&+ (g+h)^{-4} * \nabla^2 h  * \nabla^2 h + (g+h)^{-5} *  \nabla^2 h * \nabla h * \nabla h + \\
&+ (g + h)^{-6} *   \nabla h * \nabla h  * \nabla h * \nabla h \Big\}.
\end{split}
\end{align}
Using (\ref{e1rm}), we have a similar expression for the second term in (\ref{leadexp}):
\begin{align} \label{quadtermscheme}
\begin{split}
&(g+h)*(g+h)^{-2} *  Rm_{g+h} * Rm_{g+h} \\
& =(g+h)* \Big\{ (g+h)^{-2} * Rm_g * Rm_g + (g+h)^{-3} * Rm_g * \nabla^2 h\\
& + (g+h)^{-4} * Rm_g * \nabla h * \nabla h + (g+h)^{-4} * \nabla^2 h * \nabla^2 h  \\
& + (g+h)^{-5} * \nabla^2 h * \nabla h * \nabla h + (g+h)^{-6}*  \nabla h * \nabla h  * \nabla h * \nabla h \Big\}.
\end{split}
\end{align}
Combining (\ref{divtermscheme}) and (\ref{quadtermscheme}) gives
\begin{align} \label{Btscheme2}
\begin{split}
&(g +h) * \big\{ (g+h)^{-2} * \nabla^2_{g+h} Rm_{g+h} + (g+h)^{-2} *  Rm_{g+h} * Rm_{g+h} \big\} \\
& = (g +h) * \big\{  (g+h)^{-2} * \nabla^2 Rm_g + (g+h)^{-2} * Rm_g * Rm_g \\
&+ (g + h)^{-3} * \nabla Rm_g *\nabla h + (g+h)^{-3} * Rm_g * \nabla^2 h \\
&+ (g+h)^{-4} * Rm_g * \nabla h * \nabla h + (g+h)^{-3} * \nabla^4 h \\
&+ (g+h)^{-4} * \nabla^3 h * \nabla h + (g+h)^{-4} * \nabla^2 h  * \nabla^2 h  \\
& + (g+h)^{-5} *  \nabla^2 h * \nabla h * \nabla h + (g + h)^{-6} *   \nabla h * \nabla h  * \nabla h * \nabla h \Big\}.
\end{split}
\end{align}

Returning to the formula (\ref{Pdef2}), the gauge-fixing term can be written
\begin{align} \label{GF1} \begin{split}
\mathcal{K}_{g+h}&\Box_g \beta_g h =  (g+h)^{-1} * (g+h) *\nabla_{g+h} ( \Box_g \beta_g h)\\
& =  (g+h)^{-1} * (g+h) * ( \nabla_g + (g+h)^{-1} * \nabla h) * ( \Box_g \beta_g h)\\
& = (g+h)^{-1} * (g+h) * g^{-3} * g * \nabla^4 h \\
& + (g+h)^{-2} * (g+h)* g^{-3} * g * \nabla h * \nabla^3 h.
\end{split}
\end{align}
Combining (\ref{Btscheme2}) and (\ref{GF1}) we finally have
\begin{align} \label{PS1} \begin{split}
P_g(h) = (g+&h) * \Big\{ (g+h)^{-2} * \nabla^2 Rm_g + (g+h)^{-2} * Rm_g * Rm_g  \\
& + (g + h)^{-3} * \nabla Rm_g *\nabla h + (g+h)^{-3} * Rm_g * \nabla^2 h \\
& + (g+h)^{-4} * Rm_g * \nabla h * \nabla h + (g+h)^{-3} * \nabla^4 h \\
& + (g+h)^{-1}* g^{-3} *g * \nabla^4 h    + (g+h)^{-4} * \nabla^3 h * \nabla h \\
& + (g+h)^{-2} * g^{-3} * g* \nabla^3 h * \nabla h
+ (g+h)^{-4} * \nabla^2 h  * \nabla^2 h  \\
&+ (g+h)^{-5} *  \nabla^2 h * \nabla h * \nabla h + (g + h)^{-6} *   \nabla h * \nabla h  * \nabla h * \nabla h \Big\}.
\end{split}
\end{align}

Since we are trying to estimate the remainder terms in the Taylor expansion of $P(h)$, we want to write the above expression in terms of its linearization; i.e.,
\begin{align*}
P_g(h) &= P_g(0) + Sh + \cdots \\
&= B_g + t C_g + Sh + \cdots
\end{align*}
To do this, we use the identity (which holds for $h$ small)
\begin{align}
\label{id}
(g + h)^{-1} - g^{-1} =  g^{-2} * h + \sum_{k \geq 2} g^{-k-1} * h^k,
\end{align}
which follows from the usual geometric series formula.  Therefore,
\begin{align} \label{geomdiff}
(g+h_1)^{-1} - (g+h_2)^{-1} = g^{-2} * (h_1 - h_2) + \sum_{k \geq 2} g^{-k-1} * \big( h_1^k - h_2^k \big).
\end{align}
Each term in the sum in (\ref{geomdiff}) can be written
\begin{align} \label{diffhk}
 g^{-k-1} * h_1^k - g^{-k-1} * h_2^k = g^{-k-1} * (h_1 - h_2) * \sum_{i+j = k-1} h_1^i * h_2^j.
 \end{align}
Therefore, for $h$ small we can write
\begin{align} \label{rem1}
(g + h)^{-1} - g^{-1} = g^{-2} * h + r_1(h),
\end{align}
where $r_1$ satisfies
\begin{align} \label{r1diff}
| r_1(h_1) - r_1(h_2) | \leq C(g) \big( |h_1| + |h_2| \big) |h_1 - h_2|
\end{align}
for $h_1, h_2$ small. In general we can write
\begin{align} \label{remk}
(g+h)^{-k} - g^{-k} = g^{-k - 1} * h + r_k(h),
\end{align}
where the remainder satisfies
\begin{align} \label{rkdiff}
| r_k(h_1) - r_k(h_2) | \leq C_k(g) \big( |h_1| + |h_2| \big) |h_1 - h_2|,
\end{align}
with a similar estimate for the H\"older norm.

We note that, using the restrictions on the weight function assumed in \eqref{wdoes}, 
the assumption that $\|h_i \|_{C^{4,\alpha}_{\delta}}$ is small 
implies that the $C^0$-norm of $h_i$ is also small, so we are free
to employ \eqref{rkdiff} in the following. 

Next, we substitute (\ref{remk}) into each term of (\ref{PS1}) involving a power of $(g+h)^{-1}$, then collect all terms which are zeroth order in $h$ (which combine to give $P_g(0)$),
those which are linear in $h$ (which combine to give $Sh$), and those which are higher order in $h$.
For example, consider the term
\begin{align*}
(g+h)* (g+h)^{-3} * \nabla^4 h
&= (g + h)* ( g^{-3} + g^{-4} * h + r_3(h)) * \nabla^4 h\\
&= g * g^{-3} * \nabla^4 h + g * g^{-4} * h * \nabla^4 h
+ g* r_3(h) * \nabla^4 h\\
& + g^{-3} *h * \nabla^4 h + g^{-4}* h * h * \nabla^4 h + r_3(h) * h * \nabla^4 h.
\end{align*}

Next, apply (\ref{remk}) to each term in (\ref{PS1}) in a similar fashion, and write the resulting expression as
\begin{align} \label{PS2}
P_g(h)= P_g(0) + Sh + Q(h),
\end{align}
where $Q$ is
\begin{align} \label{Qexp} \begin{split}
Q(h) & = (g+h) * \Big\{ r_2(h) * \nabla^2 Rm_g + r_2(h) * Rm_g * Rm_g + g^{-4} * Rm_g * h * \nabla^2 h \\
& + g^{-4} * \nabla Rm_g *  h * \nabla h + r_3(h) * \nabla Rm_g * \nabla h + g^{-4} * Rm_g *h * \nabla^2 h  \\
& + r_3(h) * Rm_g * \nabla^2 h + g^{-5} * Rm_g * h * \nabla h * \nabla h  \\
& + r_4(h) * Rm_g * \nabla h * \nabla h + g^{-4} * h * \nabla^4 h + g^{-5} * g * h * \nabla^4 h \\
&+ g^{-3}*g * r_1(h) * \nabla^4 h + r_3(h) * \nabla^4 h \\
& + g^{-4} * \nabla^3 h * \nabla h + g^{-5} * h * \nabla h * \nabla^3 h + r_4(h) * \nabla h * \nabla^3 h \\
& + g^{-5} * g * \nabla^3 h * \nabla h +  g^{-6}*g * h * \nabla h * \nabla^3 h + g^{-3}*g * r_2(h) * \nabla h * \nabla^3 h  \\
& + g^{-4} * \nabla^2 h * \nabla^2 h + g^{-5} * h * \nabla^2 h * \nabla^2 h + r_4(h) * \nabla^2 h * \nabla^2 h \\
& + g^{-3} * \nabla h * \nabla h * \nabla^2 h + g^{-6} * h * \nabla h * \nabla h * \nabla^2 h \\
& + r_5(h) * \nabla h * \nabla h * \nabla^2 h + g^{-6} * \nabla h * \nabla h * \nabla h * \nabla h   \\
& + g^{-7} * h * \nabla h * \nabla h * \nabla h * \nabla h + r_6(h)* \nabla h * \nabla h * \nabla h * \nabla h \Big\}.
\end{split}
\end{align}
The estimate (\ref{Qsize}) follows from considering each term in
\eqref{Qexp}, inequality (\ref{rkdiff}), and the smallness of $h$.

We can then prove (\ref{quadstructure}) by a fairly straightforward---but, due to the number of terms, very lengthy---process.  We will provide the details for estimating some representative terms; the rest can be handled similarly.

For example, consider the term
\begin{align} \label{Ttdef}
T(h) = g* g^{-4} * h * \nabla^4 h.
\end{align}
Then
\begin{align*}
T(h_1) - T(h_2) &= g* g^{-4} * h_1 * \nabla^4 h_1 - g*g^{-4} * h_2 * \nabla^4 h_2 \\
&= g* g^{-4} * (h_1 - h_2) * \nabla^4 h_1 + g*g^{-4} * h_2 * \nabla^4 (h_1 - h_2)
\end{align*}
If $w$ denotes the weight, then this implies
\begin{align*}
|T(h_1) - T(h_2)|w^{4 - \delta} & \leq |h_1 - h_2| |\nabla^4 h_1| w^{4 - \delta} + |h_2| | \nabla^4 (h_1 - h_2)| w^{4 - \delta} \\
&=  \big\{ |h_1 - h_2| w^{-\delta} \big\} \big\{ |\nabla^4 h_1| w^{4 - \delta}\big\} w^{\delta}  \\
& \hskip.5in + \big\{ |h_2| w^{-\delta} \big\} \big\{ | \nabla^4 (h_1 - h_2)| w^{4 - \delta} \big\}
w^{\delta}.
\end{align*}
Since $w \geq 1$ and $\delta < 0$, taking the supremum gives
\begin{align} \label{semi} \begin{split}
\|T(h_1) - T(h_2)\|_{C^0_{\delta - 4}} &\leq \Big\{ \|h_1 - h_2\|_{C^0_{\delta}}  \| h_1\|_{C^4_{\delta - 4}}  + \|h_2\|_{C^0_{\delta}} \| h_1 - h_2\|_{C^4_{\delta - 4}} \Big\} \\
& \leq C(  \Vert h_1 \Vert_{C^{4}_{\delta}}
+ \Vert h_2 \Vert_{C^{4}_{\delta}}) \cdot
\Vert  h_1 - h_2 \Vert_{C^{4}_{\delta}},
\end{split}
\end{align}

Next, consider the term
\begin{align*}
\rho(h) = g * g^{-4} * Rm_g * h * \nabla^2 h.
\end{align*}
Taking differences as we did above yields
\begin{align*}
|\rho(h_1) - \rho(h_2)| \leq |Rm_g||h_1 - h_2| |\nabla^2 h_1| + |Rm_g||h_2||\nabla^2 (h_1 - h_2)|.
\end{align*}
Multiplying by the appropriate power of the weight,
\begin{align*}
|\rho(h_1) - \rho(h_2)|w^{4 - \delta} & \leq \big\{w^2 |Rm_g| \big\} \big\{ |h_1 - h_2|w^{-\delta} \big\} \big\{ |\nabla^2 h_1|w^{2-\delta} \big\} w^{\delta} \\
 & \hskip.5in + \big\{ w^2 |Rm_g|\big\} \big\{ |h_2| w^{-\delta} \big\} \big\{ |\nabla^2 (h_1 - h_2)| w^{2 - \delta} \big\} w^{\delta}.
\end{align*}
Using (\ref{Rmw}), we arrive at an estimate similar to (\ref{semi}).

Finally, let us consider a term in $Qh$ which has a higher order of homogeneity,
\begin{align}
K(h) = g * r_6(h) * \nabla h * \nabla h * \nabla h * \nabla h.
\end{align}
Then
\begin{align*}
&K(h_1) - K(h_2) = g*[ r_6(h_1) - r_6(h_2)] * \nabla h_1 * \nabla h_1 * \nabla h_1 * \nabla h_1 \\
& + g*r_6(h_2) \big\{ \nabla(h_1 - h_2) * \nabla h_1 * \nabla h_1 * \nabla h_1 + \nabla h_2 * \nabla (h_1 - h_2) * \nabla h_1 * \nabla h_1  \\
& + \nabla h_2 * \nabla h_2 * \nabla (h_1 - h_2) * \nabla h_1 + \nabla h_2 * \nabla h_2 * \nabla h_2 * \nabla (h_1 - h_2) \big\}
\end{align*}
Multiplying by the weight,
\begin{align*}
|K(h_1) - K(h_2)| &w^{4 - \delta} \leq C \{ | h_1 - h_2| w^{-\delta}\}  \{ |\nabla h_1| w^{1-\delta}\}^4 w^{4 \delta} \\
& + C |\nabla(h_1 - h_2)|w^{1 - \delta} \Big[  \{  |\nabla h_1| w^{1-\delta} \}^3  + \{ |\nabla h_2| w^{1-\delta} \} \{ |\nabla h_1|w^{1-\delta}\}^2 \\
& + \{|\nabla h_1| w^{1-\delta}\} \{ |\nabla h_2|w^{1-\delta}\}^2 +  \{  |\nabla h_1| w^{1-\delta} \}^3  \Big]w^{3 \delta},
\end{align*}
which gives an estimate as in (\ref{semi}).

Similar arguments (estimating difference quotients) give the Holder estimate in (\ref{quadstructure}).

\end{proof}

Since the operator $P_g$ differs from $B^t$ only by the gauge term, 
a similar estimate holds for $B^t$, see the 
following Proposition. This fact will be used in several
places below (e.g., Proposition \ref{NewBtsize}).
\begin{proposition}
\label{QRemark}
Let $(B^t_g)'$ denote the linearization of the $B^t$ tensor:
\begin{align*}
(B^t_g)' h = \frac{d}{ds} B^t(g+ sh) \big|_{s=0}.
\end{align*}
If we write
\begin{align}
B^t(g + h) = B^t(g) + (B^t_g)' h  + \mathcal{Q}_g(h),
\end{align}
then under the same assumptions as 
in (i) of Proposition \ref{quadest},
the remainder $\mathcal{Q}$ satisfies the 
estimate \eqref{Qsize}.
\end{proposition}

\section{Cokernel on a compact manifold}
\label{comkersec}
On a compact manifold $(Z,g_Z)$, with basepoint $z_0$,
we define the weight function to be a smooth function
satisfying
\begin{align}
w(z) =
\begin{cases}
d(z,z_0)  &   d(z,z_0) < 1/2\\
1  &   d(z,z_0) \geq 1,
\end{cases}
\end{align}
and $1/2 \leq w(z) \leq 1$ when $1/2 \leq d(z,z_0) \leq 1$.

\begin{theorem}
\label{comker}
Let $(Z, g_Z)$ be either $\CP^2$ with the Fubini-Study metric $g_{FS}$,
or $S^2 \times S^2$ with the product metric $g_{S^2} \times g_{S^2}$.
Assume that
\begin{align}
t < 0,
\end{align}
and let $h \in C^{4,\alpha}_{\delta}$ solve the equation
\begin{align}
S^t(h)  = 0
\end{align}
for $\delta < 0$ with $|\delta|$ small. If $h$ is toric-invariant and
diagonally invariant, then $h = c \cdot g_Z$ for some constant $c \in \RR$.
Consequently, if $h$ satisfies
\begin{align}
h = O(|z|^{\delta})
\end{align}
as $|z| \rightarrow 0$, for $\delta > 0$, then $h \equiv 0$.
\end{theorem}
\begin{proof}
For $t \neq 0$, we define $H^1_t$ to be
the kernel of the linearization of $P_g$:
\begin{align} \label{H1def}
H^1_t = H^1_t(M,g) = \big\{ h \in \overline{S}_{0}^2(T^{*}M)\ \big|\ S^t_g h = 0 \big\},
\end{align}
where
\begin{align}
\label{S20}
\overline{S}_0^2(T^{*}M) = \Big\{ h \in C^{4, \alpha} (S^2(T^{*}M)) \ :\ \int (tr_g\ h)\ dV_g = 0
\Big\}.
\end{align}
For $t = 0$ (the Bach tensor), we restrict to
traceless tensors:
\begin{align} 
\label{BachH1def}
H^1_{0} = H^1_{0}(M,g) = \big\{
h \in C^{4,\alpha}(S_{0}^2(T^{*}M))\ \big|\ S^0_g h = 0 \big\}.
\end{align}
If $H^1_t(M,g) = \{0\}$, we say that $(M,g)$ is {\em{infinitesimally $B^t$-rigid}}.
We next quote two crucial rigidity theorems from
\cite{GV11} with the following caveat:
as pointed out in the proof of Proposition \ref{ellprop},
a different parametrization $\tau$ was used in \cite{GV11}. The
relation between $\tau$ and $t$ is given by
\begin{align}
\tau = \frac{t}{2} - \frac{1}{3}.
\end{align}
The following is then a direct consequence of \cite[Theorem 7.8]{GV11}:
\begin{theorem}[\cite{GV11}]
\label{cp2rig}
On $(\CP^2, g_{FS})$, $H^1_t = 0$ provided that $t < 1$.
\end{theorem}
The following is a direct consequence of \cite[Theorem~7.13]{GV11}:
\begin{theorem}[\cite{GV11}]
\label{s2s2rig}
On $(S^2 \times S^2, g_{S^2 \times S^2})$,
$H^1_t = 0$ provided that $t < 2/3$ and $t \neq - 1/3$.
If $t = - 1/3$, then $H^1_t$ is one-dimensional and spanned by
the element $g_1 - g_2$.
\end{theorem}

If one knows that $h \in C^{4,\alpha}(Z)$, 
then Theorem \ref{comker} follows immediately
from Theorems \ref{cp2rig} and \ref{s2s2rig}. The only symmetry
needed for this part is the diagonal invariance for $t = -1/3$, which rules
out the kernel element $g_1 - g_2$.  We will next employ the
symmetries, in a crucial way, to prove smoothness.

\begin{proposition}
\label{indclaim}
If $t \neq 0$,
the indicial roots of $S^t$ are contained in $\ZZ$.
\end{proposition}
\begin{proof}
To determine the indicial
roots of $S^t$, we need to analyze homogeneous solutions of the equation
\begin{align} \label{SformEuc}
\begin{split}
S_0 h &\equiv \Delta_0^2 h - 2\big(t + \frac{5}{24}\big) \nabla_0^2 (\Delta_0 tr\ h) - 2 \big( t + \frac{5}{24}\big) \Delta_0 (\delta_0^2 h) g_0 \\
& \ \ \ \  + 2\big( t + \frac{5}{6}\big) \nabla_0^2 (\delta_0^2 h) + 2\big(t - \frac{7}{96}\big) \Delta_0^2 (tr\ h)g_0  = 0
\end{split}
\end{align}
on Euclidean space $(\RR^4 \setminus \{0\}, g_0)$.
Assume by contradiction that $h$ solves
\eqref{SformEuc} in $\RR^4 \setminus \{0\}$,
with $h$ corresponding to an
indicial root of $u + \I v\in \CC \setminus \{ \ZZ \}$,
and $u, v \in \RR$. This means that
$h$ has components of the form $r^{u} \cos(v r)$,
$r^{u} \sin(vr)$, or a polynomial in $\log(r)$ times one
of these (we say such a solution is
homogeneous of degree $u + \I v$).

Taking the trace of \eqref{SformEuc} gives
\begin{align*}
\Delta_0 \big[ \Delta_0 (tr\ h) - \delta_0^2 h \big] = 0,
\end{align*}
with $\Delta_0 (tr\ h ) - \delta_0^2 h $ homogeneous
of degree $u -2 + \I v$.
Since the indicial roots of the Laplacian are $\ZZ \setminus \{-1\}$,
it follows that
\begin{align} \label{treqn}
\Delta_0 (tr\ h ) - \delta_0^2 h  = 0.
\end{align}
Substituting this into (\ref{SformEuc}) implies that
\begin{align*}
\Delta_0^2 h  + \frac{5}{4}\nabla_0^2(\Delta_0 tr\ h ) - \frac{9}{16}\Delta_0^2(tr\ h )\cdot g_0 = 0.
\end{align*}
Applying the operator $\delta_0^2$ and using (\ref{treqn}) we get
\begin{align*}
\Delta_0^2 (\delta_0^2 h) = 0,
\end{align*}
which implies $\Delta_0 tr\  h  = \delta_0^2 h  = 0$, hence  $\Delta_0^2 h \equiv 0$.
We note that the indicial roots of $\Delta_0$ on symmetric tensors
are the same as those of the Laplacian on functions, which is $\ZZ \setminus \{-1\}$.
Since $u + \I v$ is not an indicial root of $\Delta_0^2$ on symmetric tensors,
we have a contradiction.
\end{proof}

 To analyze the indicial root at $0$, we first note that
any constant tensor on $\RR^4$ is a homogeneous degree zero solution,
and the dimension of the space of these solutions is $10$.
We claim that the space of all homogeneous solutions
of degree zero is of dimension $20$.
To see this, choose weight function on $\RR^4$ to be given by
\begin{align}
w(x) =
\begin{cases}
|x|  &   |x| \geq 1\\
1  &   d(x,x_0) < 1.
\end{cases}
\end{align}
With this weight function, for $\delta > 0$ small but nonzero,
consider the operator as mapping from
\begin{align}
S^t_{g_0} : C^{4,\alpha}_\delta \rightarrow C^{0,\alpha}_{\delta-4}.
\end{align}
With obvious notation, the relative index theorem of \cite{LockhartMcOwen}
states that
\begin{align}
\label{indjump}
Ind(\delta) - Ind(-\delta) &= N(0),
\end{align}
where $N(0)$ is the space of all homogeneous solutions of degree zero
on $\RR^4 \setminus \{0\}$.
We note the important fact that any bounded solution globally
defined on $\RR^4$ must be constant,
the proof is as in  \cite[Proposition 5.4]{AV12}
(the key being that the flat metric is rigid).
This implies that any globally defined decaying solution is trivial,
so we have $\dim Ker(-\delta) = 0$.
Since the adjoint weight of $\delta$ is $- \delta$, \eqref{indjump} may
then be written as
\begin{align}
2 \cdot \dim Ker ( \delta) = N(0).
\end{align}
If $\delta$ is sufficiently small, it is not an indicial root,
so any kernel element defined on all of $\RR^4$ satisfying
$h = O(|x|^{\delta})$ as $|x| \rightarrow \infty$ is constant.
Therefore $\dim Ker(\delta) = 10$, which implies that $N(0) = 20$.

The only symmetric constant tensors invariant under the
standard diagonal torus action are multiples of the identity matrix,
or multiples of the matrix
\begin{align}
\left(
\begin{matrix}
I_2  &  0 \\
0  & - I_2 \\
\end{matrix}
\right),
\end{align}
where $I_2$ is the $2 \times 2$ identity matrix.
It is easy to see that this element is not invariant under the diagonal symmetry.
Consequently, there are only $2$ invariant degree zero solutions on
$\RR^4 \setminus \{0\}$: the identity matrix, and another solution with
log-type growth (we will not need the explicit formula).
Another application of the relative index theorem applied to
the compact manifold (details are similar to above) shows that, since $c \cdot g$ extends to a global
solution, the log-type solution does not extend to a global solution
on $Z \setminus \{z_0\}$.

To finish the proof, if $h \in C^{4,\alpha}_{\delta}$ is a solution
on $Z \setminus \{z_0\}$ for $\delta < 0$ with $|\delta|$ sufficiently
small which is invariant under the group action,
then there is an expansion
\begin{align}
h = c \cdot g + O(|z|^{\epsilon})
\end{align}
for some constant $c \in \RR$ and $ \epsilon > 0 $ as $|z| \rightarrow 0$.
Since the leading term is a global solution, we then have that
$\tilde{h} = h - c \cdot g$ is solution on $Z \setminus \{z_0\}$
satisfying $\tilde{h} = O(|z|^{\epsilon})$ as $|z| \rightarrow 0$.
A standard integration-by-parts argument shows that
$\tilde{h}$ extends to a weak solution on all of $Z$, and
is therefore smooth by elliptic regularity. By the above,
$\tilde{h} \equiv 0$.
\end{proof}

\section{Cokernel on an asymptotically flat manifold}
\label{afkersec}

Let $(N,g)$ be the Green's function metric of a 
compact manifold $(Y,g_Y)$ with positive scalar curvature: more precisely,
\begin{align} \label{Nprops} \begin{split}
N &= Y \setminus \{ y_0 \}, \ y_0 \in Y; \\
g &= G^2 g_Y,
\end{split}
\end{align}
where $G$ is the Green's function of the conformal Laplacian with pole at
$y_0 \in Y$.  Assume $(Y,g_Y)$ is Bach-flat and infinitesimally Bach-rigid,
that is, $H^1_0(Y,g_Y) = \{0\}$.

Let $\{x^i\}$ denote an inverted normal coordinate system,
and choose weight function $w = w(x)$ to be given by
\begin{align}
w(x) =
\begin{cases}
|x|  &   |x| \geq R_0\\
1  &   d(x,x_0) < 1,
\end{cases}
\end{align}
where $R_0$ is large, and $x_0 \in N$ is a basepoint.

\begin{theorem}
\label{afker}

Assume $\delta < 0$ with $|\delta|$ small, and let $h \in C^{4, \alpha}_{\delta}$ solve the equation
\begin{align}
\label{cokeh}
S^t(h) = B'(h) + t C'(h) + \mathcal{K} \delta \mathcal{K} \delta (\overset{\circ}{h}) = 0,
\end{align}
where $t \neq 0$.

Then
\begin{align} \label{o1def}
h = \mathcal{K} \omega_1 + f \cdot g_N,
\end{align}
where $\omega_1$ and $f$ satisfy
\begin{align} \label{o1facts} \begin{split}
\Box \omega_1 &= 0, \\
\Delta f &= -\frac{1}{3} \langle Ric, \mathcal{K} \omega_1 \rangle,
\end{split}
\end{align}
where $\Box = \delta \mathcal{K}$.

Furthermore, suppose $(N,g)$ is either the Burns metric or the Green's function
metric of the product metric on $S^2 \times S^2$.  If $h$ is toric invariant and
diagonally invariant and $\delta > 0$, then $\omega$ and $f$ can also 
be chosen to be toric invariant and diagonally invariant, with
\begin{align} \label{asymo}
\omega & = c \cdot x^i dx^i + O(|x|^{-1 + \epsilon}),\\
\label{fexp}
f(x) &= c_0 + \frac{c'}{|x|^2} + O'(|x|^{-4 + \epsilon}),
\end{align}
where $c_0, c, c' \in \RR$ are constants, as $r \rightarrow \infty$,
for any $\epsilon > 0$.
\end{theorem}

 The remainder of this section will be devoted to the proof of Theorem \ref{afker}.
Since the Bach tensor is conformally invariant it follows that $(N,g)$ is also Bach-flat.  Also, since $(N,g)$ is scalar
flat it is also $B^t$-flat, for any value of $t$.  We also note that $h$ is smooth since $S^t$ is elliptic.

The splitting in (\ref{o1def}) reflects the fact that each term in the linearization must vanish:

\begin{proposition}  \label{prop1}  Each term in (\ref{cokeh}) vanishes; i.e.,
\begin{align} \label{splitzed} \begin{split}
B^{\prime}(h) & = 0, \\
C'(h) &= 0, \\
\mathcal{K} \delta \mathcal{K} \delta (\overset{\circ}{h}) &= 0.
\end{split}
\end{align}
Furthermore,
\begin{align} \label{zdivfree}
\delta (\overset{\circ}{h}) = 0.
\end{align}
\end{proposition}

\begin{proof} Since $(N,g)$ is $B^t$-flat, if we linearize the identity
\begin{align*}
\delta B^t = 0
\end{align*}
at $g$ we find
\begin{align*}
\delta [(B^t)'h] + (\delta_h') B^t = 0 \ \Rightarrow \ \delta [(B^t)'h] = 0.
\end{align*}
Therefore, taking the divergence of both sides of (\ref{cokeh}) gives
\begin{align} \label{box2}
\Box^2 \delta (\overset{\circ}{h}) = 0.
\end{align}

\begin{proposition}
\label{boxker} There are no decaying elements in the kernel of $\Box$.
\end{proposition}
\begin{proof}
To see this, we note the formula
\begin{align}
\Box \omega = \frac{3}{2} d \delta \omega + \delta d \omega + 2 Ric(\omega, \cdot).
\end{align}
Since the Ricci tensor decays, to determine the indicial
roots of $\Box$, we need to analyze homogeneous solutions of the operator
\begin{align}
\label{boxEuc}
\Box \omega = \frac{3}{2} d \delta \omega + \delta d \omega
\end{align}
on Euclidean space $(\RR^4, g_0)$. We claim that the indicial roots
are contained in $\ZZ \setminus \{-1\}$. To prove this,
assume by contradiction that $\omega$ solves
$\Box \omega = 0$ in $\RR^4 \setminus \{0\}$,
with $\omega$ corresponding to an
indicial root of $u + \I v\in \CC \setminus \{ \ZZ \setminus \{-1\}\}$,
and $u, v \in \RR$. This means that
$\omega$ has components of the form $r^{u} \cos(v r)$ or
$r^{u} \sin(vr), r^u \log (r)$ (similarly to above, we say such a solution is
homogeneous of degree $u + \I v$).
Applying $d$ to \eqref{boxEuc} yields that
\begin{align}
d \delta d \omega = (d \delta + \delta d) d \omega = - \Delta_H d \omega.
\end{align}
We note that the indicial roots of $\Delta_H$ are exactly $\ZZ \setminus \{-1\}$
(this is easily seen since the leading term is the rough Laplacian, so the
indicial roots are the same as for the Laplacian on functions).
Since $d \omega$ is homogeneous of degree $u - 1 + \I v$, which is not
an indicial root of $\Delta_H$, we conclude that $d \omega = 0$.
A similiar argument shows that $\delta \omega =0$.
Since both $d \omega = 0$ and $\delta \omega = 0$, we have that
$\Delta_H(\omega) = 0$, which is a contradiction since
$u + \I v$ was chosen to not be an indicial root of $\Delta_H$.

Consequently, by standard weighted space theory,
any decaying solution of $\Box \xi = 0$ on an AF space
must satisfy $\xi = O(r^{-2})$ as $r \rightarrow \infty$
\cite{Bartnik}.
An elementary integration by parts argument then shows that
$\mathcal{K}\xi = 0$. As there are no decaying conformal
Killing fields on an AF space, we conclude that $\xi = 0$.
\end{proof}
\begin{remark}{\em
By a separation of variables argument as in \cite[Section~4.1]{AV12},
it is straightforward to show that the indicial roots of
$\Box$ are in fact exactly $\ZZ \setminus \{-1\}$, although
we will not need this fact.
}
\end{remark}

By this proposition and \eqref{box2},
\begin{align*}
\Box \delta (\overset{\circ}{h}) = 0.
\end{align*}
Applying the result once again gives (\ref{zdivfree}):
\begin{align} \label{doc}
\delta \oc{h} = 0.
\end{align}
In particular,
\begin{align*}
\mathcal{K} \delta \mathcal{K} \delta (\overset{\circ}{h}) = 0
\end{align*}
and consequently
\begin{align} \label{Btpz}
 B'(h) + t C'(h) = 0.
\end{align}

If we linearize the trace-free property of the Bach tensor at $g$ it follows that
\begin{align*}
tr B'(h) = 0.
\end{align*}
Therefore, taking the trace of (\ref{Btpz}) gives
\begin{align} \label{trzed} \begin{split}
0 &= tr\ B'(h) + t\ tr\ C'(h) \\
&= t\ tr\ C'(h).
\end{split}
\end{align}

\begin{lemma}  \label{lemma1} If $(X^4,g)$ is either scalar-flat or Einstein, then
\begin{align} \label{trCp} \begin{split}
tr\ C'(h) &= - 6 \Delta R'(h) \\
&= - 6 \Delta [ - \Delta(tr\ h) + \delta^2 h - \langle Ric, h \rangle],
\end{split}
\end{align}
where $R'$ denotes the linearization of the scalar curvature.
\end{lemma}

\begin{proof}
Since $C = 0$ for scalar-flat or Einstein metrics, we have
\begin{align*}
(tr\ C)' &= (tr\ )' C + tr (C') \\
&= tr (C').
\end{align*}
Also, $tr\ C = - 6 \Delta R$, so
\begin{align*}
tr (C') = (tr\ C)' = -6 (\Delta)' R - 6 \Delta R',
\end{align*}
and since $R$ is constant we get
\begin{align*}
tr (C') = - 6 \Delta R',
\end{align*}
as claimed.
\end{proof}
In view of (\ref{trzed}) and the preceding lemma we have
\begin{align*}
\Delta R'(h) = \Delta [ - \Delta(tr\ h) + \delta^2 h - \langle Ric, h \rangle] = 0.
\end{align*}
Since
\begin{align*}
|\nabla^2 h| = O(|x|^{\delta - 2}), \quad \langle Ric, h \rangle = O(|x|^{\delta - 4}),
\end{align*}
it follows that $R'(h)$ is a decaying harmonic function.  Therefore,
\begin{align} \label{Rpzed}
R'(h) = - \Delta(tr\ h) + \delta^2 h - \langle Ric, h \rangle = 0.
\end{align}

Recall $C$ is given by
\begin{align} \label{Cform}
C = 2 \nabla^2 R - 2 (\Delta R)g - 2 R \big( Ric - \frac{1}{4}R g \big).
\end{align}
Linearizing this at $g$ (which is scalar-flat) gives
\begin{align*}
C'(h) = 2 \nabla^2 R'(h) - 2 \Delta R'(h) g - 2 R'(h) \cdot Ric.
\end{align*}
From (\ref{Rpzed}), it follows that $C'(h) = 0$, which completes the proof of Proposition \ref{prop1}.
\end{proof}

Write
\begin{align} \label{hsplit}
h = \overset{\circ}{h} + f g,
\end{align}
where $f = (tr\ h)/4$.  The conformal invariance of the Bach tensor leads to the formula
\begin{align*}
B'( \phi g) = - 2 \phi B
\end{align*}
for any function $\phi$.  Since $g$ is Bach-flat this implies
\begin{align*}
0 &= B'(\oc{h} + fg ) \\
&= B'(\oc{h}).
\end{align*}
It follows from  \cite[Proposition 2.1]{AV12} that any decaying, transverse-traceless element in the kernel of $B'$ must decay quadratically,
hence
\begin{align} \label{2decay}
| \oc{h}| = O(|x|^{-2}),
\end{align}
as $|x| \rightarrow \infty$.

Conformal invariance of the Bach tensor also implies the invariance of its linearization:
\begin{align}
0 = B'_{g} (\oc{h}) = B'_{G^2 g_0}(\oc{h}) = G^{-2} B'_{g_0}(G^{-2} \oc{h}).
\end{align}
Denote
\begin{align} \label{thdef}
\h = G^{-2} \oc{h}.
\end{align}
Then $\h \in C^{\infty}(Y \setminus \{ y_0 \})$, and
\begin{align} \label{thk}
B_{g_0}' \h = 0.
\end{align}
In addition, since $\oc{h}$ decays quadratically at infinity, $\h$ vanishes quadratically at $y_0$.  To see this, first note that
\begin{align} \label{thsize} \begin{split}
| \h |_{g_0}^2 &= (g_0)^{ik} (g_0)^{k \ell} \h_{ij} \h_{k \ell} \\
&= G^4 g^{ik} g^{j \ell} \big( G^{-2} \oc{h}_{ij} \big)\big( G^{-2} \oc{h}_{k\ell} \big) \\
&= | \oc{h} |_{g}^2.
\end{split}
\end{align}
We note the relation between $r = |x|$ and $\rho = |y|$:
\begin{align*}
r \sim \rho^{-1},
\end{align*}
so that (\ref{2decay}) and (\ref{thsize}) together imply
\begin{align} \label{2vanish}
| \h |_{g_0} = O( \rho^2),
\end{align}
as $\rho \rightarrow 0$. In particular, $\h \in C^{1,\alpha}(Y)$.

We now use the standard splitting of a trace-free symmetric tensor into the image of the conformal Killing operator and the space of
transverse-traceless tensors.  More precisely, we first solve
\begin{align} \label{Boxsplit}
\Box_{g_0} \omega_0 = \delta_{g_0} \h
\end{align}
with $\omega_0 \in C^{2,\alpha}(Y)$.
Since $\Box$ is self-adjoint with kernel given by the space of conformal Killing forms $\mathcal{C}(Y,g_0)$, this equation is solvable
whenever the right-hand side is orthogonal to $\mathcal{C}(Y,g_0)$.  However, if $\eta \in \mathcal{C}(Y,g_0)$, then
\begin{align*}
\int \langle \delta_{g_0} \h, \eta \rangle\ dV_0 &= -\frac{1}{2} \int \langle \h, \mathcal{K}_{g_0} \eta \rangle\ dV_0 = 0.
\end{align*}
It follows that (\ref{Boxsplit}) is always solvable, although the solution $\omega_0$ is only unique up to the space of conformal Killing fields. This fact will actually be crucial when
we impose toric and diagonal invariance, in which case we will need to solve (\ref{Boxsplit}) equivariantly and study the space of invariant forms (see the end of this section).

Let
\begin{align} \label{hzeddef}
h_0 = \h  - \mathcal{K}_{g_0} \omega_0.
\end{align}
Then $h_0 \in C^{1,\alpha}(Y)$, and is smooth away from $y_0$. By (\ref{Boxsplit}), $h_0$ is transverse-traceless, and on $Y \setminus \{ y_0 \}$
\begin{align*}
B_{g_0}'(h_0 ) &= B_{g_0}'( \h  - \mathcal{K}_{g_0} \omega_0 ) \\
&= B_{g_0}'(\h) \ \ \mbox{(since Im $\mathcal{K} \subset $Ker$ B'$)} \\
&= 0.
\end{align*}
A standard integration by parts argument shows that $h_0$ is a global weak solution of $B_{g_0}'h_0 = 0$, and from elliptic theory it follows that
$h_0$ is smooth on $Y$.  Since $(Y,g_0)$ is assumed to be infinitesimally 
Bach-rigid, $h_0 = 0$, and we conclude that
\begin{align}
\h = \mathcal{K}_{g_0} \omega_0.
\end{align}
By conformal invariance of the conformal Killing operator,
\begin{align}
\mathcal{K}_{g_0} \omega_0 = G^{-2} \mathcal{K}_g [ G^2 \omega_0].
\end{align}
Hence,
\begin{align*}
G^{-2} \oc{h} = \h = G^{-2} \mathcal{K}_g [ G^2 \omega_0 ],
\end{align*}
which implies
\begin{align} \label{upstairszed}
\oc{h} = \mathcal{K}_g [ G^2 \omega_0].
\end{align}
Also, by (\ref{doc}), $\omega_1 = G^2 \omega_0$ satisfies
\begin{align}
0 = \delta \oc{h} = \Box \omega_1,
\end{align}
which gives the first equation in (\ref{o1facts}). To prove the second equation, use the splitting $h = \oc{h} + fg$
in (\ref{Rpzed}); this gives
\begin{align}
- 3 \Delta f - \langle Ric, \oc{h} \rangle = 0
\end{align}
(note we have used the scalar-flat condition again).

Up to this point we have not used the invariance of $h$.  In general, the form $\omega_1$ can grow quadratically on $N$; however, using
invariance we can choose a solution $\omega_0$ of (\ref{Boxsplit}) so that the resulting form $\omega_1$ has linear growth on $N$,
with highest order given by (\ref{asymo}).
To see this, we argue as follows. Since $\omega_0 \in C^{2, \alpha}(Y)$, it
admits an expansion
\begin{align}
\omega_0 = \omega^{(0)} + \omega^{(1)} + \omega^{(2)} + O(\rho^{2 + \alpha}),
\end{align}
as $\rho \rightarrow 0$, where
\begin{align}
\omega^{(0)} &= \omega^{(0)}_{i} dy^i\\
\omega^{(1)} &= \omega^{(1)}_{ij}y^i dy^j\\
\omega^{(2)} &= \omega^{(2)}_{ijk}y^iy^j dy^k,
\end{align}
where the $\{y\}$-coordinates are local normal coordinates near $y_0$
with torus action
\begin{align}
\label{gpact}
(y_1, y_2, y_3, y_4)
\mapsto \big( e^{  \I \theta_1} (y_1 + \I y_2), e^{ \I \theta_2}  (y_3 + \I y_4)\big).
\end{align}
Denote $\rho_1^2 = y_1^2 + y_2^2$, $\rho_2^2 = y_3^2 + y_4^2$, and
$\rho^2 = \rho_1^2 + \rho_2^2$,
and let $\theta_1$, $\theta_2$ denote the corresponding angular coordinates.
Since the group is compact, we can average over the group to
find a solution of \eqref{Boxsplit} which is also invariant under the
group action \eqref{gpact}, as well as the diagonal symmetry.
It is elementary to see that there is no form with
constant coefficients which is invariant under the torus action (\ref{gpact}).
The only toric-invariant $1$-forms with linear coefficients are
\begin{align}
c_1 d\rho_1 + c_2 \rho_1 d \theta_1 + c_3 d\rho_2 + c_4 \rho_2 d \theta_2.
\end{align}
The forms $\rho_1 d \theta_1$ and $\rho_2 d \theta_2$ extend to global
Killing forms, so we may assume that $c_2 = c_4 = 0$.
Invariance under the diagonal symmetry implies that $c_1 = c_3$,
so we have that
\begin{align}
\omega_0 =  c \cdot \rho d \rho + \omega^{(2)} + O(\rho^{2 + \alpha}),
\end{align}
for some constant $c \in \RR$. This implies the expansion
\begin{align}
\omega_1 = c \cdot x^i dx^i + O(|x|^{-\alpha}),
\end{align}
as $|x| \rightarrow \infty$. Averaging over the group, we may assume 
$\omega_1$ is also invariant under the group action. 

To obtain the expansion for $f$, extend the function $|x|^{-2}$ to
all of $N$ by a cutoff function (which we supress).
It is not hard to see that $\Delta(|x|^{-2}) = O(|x|^{-6})$ as $|x| \rightarrow \infty$,
and there exists a constant $c' \in \RR$ so that
\begin{align}
\int_N \Big( -\frac{1}{3} \langle Ric, \mathcal{K} \omega_1 \rangle - c' \Delta |x|^{-2}
\Big) dV =0.
\end{align}
Next, consider $\Delta:  C^{2,\alpha}_{-4 + \epsilon}(N) \rightarrow  C^{0,\alpha}_{-6 + \epsilon}(N)$.
The adjoint weight is $2 - \epsilon$, so from toric invariance, the kernel of the adjoint contains only constants.
We may then solve the equation
\begin{align}
\Delta \tilde{f} &= -\frac{1}{3} \langle Ric, \mathcal{K} \omega_1 \rangle
 - c' \Delta ( |x|^{-2})
\end{align}
with $\tilde{f} \in  C^{2,\alpha}_{-4 + \epsilon}$.
Equivalently,
\begin{align}
\Delta ( \tilde{f} + c' |x|^{-2}) =  -\frac{1}{3} \langle Ric, \mathcal{K} \omega_1 \rangle.
\end{align}
Since there are no decaying harmonic functions, we must have
\begin{align} \label{fdecay}
f =  \tilde{f} + c' |x|^{-2}
\end{align}
with $\tilde{f} = O(|x|^{-4 + \epsilon})$ as $|x| \rightarrow \infty$
for any $\epsilon > 0$.
Again, averaging over the group, we may assume that $f$ is invariant under
the group action. 

 Finally, we consider the case that $\delta > 0$. Using the
same argument as in the proof of Theorem \ref{comker} involving
the relative index theorem, the
toric and diagonal symmetries imply that the only
possible leading terms are $c \cdot g_N$ and a log-type solution.
Since $c \cdot g_N$ extends to a global solution, again the
relative index theorem implies that the log-type
solution does not occur. Consequently, after subtracting
a multiple of the metric, the solution is decaying, and \eqref{fexp}
follows from the previous expansion.

\section{Asymptotics of the cokernel}

Denote the (normalized) cokernel element described in Theorem \ref{afker} by
\begin{align} \label{odef}
o_1 = \kappa + fg,
\end{align}
where
\begin{align}
\kappa &= \mathcal{K}[\omega_1],
\end{align}
with
\begin{align}
(\omega_1)_i &= x^i + O(1).
\end{align}
 This section will be devoted to proving the following
\begin{theorem}
\label{ALEcokethm}
The tracefree part of AF-cokernel element $o_1$ satisfies
\begin{align} \label{kay}
\kappa_{ij} = \frac{2}{3} W_{ikj\ell}(y_0) \frac{x^k x^{\ell}}{|x|^4} + O(|x|^{-4 + \epsilon})
\end{align}
as $|x| \rightarrow \infty$, for any $\epsilon > 0$.
\end{theorem}
Recall in inverted normal coordinates at the point $y_0$, the AF metric
has the expansion
\begin{align} \label{gform}
g_{ij} = \delta_{ij} - \frac{1}{3}R_{ikj \ell}(y_0)\frac{x^k x^{\ell}}{|x|^4} + \frac{2A}{|x|^2} \delta_{ij} + O(|x|^{-3}).
\end{align}
In the following, we will need to have expansions for the Christoffel symbols:
\begin{lemma} \label{CSLemma} In inverted normal coordinates,
\begin{align} \label{Gamma1}  \begin{split}
\Gamma_{ij}^k =& -\frac{1}{3} R_{i \alpha k j}(y_0) \frac{x^{\alpha}}{|x|^4} - \frac{1}{3} R_{j \alpha k i}(y_0) \frac{x^{\alpha}}{|x|^4} \\
& \hskip.2in  + \frac{2}{3} R_{i \alpha k \beta }(y_0) \frac{x^j x^{\alpha} x^{\beta}}{|x|^6} + \frac{2}{3} R_{j \alpha k \beta }(y_0) \frac{x^i x^{\alpha} x^{\beta}}{|x|^6}
- \frac{2}{3} R_{i \alpha j \beta }(y_0) \frac{x^k x^{\alpha} x^{\beta}}{|x|^6} \\
& \hskip.2in - 2A \frac{x^i \delta_{jk}}{|x|^4} - 2A \frac{x^j \delta_{ik}}{|x|^4} + 2A \frac{x^k \delta_{ij}}{|x|^4} + O(|x|^{-4}).
\end{split}
\end{align}
\end{lemma}

\begin{proof} Recall
\begin{align*}
\Gamma_{ij}^k = \frac{1}{2} g^{km} \big( \partial_i g_{jm} + \partial_j g_{im} - \partial_m g_{ij} \big).
\end{align*}
By (\ref{gform}),
\begin{align*}
\partial_i g_{jm} &= \partial_i \big\{ -\frac{1}{3} R_{j \alpha m \beta}(y_0) \frac{x^{\alpha} x^{\beta}}{|x|^4} + \frac{2A}{|x|^2} \delta_{jm} \big\} + \cdots \\
&= -\frac{1}{3} R_{j i m \beta}(y_0) \frac{x^{\beta}}{|x|^4} - \frac{1}{3} R_{j \alpha m i}(y_0) \frac{x^{\alpha}}{|x|^4} + \frac{4}{3} R_{j \alpha m \beta}(y_0) \frac{x^i x^{\alpha} x^{\beta}}{|x|^6} \\
& \hskip.2in - 4A \frac{x^i \delta_{jm}}{|x|^4}   + \cdots
\end{align*}
Therefore, after combining terms and rearranging,
\begin{align*}
\partial_i g_{jm} + \partial_j g_{im} &- \partial_m g_{ij} = -\frac{2}{3} R_{j \alpha m i}(y_0) \frac{x^{\alpha}}{|x|^4} - \frac{2}{3} R_{i \alpha m j}(y_0)\frac{x^{\alpha}}{|x|^4} \\
& + \frac{4}{3}R_{j \alpha m \beta}(y_0) \frac{x^i x^{\alpha} x^{\beta}}{|x|^6} + \frac{4}{3} R_{i \alpha m \beta}(y_0) \frac{x^j x^{\alpha} x^{\beta}}{|x|^6} - \frac{4}{3} R_{j \alpha i \beta}(y_0) \frac{x^m x^{\alpha} x^{\beta}}{|x|^6} \\
& - 4A \frac{x^i \delta_{jm}}{|x|^4} - 4A \frac{x^j \delta_{im}}{|x|^4} + 4 A \frac{ x^m \delta_{ij}}{|x|^4} + \cdots
\end{align*}
It follows from (\ref{gform}) that the inverse matrix $g^{km}$ is given by
\begin{align} \label{ginv}
g^{km} = \delta_{km} + \frac{1}{3} R_{k \alpha m \beta}(y_0) \frac{x^{\alpha} x^{\beta}}{|x|^4} - \frac{2A}{|x|^2}\delta_{km}  + \cdots
\end{align}
Consequently,
\begin{align} \label{Gamma0} 
\begin{split}
\Gamma_{ij}^k 
&=  - \frac{1}{3} R_{j \alpha k i}(y_0) \frac{x^{\alpha}}{|x|^4} -\frac{1}{3} R_{i \alpha k j}(y_0) \frac{x^{\alpha}}{|x|^4} \\
& \hskip.2in  + \frac{2}{3} R_{j \alpha k \beta }(y_0) \frac{x^j x^{\alpha} x^{\beta}}{|x|^6}  + \frac{2}{3} R_{i \alpha k \beta }(y_0) \frac{x^j x^{\alpha} x^{\beta}}{|x|^6}
- \frac{2}{3} R_{j \alpha i \beta }(y_0) \frac{x^k x^{\alpha} x^{\beta}}{|x|^6} \\
& \hskip.2in - 2A \frac{x^i \delta_{jk}}{|x|^4} - 2A \frac{x^j \delta_{ik}}{|x|^4} + 2 A \frac{ x^k \delta_{ij}}{|x|^4}+ O(|x|^{-4}),
\end{split}
\end{align}
which is the same as (\ref{Gamma1}) after rearranging and re-indexing.
\end{proof}
Next, we consider the form $\omega = \omega_j dx^j$ with
\begin{align} \label{wform}
\omega_j = x^j + b_{jk}\frac{x^k}{|x|^2},
\end{align}
where
\begin{align} \label{bhform0}
b_{ij} &= -\frac{1}{3} S_{ij}(y_0) + 2A \delta_{ij},
\end{align}
where
\begin{align}
S_{ij}(y_0) =   \frac{1}{2}\Big( R_{ij}(y_0) - \frac{1}{6}R(y_0) \delta_{ij} \Big)
\end{align}
is the Schouten tensor.
We extend $\omega$ to be a globally defined form on all of $N$
by cutting it off at some finite distance from the basepoint.
Since this cutoff will not matter in the following, we will suppress it
from the following computations.
\begin{lemma}  \label{Lemma1} In inverted normal coordinates,
\begin{align} \label{Kform} \begin{split}
\mathcal{K}[\omega] &=  2 \frac{b_{ij}}{|x|^2} - 2\frac{b_{jk}x^k x^i}{|x|^4} - 2 \frac{b_{ik}x^k x^j}{|x|^4} - \frac{1}{2}\frac{b_{kk}}{|x|^4}\delta_{ij} + \frac{b_{k\ell}x^k x^{\ell}}{|x|^4}\delta_{ij} \\
&\hskip.2in  + \frac{2}{3}R_{ikj \ell}(y_0)\frac{x^k x^{\ell}}{|x|^4}
 - \frac{1}{6}R_{k\ell}(y_0)\frac{x^k x^{\ell}}{|x|^4}\delta_{ij} \\
& \hskip.2in + \frac{2A}{|x|^4} \big[ 4 x^i x^j - |x|^2 \delta_{ij} \big] + O(|x|^{-4}),
\end{split}
\end{align}
as $|x| \rightarrow \infty$.
\end{lemma}

\begin{proof}  We begin by noting
\begin{align*}
\nabla_i \omega_j = \partial_i \omega_j - \Gamma_{ij}^k \omega_k.
\end{align*}
Using (\ref{wform}),
\begin{align} \label{pomega}
\partial_i \omega_j = \delta_{ij} + \frac{b_{ji}}{|x|^2} - 2 \frac{b_{jk}x^k x^i}{|x|^4} + \cdots,
\end{align}
while
\begin{align} \label{Gamomega} 
\begin{split}
\Gamma_{ij}^k \omega_k 
&= -\frac{1}{3} R_{i \alpha k j}(y_0) \frac{x^{\alpha} x^k }{|x|^4} - \frac{1}{3} R_{j \alpha k i}(y_0) \frac{x^{\alpha} x^k }{|x|^4} \\
& \hskip.2in  + \frac{2}{3} R_{i \alpha k \beta }(y_0) \frac{x^j x^k x^{\alpha} x^{\beta}}{|x|^4} + \frac{2}{3} R_{j \alpha k \beta }(y_0) \frac{x^j x^k x^{\alpha} x^{\beta}}{|x|^4}
- \frac{2}{3} R_{i \alpha j \beta }(y_0) \frac{ x^{\alpha} x^{\beta}}{|x|^2} \\
& \hskip.2in - 4A \frac{x^i x^j}{|x|^4} + 2A \frac{\delta_{ij}}{|x|^2} + O(|x|^{-3}).
\end{split}
\end{align}
By the symmetries of the curvature tensor, the third and fourth terms above obviously vanish.  If we re-index in the first two terms, $k \leftrightarrow \beta$, then we can rewrite them
as
\begin{align} \label{reind} \begin{split}
-\frac{1}{3} R_{i \alpha k j}(y_0) \frac{x^{\alpha} x^k }{|x|^4} - \frac{1}{3} R_{j \alpha k i}(y_0) \frac{x^{\alpha} x^k }{|x|^4} &=
-\frac{1}{3} R_{i \alpha \beta j}(y_0) \frac{x^{\alpha} x^{\beta} }{|x|^4} - \frac{1}{3} R_{j \alpha \beta i}(y_0) \frac{x^{\alpha} x^{\beta} }{|x|^4} \\
&= \frac{1}{3} R_{i \alpha j \beta }(y_0) \frac{x^{\alpha} x^{\beta} }{|x|^4} + \frac{1}{3} R_{i \beta j \alpha}(y_0) \frac{x^{\alpha} x^{\beta} }{|x|^4} \\
&= \frac{2}{3} R_{i \alpha j \beta }(y_0) \frac{x^{\alpha} x^{\beta} }{|x|^4}.
\end{split}
\end{align}
Substituting this back into (\ref{Gamomega}), we find that
\begin{align*}
\Gamma_{ij}^k \omega_k =  - 4A \frac{x^i x^j}{|x|^4} + 2A \frac{\delta_{ij}}{|x|^2} + O(|x|^{-3}).
\end{align*}
Therefore,
\begin{align} \label{Domega}
\nabla_i \omega_j = \delta_{ij} + \frac{b_{ji}}{|x|^2} - 2 \frac{b_{jk}x^k x^i}{|x|^4} + 4A \frac{x^i x^j}{|x|^4} - 2A \frac{\delta_{ij}}{|x|^2} + O(|x|^{-3}).
\end{align}
The divergence of $\omega$ is
\begin{align} \label{divo} \begin{split}
\delta \omega &= g^{ij} \nabla_i \omega_j \\
&= \big\{ \delta_{ij} + \frac{1}{3} R_{i \alpha j \beta}(y_0)\frac{x^{\alpha} x^{\beta} }{|x|^4} - \frac{2A}{|x|^2} \delta_{ij}  + \cdots \big\}\\
 & \hskip.25in \times \big\{  \delta_{ij} + \frac{b_{ji}}{|x|^2} - 2 \frac{b_{jk}x^k x^i}{|x|^4} - 2 \frac{b_{jk}x^k x^i}{|x|^4} + 4A \frac{x^i x^j}{|x|^4} - 2A \frac{\delta_{ij}}{|x|^2} + \cdots \big\} \\
&= 4 + \frac{b_{kk}}{|x|^2} - 2 \frac{b_{k \ell}x^k x^{\ell}}{|x|^4} + \frac{1}{3} R_{k \ell}(y_0) \frac{x^{k} x^{\ell}}{|x|^4} - \frac{12A}{|x|^2} + \cdots.
\end{split}
\end{align}
Hence,
\begin{align} \label{divtimesg} \begin{split}
(\delta \omega)g_{ij} &= \big\{ 4 + \frac{b_{kk}}{|x|^2} - 2 \frac{b_{k \ell}x^k x^{\ell}}{|x|^4} + \frac{1}{3} R_{k \ell}(y_0) \frac{x^{k} x^{\ell}}{|x|^4} - \frac{12A}{|x|^2} + \cdots \big\} \\
& \hskip.25in \times \big\{ \delta_{ij} - \frac{1}{3}R_{ikj \ell}(y_0)\frac{x^k x^{\ell}}{|x|^4} + \frac{2A}{|x|^2} \delta_{ij} +  \cdots \big\} \\
&= 4 \delta_{ij} + \frac{b_{kk}}{|x|^2}\delta_{ij} - 2 \frac{ b_{k \ell} x^k x^{\ell}}{|x|^4}\delta_{ij} - \frac{4}{3} R_{i k j \ell}(y_0) \frac{x^k x^{\ell}}{|x|^4}
+ \frac{1}{3} R_{k\ell}(y_0)  \frac{x^k x^{\ell}}{|x|^4} \delta_{ij} \\
& \hskip.3in - \frac{4A}{|x|^2} \delta_{ij} + \cdots
\end{split}
\end{align}

Finally, combining (\ref{Domega}) and (\ref{divtimesg}) we get
\begin{align} \label{Ko1} \begin{split}
\mathcal{K}[\omega]_{ij} &= \nabla_i \omega_j + \nabla_j \omega_i - \frac{1}{2} (\delta \omega)g_{ij} \\
&= 2 \frac{b_{ij}}{|x|^2} - 2\frac{b_{jk}x^k x^i}{|x|^4} - 2 \frac{b_{ik}x^k x^j}{|x|^4} - \frac{1}{2}\frac{b_{kk}}{|x|^4}\delta_{ij} + \frac{b_{k\ell}x^k x^{\ell}}{|x|^4}\delta_{ij} \\
&\hskip.2in  + \frac{2}{3}R_{ikj \ell}(y_0)\frac{x^k x^{\ell}}{|x|^4}
 - \frac{1}{6}R_{k\ell}(y_0)\frac{x^k x^{\ell}}{|x|^4}\delta_{ij} + 8A \frac{x^i x^j}{|x|^4} - 2A \frac{\delta_{ij}}{|x|^2} + \cdots,
\end{split}
\end{align}
which completes the proof.
\end{proof}
This implies the following decay rate for $\Box \omega$:
\begin{lemma}  \label{Lemma2}   In inverted normal coordinates,
\begin{align} \label{Boxform} \begin{split}
\Box \omega = O (|x|^{-4}),
\end{split}
\end{align}
as $|x| \rightarrow \infty$.
\end{lemma}

\begin{proof}  Recall that $\Box$ is given by
\begin{align*}
\Box \omega_j = \big( \delta \mathcal{K}[\omega]\big)_j
= g^{i k} \nabla_k \mathcal{K}[\omega]_{i j} = g^{i k} \partial_k \mathcal{K}[\omega]_{ij} + \mathcal{K}[\omega] \ast \Gamma.
\end{align*}
Note that
\begin{align*}
(\mathcal{K}[\omega]) \ast \Gamma \sim |x|^{-5},
\end{align*}
so it is much lower order than the derivative term.  Also,
\begin{align*}
g^{i k} \partial_k \mathcal{K}[\omega]_{ij} &= \big( \delta_{i k} + O(|x|^{-2}) \big) \partial_k \mathcal{K}_{ij}[\omega] \\
&= \partial_i \mathcal{K}[\omega]_{ij} + (\mbox{lower}).
\end{align*}
Consequently, we obtain
\begin{align} \label{Boxformlong} \begin{split}
(\Box \omega)_j
&= -2 \frac{b_{ij}x^i}{|x|^4} - 4 \frac{b_{jk}x^k}{|x|^4} - \frac{b_{kk}}{|x|^4}x^j + 4 \frac{b_{k\ell} x^k x^{\ell}x^j}{|x|^6} \\
&\hskip.2in + \frac{2}{3}R_{k\ell}(y_0)\frac{x^k x^{\ell} x^j}{|x|^6} - R_{jk}(y_0)\frac{x^k}{|x|^4} +12A \frac{x^j}{|x|^4} + O(|x|^{-4})
\end{split}
\end{align}
as $|x| \rightarrow \infty$.
Substituting \eqref{bhform0} into this completes the proof.
\end{proof}
Consider $\Box: C^{2, \alpha}_{-2 + \epsilon}(N) \rightarrow C^{0, \alpha}_{-4 + \epsilon}(N)$,
with $\epsilon >0$ small, and consider the equation
\begin{align}
\label{boeq}
\Box ( \omega') = \Box \omega
\end{align}
The cokernel of this operator has domain weight $- \epsilon$, so from
Proposition \ref{boxker}, there is no cokernel. Consequently,
\eqref{boeq} has a solution $\omega' \in  C^{2, \alpha}_{-2 + \epsilon}$.
The form $ \tilde{\omega} = \omega - \omega'$ is then a solution
of $\Box \tilde{\omega} = 0$ with expansion
\begin{align}
\tilde{\omega}_j = x^j + b_{jk}\frac{x^k}{|x|^2} + O(|x|^{-2 + \epsilon}),
\end{align}
for any $\epsilon > 0$. Since $\tilde{\omega}$ and $\omega_1$
have the same leading term, and their difference is decaying,
we must have $\tilde{\omega} = \omega_1$, so
of course $\omega_1$ admits the same expansion.
\begin{proof}[Proof of Theorem \ref{ALEcokethm}] Substituting (\ref{bhform0}) into (\ref{Kform}) and using the decomposition of the curvature tensor into Weyl and Schouten parts gives
\begin{align*}
\mathcal{K}[\omega]_{ij} &= 2 \big\{ -\frac{1}{3}S_{ij}(y_0) + 2A \delta_{ij} \big\}\frac{1}{|x|^2} - 2\big\{ -\frac{1}{3}S_{jk}(y_0) + 2 A \delta_{jk} \big\}\frac{x^i x^k}{|x|^4} \\
& - 2\big\{ -\frac{1}{3}S_{ik}(y_0) + 2 A \delta_{ik} \big\}\frac{x^j x^k}{|x|^4} - \frac{1}{2} \big\{ -\frac{1}{18}R(y_0) + 8A \big\}\frac{\delta_{ij}}{|x|^2} \\
& + \big\{ -\frac{1}{3}S_{k\ell}(y_0) + 2A \delta_{k \ell} \big\}\frac{x^k x^{\ell}}{|x|^4}\delta_{ij} + \frac{2}{3} \big\{ W_{ik j \ell}(y_0) + ( \delta_{ij} S_{k\ell}(y_0)
- \delta_{i \ell} S_{jk}(y_0) \\
& - \delta_{jk} S_{i \ell}(y_0) + \delta_{k\ell} S_{ij}(y_0) ) \big\} \frac{x^k x^{\ell}}{|x|^4} + 8A \frac{x^i x^j}{|x|^4}- \frac{1}{6}R_{k\ell}(y_0)\frac{x^k x^{\ell}}{|x|^4}\delta_{ij}  - 2A \frac{\delta_{ij}}{|x|^2} + \cdots \\
&= \frac{2}{3} W_{ikj \ell}(y_0) \frac{x^k x^{\ell}}{|x|^4} + \cdots
\end{align*}
\end{proof}

\section{Some auxiliary linear equations}
\label{auxiliary}

In this section, we solve two linear equations.
First, an equation on the AF metric $(N, g_N)$, and
second, an equation on the compact manifold $(Z, g_Z)$.
The ``group action'' will refer to the $U(2)$-action in
the cases $g_Z$ is the Fubini-Study
metric and $g_N$ is the Burns metric,
and to the toric action plus diagonal symmetry in
the case $g_Z$ is $g_{S^2 \times S^2}$ and $g_N$ is
the corresponding Green's function metric.
\subsection{A linear equation on $(N,g_N)$}
On the compact manifold $(Z, g_Z)$, in normal
coordinates $\{z^i\}$ around $z_0$, we have the expansion
\begin{align}
g_Z = (g_Z)_{ij} dz^i dz^j = (\delta_{ij} + H_2(z)_{ij} + O(|z|^4)_{ij}) dz^i dz^j,
\end{align}
where
\begin{align}
H_2(z)_{ij} =  - \frac{1}{3} R_{ikjl}(z_0) z^k z^l.
\end{align}
Again let $(N,g_N)$ be the conformal blow-up of the Bach-flat
manifold $(Y, g_Y)$, as above.
Consider the quadratic tensor
\begin{align}
H_2(x) = (- \frac{1}{3} R_{ikjl}(z_0) x^k x^l) dx^i dx^j.
\end{align}
This tensor $H_2(x)$ of course does not live on all of $N$, since it is
only defined in the AF coordinate system. To extend
$H_2(x)$ to all of $N$, let $0 \leq \phi \leq 1$ be a cut-off function satisfying
\begin{align} \label{cutdef}
\phi(t) =
\begin{cases}
1 & t \leq 1 \\
0 & t \geq 2,\\
\end{cases}
\end{align}
and consider $(1 -\phi(R_0^{-1} x)) H_2(x)$,
where $R_0$ is very large.
\begin{proposition} \label{SH2decayProp}
Let $S$ denote the linearized operator on $N$, then
\begin{align}
S ( (1 -\phi(R_0^{-1} x) H_2(x)) = O( |x|^{-4})
\end{align}
as $|x| \rightarrow \infty$
\end{proposition}
\begin{proof}
From \eqref{PS1}, the linearized operator has the general form
\begin{align}
\begin{split} \label{Sgenform}
S h &= ( g^{-2} + g* g^{-3} ) *\nabla^4 h + g* g^{-3} *Rm * \nabla^2 h
+ g* g^{-3}* \nabla Rm * \nabla h\\
&+ ( g^{-2} + g * g^{-3})*( \nabla^2 Rm + Rm * Rm) *h.
\end{split}
\end{align}
It is easy to see that for $|x|$ sufficiently large and any tensor $h$,
\begin{align}
\begin{split}
\nabla^4 h &= \partial^4 h + \Gamma * \partial^3 h + ( \partial \Gamma + \Gamma * \Gamma) * \partial^2 h\\
& + ( \partial^2 \Gamma + \Gamma * \partial \Gamma) * \partial h
+ ( \partial^3 \Gamma + \partial \Gamma * \partial \Gamma
+ \Gamma * \partial^2 \Gamma) * h,
\end{split}
\end{align}
where $\partial$ denotes coordinate partial derivatives.
If $h$ grows quadratically, then since $g_N$ is AF of order $2$,
we see that
\begin{align}
\nabla^4 h = \partial^4 h + O (|x|^{-4}).
\end{align}
Since $(g_N)_{ij} = \delta_{ij} + O(|x|^{-2})$, it follows that
\begin{align*}
( g^{-2} + g* g^{-3} ) *\nabla^4 h = S_0 h + O (|x|^{-4}),
\end{align*}
where $S_0$ is the linearized operator with respect to the flat metric.
Estimating the other terms on the right-hand side of (\ref{Sgenform}) in a
similar manner, we find
\begin{align}
S(h) = S_{0} h + O(|x|^{-4})
\end{align}
as $|x| \rightarrow \infty$.
Since $H_2$ has quadratic leading term and $S_0$ is a fourth-order
operator, we clearly have
\begin{align} \label{SH20}
S_0 (H_2) = 0.
\end{align}
Therefore,
\begin{align}
S(H_2) = O(|x|^{-4})
\end{align}
as $|x| \rightarrow \infty$.
\end{proof}
Next, given $\epsilon > 0$, consider
\begin{align}
\label{smap}
S: C_{\epsilon}^{4,\alpha}(N) \rightarrow C^{0,\alpha}_{\epsilon - 4}(N).
\end{align}
The cokernel of this mapping is the kernel of
\begin{align}
S^*: C_{- \epsilon}^{4,\alpha}(N) \rightarrow C_{-\epsilon - 4}^{0,\alpha}(N),
\end{align}
which consists of the decaying elements.

By Theorem \ref{afker}, $Ker(S^*)$ is $1$-dimensional, and spanned by the element
\begin{align}
o_1 = \mathcal{K} \omega_1 + f \cdot g_N.
\end{align}
Since $Ker(S^*)$ is nontrivial,
this means the map in \eqref{smap} is not surjective, that is,
$S( C_{\epsilon}^{4,\alpha} ) \subset C_{\epsilon - 4}^{0,\alpha}$ is a proper subset,
and the quotient space
\begin{align}
C_{\epsilon - 4}^{0,\alpha} / S( C_{\epsilon}^{4,\alpha} )
\end{align}
is $1$-dimensional.
A tensor $h \in C_{\epsilon - 4}^{0,\alpha}$ is in the image of $C_{\epsilon}^{4, \alpha}$ under
$S$ if and only if it pairs trivially with $Ker(S^*)$ under the $L^2$ pairing.
That is
\begin{align}
h \in S(C_{\epsilon}^{4,\alpha}) \Longleftrightarrow \int_N \langle h, o_1 \rangle dV = 0.
\end{align}
Since the quotient space is $1$-dimensional, we
choose $k_1^{(0)} \in C_{\epsilon - 4}^{0,\alpha}$ having compact support in $B(x_0,R_0)$ 
(where $x_0$ is a basepoint) satisfying
\begin{align}
\int \langle o_1, k_1^{(0)} \rangle dV = 1,
\end{align}
and we can write
\begin{align}
C_{\epsilon-4}^{0,\alpha} = S(  C_{\epsilon}^{4,\alpha}) \oplus \mathbb{R} \cdot k_1^{(0)}.
\end{align}
By averaging over the group, we may assume that $k_1^{(0)}$ is invariant
under the group action. Therefore, we can write
\begin{align}
S (  (1 -\phi(R_0^{-1} x)) H_2(x)) =  S ( h_{\epsilon}) + \lambda k_1^{(0)},
\end{align}
where $ h_{\epsilon} \in  C_{\epsilon}^{4,\alpha}$, and $\lambda \in \mathbb{R}$.
Again, by averaging over the group, we may assume that $h_{\epsilon}$ is
invariant under the group action. Rewriting this as
\begin{align}
S (  (1 -\phi(R_0^{-1} x)) H_2 - h_{\epsilon} ) = \lambda k_1^{(0)},
\end{align}
we now define
\begin{align}
\tilde{H}_2  \equiv  (1 -\phi(R_0^{-1} x)) H_2- h_{\epsilon}.
\end{align}
Since  $h_{\epsilon} \in  C_{\epsilon}^{4,\alpha}$, clearly
$\tilde{H}_2$ has leading term {\em{exactly}} equal to $H_2$ as $|x| \rightarrow \infty$.
To summarize, we have solved

\begin{proposition} \label{LamProp}
On $(N,g_N)$, there exists a solution of
\begin{align} \label{lamDefEqn} \begin{split}
S (\tilde{H}_2) &=  \lambda k_1^{(0)}, \\
\tilde{H}_2(x) &= H_2(x) + O^{(4)}(|x|^{\epsilon}), \mbox{ as } |x| \rightarrow \infty,
\end{split}
\end{align}
where $k_1^{(0)}$ is a tensor with compact support on $(N, g_N)$ satisfying
\begin{align}
\int_N \langle o_1, k_1^{(0)} \rangle\ dV = 1.
\end{align}
Furthermore, $\tilde{H}_2$ can be chosen to be invariant under the group action.
\end{proposition}

\subsection{A linear equation on $(Z,g_Z)$}
Next we return to the compact metric $(Z, g_Z)$.
Recall on $(N,g_N)$, we have an AF-coordinate system
satisfying
\begin{align} \label{gNext}
g_N = (g_N)_{ij}dx^i dx^j = ( \delta_{ij} + H_{-2}(x)_{ij} + O(|x|^{-4 + \epsilon})_{ij}) dx^i dx^j,
\end{align}
where
\begin{align}
H_{-2}(x) = \Big(- \frac{1}{3} R_{ikjl}(y_0) \frac{x^k x^l}{|x|^4}
+ 2 A \frac{1}{|x|^2} \delta_{ij} \Big) dx^i dx^j
\end{align}
is a $2$-tensor with components
\begin{align} \label{7Hminus2form}
H_{-2}(x)_{ij} = - \frac{1}{3} R_{ikjl}(y_0) \frac{x^k x^l}{|x|^4}
+ 2 A \frac{1}{|x|^2} \delta_{ij}.
\end{align}
Consider the inverse quadratic tensor
\begin{align}
H_{-2}(z)= \Big(- \frac{1}{3} R_{ikjl}(y_0) \frac{z^k z^l}{|z|^4}
+ 2 A \frac{1}{|z|^2} \delta_{ij} \Big) dz^i dz^j.
\end{align}
Extend this tensor to all of $Z$ by $ \phi( (R')^{-1}z) H_{-2}(z)$,
where $b < R' < inj_{z_0} (g_Z)$ is some fixed radius.

We will need the following technical lemma both in this Section, and later in Section~\ref{better}:

\begin{lemma} \label{SBKer}  Let $S_0$ denote the linearized operator with respect to the flat metric. Then
\begin{align} \label{S0H2}
S_0 (H_{-2}) = 0,
\end{align}
where $H_{-2}$ is viewed as a tensor on $\mathbb{R}^4 \setminus \{0\}$.

Furthermore, if $(B_0^t)'$ denotes the linearization of the $B^t$-tensor at the flat metric, then
\begin{align} \label{B0H2}
(B_0^t)'(H_{-2}) = 0.
\end{align}
\end{lemma}

\begin{proof}  To prove the Lemma we use the expansion (\ref{gNext}):
\begin{align}
g_N = g_0 + H_{-2} + O(|x|^{-4 + \epsilon})
\end{align}
where $g_0$ is the flat metric.  Let $\theta = g_N - g_0$.
Since $g_N$ is $B^t$-flat,
\begin{align} \label{exP1} \begin{split}
P_{g_0}( \theta) &= B^t(g_N) + \mathcal{K}_{g_N} \delta_0 \mathcal{K}_0 \delta_0 \overset{\circ}{\theta} \\
&= \mathcal{K}_{g_N} \delta_0 \mathcal{K}_0 \delta_0 \overset{\circ}{\theta}.
\end{split}
\end{align}
We can also use the expansion of $P$ at the flat metric $g_0$ to write
\begin{align} \label{exP2} \begin{split}
P_{g_0}( \theta) &= P_{g_0}(0) + S_0 (\theta) + Q(\theta) \\
&= B^t(g_0) + S_0 (\theta) + Q(\theta) \\
&= S_0 (\theta) + Q(\theta).
\end{split}
\end{align}
Combining (\ref{exP1}) and (\ref{exP2}) we find
\begin{align}
S_0 \theta = \mathcal{K}_{g_N} \delta_0 \mathcal{K}_0 \delta_0 \overset{\circ}{\theta} - Q(\theta).
\end{align}
Since
\begin{align} \label{thform}
\theta = H_{-2} + O(|x|^{-4 + \epsilon})
\end{align}
and $S_0$ is fourth order,
\begin{align}
S_0 \theta = S_0 (H_{-2}) + O(|x|^{-8 +\epsilon}),
\end{align}
hence
\begin{align}
S_0 (H_{-2}) = \mathcal{K}_{g_N} \delta_0 \mathcal{K}_0 \delta_0 \overset{\circ}{\theta} - Q(\theta) + O(|x|^{-8+ \epsilon}).
\end{align}
Also, using (\ref{Qsize}) we have
\begin{align}
 Q(\theta) = O(|x|^{-8}),
 \end{align}
so that
\begin{align} \label{SH28}
S_0 (H_{-2}) = \mathcal{K}_{g_N} \delta_0 \mathcal{K}_0 \delta_0 \overset{\circ}{\theta} + O(|x|^{-8 + \epsilon}).
\end{align}

It remains to estimate the gauge-fixing operator acting on $\theta$. By (\ref{7Hminus2form}),
\begin{align}
\overset{\circ}{\theta}_{ij} = - \frac{1}{3} W_{ikjl}(y_0) \frac{x^k x^l}{|x|^4} + \frac{R(y_0)}{36} \Big\{ \frac{x^i x^j }{|x|^4} - \frac{\delta_{ij}}{|x|^2} \Big\} + O(|x|^{-4}).
\end{align}
 Using the skew-symmetry of the Weyl tensor we find
 \begin{align}
 (\delta_0 \overset{\circ}{\theta})_j = \frac{R(y_0)}{24}\frac{x^j}{|x|^4} + O(|x|^{-5}).
 \end{align}
 We next calculate
 \begin{align*}
 \delta_0 \mathcal{K}_0 \delta_0 \overset{\circ}{\theta} = \Box \delta_0 \overset{\circ}{\theta}.
 \end{align*}
 It is easy to check that the form
\begin{align}
\omega = \frac{x^j}{|x|^4}dx^j
\end{align}
is harmonic.  Therefore, using the formula (\ref{boxEuc}) for $\Box$ on Euclidean space,
 \begin{align*}
\Box \omega &= \frac{3}{2} d \delta \omega + \delta d \omega = 0.
\end{align*}
Consequently,
 \begin{align} \label{gftH}
 \delta_0 \mathcal{K}_0 \delta_0 \overset{\circ}{\theta} &= \Box \delta_0 \overset{\circ}{\theta}= \Box \Big(  \frac{R(y_0)}{24} \omega + O(|x|^{-5}) \Big)=  O(|x|^{-7}).
 \end{align}
 It follows that
\begin{align}
\mathcal{K}_{g_N} \delta_0 \mathcal{K}_0 \delta_0 \overset{\circ}{\theta}   = O(|x|^{-8}),
\end{align}
which, using (\ref{SH28}),  implies
\begin{align} \label{Szed8}
 S_0 ( H_{-2}) = O(|x|^{-8 + \epsilon}).
\end{align}
However, since $H_{-2}$ is homogeneous of degree $-2$,  $S_0(H_{-2})$ must be
homogeneous of degree $-6$.  Therefore, (\ref{Szed8}) implies that $S_0(H_{-2})$ vanishes.

A similar argument (expanding the $B^t$-tensor as in Proposition \ref{QRemark}) gives (\ref{B0H2}).
\end{proof}

\begin{proposition}  \label{SZProp}
Let $S$ denote the linearized operator on $Z$, then
\begin{align}
S ( \phi((R')^{-1} z) H_{-2}(z)) = O( |z|^{-4})
\end{align}
as $|z| \rightarrow 0$.
\end{proposition}
\begin{proof}
As above, for $|z|$ sufficiently small and any tensor $h$,
\begin{align}
\begin{split}
\nabla^4 h &= \partial^4 h + \Gamma * \partial^3 h + ( \partial \Gamma + \Gamma * \Gamma) * \partial^2 h\\
& + ( \partial^2 \Gamma + \Gamma * \partial \Gamma) * \partial h
+ ( \partial^3 \Gamma + \partial \Gamma * \partial \Gamma
+ \Gamma * \partial^2 \Gamma) * h,
\end{split}
\end{align}
where $\partial$ denotes coordinates partial derivatives.
If $h$ blows-up inverse quadratically, then since $\{z^i\}$ are
Riemannian normal coordinates,
we see that
\begin{align}
\nabla^4 h = \partial^4  h + O (|z|^{-4}).
\end{align}
Arguing as we did in the proof of Proposition \ref{SH2decayProp}, we find that
\begin{align} \label{SvSflat}
S(h) = S_{0} h + O(|z|^{-4}),
\end{align}
where $S_0$ is the linearized operator with respect to the flat metric.  If we take $h = H_{-2}$ in (\ref{SvSflat}),
then (\ref{S0H2}) of Lemma \ref{SBKer} gives
\begin{align}
S(H_{-2}) = O(|z|^{-4})
\end{align}
as $|z| \rightarrow 0$, and the Proposition follows.
\end{proof}

 Next, for $\epsilon > 0$, we have
\begin{align}
S : C_{-\epsilon}^{4,\alpha}(Z) \rightarrow C_{- \epsilon - 4}^{0,\alpha}(Z),
\end{align}
with adjoint mapping
\begin{align}
S^* : C_{\epsilon}^{4,\alpha}(Z) \rightarrow C_{\epsilon - 4}^{0,\alpha}(Z).
\end{align}
By Theorem \ref{comker}, there is no (invariant) cokernel.
Thus there exists $h_{- \epsilon} \in C_{- \epsilon}^{4,\alpha} $ such that
\begin{align}
S (\phi( (R')^{-1} z) H_{-2}(z)) = S( h_{\epsilon}),
\end{align}
or rather
\begin{align}
S (  \phi( (R')^{-1} z) H_{-2}(z) - h_{\epsilon}) = 0.
\end{align}
Averaging over the group, we may assume that $h_{\epsilon}$
is invariant under the group action. We then define
\begin{align}
\tilde{H}_{-2}(z) =  \phi( (R')^{-1} z) H_{-2}(z) - h_{\epsilon}.
\end{align}
To summarize, we have proved
\begin{proposition} \label{Hminus2}
On $(Z, g_Z)$, there exists a solution $\tilde{H}_{-2}$ of
\begin{align}
S (\tilde{H}_{-2}(z) ) &= 0\\
\tilde{H}_{-2}(z) &= H_{-2}(z) + O(|z|^{-\epsilon}), \mbox{ as } |z| \rightarrow 0.
\end{align}
Furthermore, $\tilde{H}_{-2}$ can be chosen to be invariant under the group action.
\end{proposition}

\begin{remark} \label{epvsdel} {\em
From now on, we will fix $\epsilon > 0$ small.
}
\end{remark}

\section{Computation of the leading term}

 In this section we compute the constant $\lambda$ which
arose above in Proposition \ref{LamProp}.
As the title of this section indicates, we will refer to this
constant as ``the leading term'' for reasons which will become
clear later in Section \ref{Kuranishi}.

Recall from Proposition \ref{LamProp} that $\lambda$ was defined via equation (\ref{lamDefEqn}):
\begin{align} \label{SH2lam}
S (\tilde{H}_2) =  \lambda k_1^{(0)} \ \ \mbox{on $N$},
\end{align}
with
\begin{align}
\label{H2Hdiff}
\tilde{H}_2(x) &= H_2(x) + O^{(4)}(|x|^{\epsilon})
\end{align}
as $|x| \rightarrow \infty$, and
\begin{align} \label{H2def}
(H_2)_{ij} = -\frac{1}{3}R_{ikj\ell}(z_0)x^k x^{\ell}.
\end{align}
Pairing both side of the defining equation for $\lambda$ with the cokernel element $o_1$ and integrating gives
\begin{align} \label{laminit}
\lambda = \int_{N} S (\tilde{H}_2), o_1 \rangle\ dV,
\end{align}
since
\begin{align*}
\int \langle k_1^{(0)}, o_1 \rangle\ dV = 1.
\end{align*}

\begin{proposition} \label{SadjProp}
The constant $\lambda$ is given by
\begin{align} \label{Sadj}
\lambda  &=  \frac{4}{9} \omega_3 \big[ W_{ikj\ell}(y_0) W_{ikj\ell}(z_0) + W_{ikj\ell}(y_0) W_{i\ell j k}(z_0) \big]
+ 4t \omega_3 R(z_0) \mathrm{mass}(g_N),
\end{align}
where $\omega_3 = Vol(S^3)$.
\end{proposition}

We prove this formula through a series of lemmas.   To begin, let 
\begin{align}
\label{balldef}
B = \{ x \in N \ : \ |x| < a^{-1} \}, 
\end{align}
(where we extend $|x|$ to be defined on all of $N$ by letting it be a constant 
outside of the AF region of $N$), and use (\ref{Sform}) to write
\begin{align} \label{3terms} \begin{split}
\int_{B} \langle S \tilde{H}_2 , o_1 \rangle &=  \\
& \hskip-.5in \int_{B} \langle \Delta^2 (\overset{\circ}{\tilde{H}_2}), o_1 \rangle + \int_{B} \langle \mathcal{K}[ d( \mathcal{D}_2( \tilde{H}_2))], o_1 \rangle
+ \frac{3}{2}t \int_{B} \langle \big[ \Delta^2 (tr\ \tilde{H}_2) - \Delta (\delta^2 \tilde{H}_2) \big]g, o_1 \rangle,
\end{split}
\end{align}
where $\overset{\circ}{T}$ denotes the trace-free part of the symmetric two-tensor $T$.

\begin{lemma} \label{T1Lemma} As $a \rightarrow 0$,
\begin{align} \label{Term1}  \begin{split}
\int_{B} \langle \Delta^2 \overset{\circ}{\tilde{H}_2}, o_1 \rangle &= \int_{B} \langle \tilde{H}_2, \Delta^2 \kappa \rangle \\
& \hskip.5in + \frac{4}{9} \omega_3 \big[ W_{ikj\ell}(y_0) W_{ikj\ell}(z_0) + W_{ikj\ell}(y_0) W_{i\ell j k}(z_0) \big] + o(1).
\end{split}
\end{align}
\end{lemma}

\begin{proof} Since $\kappa$ is the trace-free part of $o_1$,
\begin{align*}
\int_{B} \langle \Delta^2 \overset{\circ}{\tilde{H}_2}, o_1 \rangle = \int_{B} \langle \Delta^2 \overset{\circ}{\tilde{H}_2} , \kappa \rangle.
\end{align*}
Integrating by parts,
\begin{align} \label{IBP1} \begin{split}
\int_{B} \langle \Delta^2 \overset{\circ}{\tilde{H}_2} , \kappa \rangle &= \int_{B} \langle \overset{\circ}{\tilde{H}_2}, \Delta^2 \kappa \rangle + \bint \langle \nabla_N (\Delta \overset{\circ}{\tilde{H}_2}), \kappa \rangle - \bint \langle \Delta \overset{\circ}{\tilde{H}_2}, \nabla_N \kappa \rangle \\
& \hskip.25in  + \bint \langle \nabla_N \overset{\circ}{\tilde{H}_2}, \Delta \kappa \rangle - \bint \langle \overset{\circ}{\tilde{H}_2}, \nabla_N(\Delta \kappa) \rangle,
\end{split}
\end{align}
where $N$ is the outward unit normal to $N$. All the boundary integrals in (\ref{IBP1}) are with respect to the approximate metric $g$.  To estimate each boundary term we use
the fact that on $\partial B$, the metric and Christoffel symbols satisfy
\begin{align} \label{ALEg} \begin{split}
g &= \delta + O(a^2), \\
\Gamma &= O(a^3), \\
\partial \Gamma &= O(a^4),
\end{split}
\end{align}
where $\delta$ denotes the flat metric.  For a symmetric $2$-tensor $T = T_{ij}$,
\begin{align} \label{LapT}
\Delta T = g^{\alpha \beta} \nabla_{\alpha} \nabla_{\beta} T,
\end{align}
and
\begin{align} \label{HessT}
\nabla_{\alpha} \nabla_{\beta} T = \partial_{\alpha} \partial_{\beta} T + \Gamma * \partial T + \partial \Gamma * T + \Gamma * \Gamma * T,
\end{align}
hence
\begin{align} \label{LapT2}
\Delta T_{ij} = \Delta_0 T_{ij} + O(a^2)* \partial^2 T + O(a^3)* \partial T + O(a^4)* T,
\end{align}
where $\Delta_0$ denotes the flat Laplacian.

Taking $T = \overset{\circ}{\tilde{H}_2}$ and using (\ref{H2Hdiff}) we first note
\begin{align} \label{TFH2} \begin{split}
(\overset{\circ}{\tilde{H}_2})_{ij} &= (\tilde{H}_2)_{ij} - \frac{1}{4} \big[ g^{\alpha \beta} (\tilde{H}_2)_{\alpha \beta} \big]g_{ij} \\
&= -\frac{1}{3} R_{ikj\ell}(z_0) x^k x^{\ell} + \frac{1}{12}R_{k\ell}(z_0) x^k x^{\ell} \delta_{ij} + O(|x|^{\epsilon}).
\end{split}
\end{align}
Therefore,
\begin{align} \label{HessDiff2} \begin{split}
\partial_{\alpha} (\overset{\circ}{\tilde{H}_2})_{ij} &=  -\frac{1}{3} R_{i \alpha j \ell}(z_0) x^{\ell} - \frac{1}{3} R_{ikj \alpha}(z_0) x^k + \frac{1}{6} R_{\alpha k}(z_0) x^k \delta_{ij} + O(|x|^{\epsilon-1}), \\
\partial_{\alpha} \partial_{\beta} (\overset{\circ}{\tilde{H}_2})_{ij} &= -\frac{1}{3}R_{i \alpha j \beta}(z_0) - \frac{1}{3} R_{i \beta j \alpha}(z_0) + \frac{1}{6} R_{\alpha \beta}(z_0) \delta_{ij} + O(|x|^{\epsilon-2}),
\end{split}
\end{align}
hence
\begin{align} \label{LapDiff1}
(\Delta \overset{\circ}{\tilde{H}_2})_{ij} = -\frac{2}{3}\big[ R_{ij}(z_0)- \frac{1}{4}R(z_0)\delta_{ij} \big] + O(|x|^{\epsilon-2}).
\end{align}
Assuming $(Z,g_{Z})$ is Einstein, then
\begin{align*}
 R_{ij}(z_0)- \frac{1}{4}R(z_0)\delta_{ij} = 0.
\end{align*}
It follows that
\begin{align} \label{LapH20} \begin{split}
(\Delta \overset{\circ}{\tilde{H}_2})_{ij} &= O(|x|^{\epsilon-2}), \\
\nabla_N (\Delta \overset{\circ}{\tilde{H}_2})_{ij} &= O(|x|^{\epsilon-3}),
\end{split}
\end{align}
as $|x| \rightarrow \infty$.

By (\ref{kay}), on $\partial B$ we have
\begin{align} \label{ksize} \begin{split}
|\kappa| &= O(a^2), \\
|\nabla \kappa| &= O(a^3).
\end{split}
\end{align}
Using these estimates along with those of (\ref{LapH20}) we find
\begin{align} \label{bdyT1} \begin{split}
\big| \bint \langle \nabla_N (\Delta \overset{\circ}{\tilde{H}_2}), \kappa \rangle \big| &= O(a^{2-\epsilon}), \\
\big| \bint \langle \Delta \overset{\circ}{\tilde{H}_2} , \nabla_N \kappa \rangle \big| &= O(a^{2-\epsilon}).
\end{split}
\end{align}

Next, we take $T = \kappa_{ij}$.  Using (\ref{kay}),
\begin{align} \label{dkay} \begin{split}
\partial_{\alpha} \kappa_{ij} &= \frac{2}{3}W_{i \alpha j \ell}(y_0)\frac{x^{\ell}}{|x|^4} + \frac{2}{3}W_{ikj\alpha}(y_0)\frac{x^k}{|x|^4} - \frac{8}{3}W_{ikj\ell}(y_0)\frac{x^k x^{\ell} x^{\alpha}}{|x|^6} + O(a^4), \\
\partial_{\alpha} \partial_{\beta} \kappa_{ij} &= \frac{2}{3}W_{i\alpha j \beta}(y_0)\frac{1}{|x|^4} + \frac{2}{3}W_{i\beta j \alpha}(y_0)\frac{1}{|x|^4}
 + 16 W_{ikj\ell}(y_0)\frac{x^k x^{\ell} x^{\alpha} x^{\beta}}{|x|^8} \\
 &\hskip.1in - \frac{8}{3}W_{i \alpha j \ell}(y_0)\frac{x^{\ell}x^{\beta}}{|x|^6} - \frac{8}{3}W_{i k j \alpha }(y_0)\frac{x^{k}x^{\beta}}{|x|^6} - \frac{8}{3}W_{i \beta j \ell}(y_0)\frac{x^{\ell}x^{\alpha}}{|x|^6} \\
 &\hskip.2in - \frac{8}{3}W_{i k j \beta}(y_0)\frac{x^k x^{\alpha}}{|x|^6} - \frac{8}{3}W_{ikj\ell}(y_0) \frac{x^k x^{\ell}}{|x|^6} \delta_{\alpha \beta} + O(a^5).
\end{split}
\end{align}
Therefore,
\begin{align} \label{Lap0k} \begin{split}
\Delta \kappa_{ij} &= - \frac{16}{3} W_{ikj\ell}(y_0) \frac{x^k x^{\ell}}{|x|^6} + O(a^5) \\
&= - \frac{8}{|x|^2} \kappa_{ij} + O(a^5).
\end{split}
\end{align}

On $\partial B$,
\begin{align} \label{Nform}
N^{\alpha} = \frac{x^{\alpha}}{|x|} + O(a^2),
\end{align}
hence
\begin{align} \label{DN1}
\nabla_N T_{ij} = \frac{x^{\alpha}}{|x|}\partial_{\alpha} T_{ij} + O(a^2)*\partial T + O(a^3)*T.
\end{align}
From (\ref{TFH2}),(\ref{dkay}), and (\ref{DN1}) (or, by reasons of homogeneity) we conclude
\begin{align} \label{ndiv} \begin{split}
\nabla_N (\overset{\circ}{\tilde{H}_2})_{ij} &= -\frac{2}{3}R_{ikj\ell}(z_0)\frac{x^k x^{\ell}}{|x|} + \frac{1}{6}R_{k \ell}(z_0)\frac{x^k x^{\ell}}{|x|}\delta_{ij} + O(a^{1-\epsilon}) \\
&= \frac{2}{|x|} (\overset{\circ}{\tilde{H}_2})_{ij} + O(a^{1-\epsilon}), \\
\nabla_N (\Delta \kappa)_{ij} &= \frac{64}{3} W_{ikj\ell}(y_0) \frac{x^k x^{\ell}}{|x|^7} + O(a^7) \\
&= \frac{32}{|x|^3} \kappa_{ij} + O(a^7).
\end{split}
\end{align}
It follows that
\begin{align} \label{last2}
\langle \nabla_N \overset{\circ}{\tilde{H}_2}, \Delta \kappa \rangle - \langle \overset{\circ}{\tilde{H}_2}, \nabla_N(\Delta \kappa) \rangle = -48  \frac{1}{|x|^3} \langle \overset{\circ}{\tilde{H}_2}, \kappa \rangle + O(a^4),
\end{align}
hence
\begin{align} \label{last2b}
\bint \langle \nabla_N \overset{\circ}{\tilde{H}_2}, \Delta \kappa \rangle - \langle \overset{\circ}{\tilde{H}_2}, \nabla_N(\Delta \kappa) \rangle = \frac{32}{3} \int_{|\xi| = 1} W_{ikj\ell}(y_0) R_{i \alpha j \beta}(z_0) \xi^k \xi^{\ell} \xi^{\alpha} \xi^{\beta}\ dS + O(a).
\end{align}
If we decompose the curvature tensor of $R_{i \alpha j \beta}(z_0)$ (again assuming $(Z,g_Z)$ is Einstein),
\begin{align*}
R_{i \alpha j \beta}(z_0) = W_{i \alpha j \beta}(z_0) + \frac{1}{12}R(z_0)\big(\delta_{ij} \delta_{\alpha \beta} - \delta_{i \beta}\delta_{j \alpha} \big).
\end{align*}
Therefore, the integrand in (\ref{last2b}) can be written
\begin{align*}
W_{ikj\ell}(y_0) R_{i \alpha j \beta}(z_0) = W_{ikj\ell}(y_0) W_{i \alpha j \beta}(z_0) - \frac{1}{12}R(z_0)W_{\beta k \alpha \ell}(y_0).
\end{align*}
hence
\begin{align} \label{last22} \begin{split}
\int_{|\xi| = 1} W_{ikj\ell}(y_0) R_{i \alpha j \beta}(z_0) \xi^k \xi^{\ell} \xi^{\alpha} \xi^{\beta}\ dS &= \int_{|\xi| = 1} W_{ikj\ell}(y_0) W_{i \alpha j \beta}(z_0) \xi^k \xi^{\ell} \xi^{\alpha} \xi^{\beta}\ dS \\
& \hskip.25in  - \frac{1}{12}R(z_0) \int_{|\xi| = 1} W_{\beta k \alpha \ell}(y_0) \xi^k \xi^{\ell} \xi^{\alpha} \xi^{\beta}\ dS.
\end{split}
\end{align}
The last integral vanishes by skew-symmetry of the Weyl tensor; therefore,
\begin{align} \label{last2W}
\bint \langle \nabla_N \overset{\circ}{\tilde{H}_2}, \Delta \kappa \rangle - \langle \overset{\circ}{\tilde{H}_2}, \nabla_N(\Delta \kappa) \rangle = \frac{32}{3} \int_{|\xi| = 1} W_{ikj\ell}(y_0) W_{i \alpha j \beta}(z_0) \xi^k \xi^{\ell} \xi^{\alpha} \xi^{\beta}\ dS + O(a).
\end{align}

We now use the identity (see \cite{Brendle2008})
\begin{align} \label{SH}
\int_{|\xi|=1} \xi^k \xi^{\ell} \xi^{\alpha} \xi^{\beta}\ dS = \frac{\omega_3}{24} \big( \delta_{k\ell} \delta_{\alpha \beta} + \delta_{k \alpha} \delta_{\beta \ell} + \delta_{k \beta} \delta_{\alpha \ell} \big).
\end{align}
Plugging this into (\ref{last2W}), we obtain
\begin{align} \label{preWeylterm}  \begin{split}
\bint \langle \nabla_N \overset{\circ}{\tilde{H}_2}, \Delta \kappa \rangle - \langle \overset{\circ}{\tilde{H}_2}, \nabla_N(\Delta \kappa) \rangle &= \frac{4}{9}\omega_3\big[  W_{ikj\ell}(y_0)W_{ikj\ell}(z_0) + W_{ikj\ell}(y_0)W_{i \ell j k}(z_0) \big] \\
&\hskip.25in + O(a),
\end{split}
\end{align}
which proves the Lemma.
\end{proof}

\begin{lemma} \label{T2Lemma} As $a \rightarrow 0$,
\begin{align} \label{2ndterm}
\int_{B} \langle \mathcal{K}[ d( \mathcal{D}_2( \tilde{H}_2)], o_1 \rangle = O(a^{2-\epsilon}).
\end{align}
\end{lemma}

\begin{proof}
Since $\mathcal{K}[\cdot]$ is trace-free, we can rewrite the integrand in (\ref{2ndterm}) as
\begin{align*}
\int_{B} \langle \mathcal{K}[ d( \mathcal{D}_2( \tilde{H}_2)], o_1 \rangle = \int_{B} \langle \mathcal{K}[ d( \mathcal{D}_2( \tilde{H}_2)], \kappa \rangle.
\end{align*}
Integrating by parts and using the fact that $\kappa$ is divergence-free, we get
\begin{align} \label{Kbox} \begin{split}
\int_{B} \langle \mathcal{K}[ d( \mathcal{D}_2( \tilde{H}_2)], \kappa \rangle &= - 2 \int_{B} \langle d( \mathcal{D}_2( \tilde{H}_2), \delta \kappa \rangle  \\
&\hskip.25in + 2 \bint \kappa( d( \mathcal{D}_2( \tilde{H}_2)), N) \\
&= 2 \bint \kappa( d( \mathcal{D}_2( \tilde{H}_2)), N).
\end{split}
\end{align}
Using (\ref{HessT}) and computing as we did in the proof of Lemma \ref{T1Lemma}, on $\partial B$ we find
\begin{align} \label{D2H21} \begin{split}
\delta^2 \tilde{H}_2 &= \frac{1}{3}R(z_0) + O(a^{2-\epsilon}), \\
\Delta (tr\ \tilde{H}_2) &= -\frac{2}{3}R(z_0) + O(a^{2-\epsilon}).
\end{split}
\end{align}
Therefore,
\begin{align*}
\mathcal{D}_2(\tilde{H}_2) &= ( t + \frac{5}{12}) R(z_0) + O(a^{2-\epsilon}), \\
d( \mathcal{D}_2( \tilde{H}_2)) &= O(a^{3-\epsilon}).
\end{align*}
Since $\kappa = O(a^2)$ on $\partial B$, we see that the boundary term in (\ref{Kbox}) is $O(a^{2-\epsilon})$, which proves the Lemma.
\end{proof}

\begin{lemma}  \label{T3Lemma}  As $a \rightarrow 0$,
\begin{align} \label{Straceterm} \begin{split}
\frac{3}{2}t \int_{B} \langle \big[ \Delta^2 (tr\ \tilde{H}_2) - \Delta (\delta^2 \tilde{H}_2) \big]g, o_1 \rangle &= \int_{B} \langle \tilde{H}_2, 6t \big[ (\Delta^2 f)g - \nabla^2 (\Delta f)\big] \rangle  \\
& \hskip.5in + 4 t \Big( 12 A - \frac{R(y_0)}{12} \Big) \omega_3 R(z_0) + o(1).
\end{split}
\end{align}
\end{lemma}

\begin{proof}
Since $o_1 = \kappa + fg$ with $\kappa$ trace-free, we have
\begin{align} \label{downtrace} \begin{split}
\frac{3}{2}t \int_{B} \langle \big[ \Delta^2 (tr\ \tilde{H}_2) - \Delta (\delta^2 \tilde{H}_2) \big]g, o_1 \rangle &= \frac{3}{2}t \int_{B} \langle \big[ \Delta^2 (tr\ \tilde{H}_2) - \Delta (\delta^2 \tilde{H}_2) \big]g, \kappa + f g \rangle \\
&= \frac{3}{2}t \int_{B} \langle \big[ \Delta^2 (tr\ \tilde{H}_2) - \Delta (\delta^2 \tilde{H}_2) \big]g, f g \rangle \\
&= 6t \int_{B}  \big[ \Delta^2 (tr\ \tilde{H}_2) - \Delta (\delta^2 \tilde{H}_2) \big] f.
\end{split}
\end{align}
Integrating by parts, we find
\begin{align} \label{IBP3} \begin{split}
\int_{B}  \big[ \Delta^2 (tr\ \tilde{H}_2) &- \Delta (\delta^2 \tilde{H}_2) \big] f = \int_{B} \langle \tilde{H}_2, (\Delta^2 f)g - \nabla^2 (\Delta f) \rangle \\
& + \bint f \DN \Delta (tr\ \tilde{H}_2) - \bint \Delta(tr\ \tilde{H}_2) \DN f + \bint \DN (tr\ \tilde{H}_2) \Delta f \\
& - \bint (tr\ \tilde{H}_2) \frac{\partial}{\partial N}(\Delta f)  + \bint \tilde{H}_2( N, \nabla (\Delta f)) - \bint (\Delta f) \langle \delta \tilde{H}_2 , N \rangle \\
& + \bint \DN f (\delta^2 \tilde{H}_2) - \bint f \DN(\delta^2 \tilde{H}_2) \\
&=  \int_{B} \langle \tilde{H}_2, (\Delta^2 f)g - \nabla^2 (\Delta f) \rangle + I_1 + \cdots + I_8.
\end{split}
\end{align}

By Theorem \ref{afker}, on $\partial B$
\begin{align} \label{DDffacts} \begin{split}
\Delta f &= -\frac{1}{3} \langle Ric, o_1 \rangle = O(a^6), \\
\DN (\Delta f) &= -\frac{1}{3} \langle \nabla_N Ric, o_1 \rangle - \frac{1}{3} \langle Ric, \nabla_N o_1 \rangle = O(a^7).
\end{split}
\end{align}
Also, from the preceding lemma (see (\ref{D2H21}))
\begin{align} \label{D2H2facts} \begin{split}
tr\ \tilde{H}_2 &= -\frac{1}{3}R_{k\ell}(z_0)x^k x^{\ell} + O(a^{-\epsilon}) = O(a^{-2}), \\
\DN(tr\ \tilde{H}_2) &= -\frac{2}{3}  R_{k\ell}(z_0)\frac{x^k x^{\ell}}{|x|} + O(a^{1-\epsilon}) = O(a^{-1}), \\
\delta^2 \tilde{H}_2 &= \frac{1}{3}R(z_0) + O(a^{2-\epsilon}) = O(1), \\
\Delta (tr\ \tilde{H}_2) &= -\frac{2}{3}R(z_0) + O(a^{2-\epsilon}) = O(1) \\
\DN (\delta^2 \tilde{H}_2) &= O(a^{3-\epsilon}), \\
\DN (\Delta (tr\ \tilde{H}_2)) &= O(a^{3-\epsilon}).
\end{split}
\end{align}
Therefore,
\begin{align} \label{Inot27}  \begin{split}
I_1 &= \bint f \DN \Delta (tr\ \tilde{H}_2) = O(a^{2-\epsilon}) \\
I_3 &= \bint \DN (tr\ \tilde{H}_2) \Delta f = O(a^2), \\
I_4 &= - \bint (tr\ \tilde{H}_2) \frac{\partial}{\partial N}(\Delta f) = O(a^2), \\
I_5 &=  \bint \tilde{H}_2( N, \nabla (\Delta f))= O(a^2), \\
I_6 &=  - \bint (\Delta f) \langle \delta \tilde{H}_2 , N \rangle = O(a^2), \\
I_8 &=  - \bint f \DN(\delta^2 \tilde{H}_2) = O(a^{2-\epsilon}).
\end{split}
\end{align}
Therefore, it remains to calculate $I_2$ and $I_7$.

First, using (\ref{D2H2facts}) we have
\begin{align} \begin{split}
I_2 + I_7 &= \bint \big[ \delta^2 \tilde{H}_2 - \Delta(tr\ \tilde{H}_2) \big] \DN f  \\
&= \bint \big[ R(z_0) + O(a^{2-\epsilon})\big] \DN f \\
&= R(z_0)  \bint \DN f + O(a^{2-\epsilon}) \ \ \mbox{ (by (\ref{asymo}))}\\
&= R(z_0)  \int_{B} \Delta f  + O(a^{2-\epsilon}) \ \ \mbox{(by the divergence theorem)} \\
&= -\frac{1}{3} R(z_0) \int_{B} \langle Ric, \kappa \rangle + O(a^{2-\epsilon}) \ \ \mbox{(by (\ref{o1facts})).}
\end{split}
\end{align}

Using the fact that $\kappa = \mathcal{K}[\omega_1]$, we can integrate by parts to obtain
\begin{align*}
\int_{B} \langle Ric, \kappa \rangle &= \int_{B} \langle Ric, \mathcal{K}[\omega_1] \rangle \\
&= 2 \int_{B} R_{ij} \nabla^i \omega_1^j  \\
&= - 2 \int_{B} \nabla^i R_{ij} \omega_1^j + 2 \bint R_{ij} N^i \omega_1^j
\end{align*}
Using the second Bianchi identity and the fact that the scalar curvature is zero, the solid integral above vanishes and we conclude
\begin{align} \label{RicInt}
I_2 + I_7 = -\frac{2}{3}R(z_0) \bint Ric(N, \omega_1)  + O(a^{2-\epsilon}).
\end{align}
By (\ref{asymo}) and (\ref{Nform}),
\begin{align} \label{RicNN}
Ric(N, \omega_1) = R_{ij} \frac{x^i x^j}{|x|} + O(|x|^{-5}).
\end{align}

\begin{proposition}
\label{Ricprop}
As $|x| \rightarrow \infty$,
\begin{align}
\begin{split}\label{Ricex}
R_{ij} &= - \frac{4}{3}W_{ikj\ell}(y_0) \frac{x^k x^{\ell}}{|x|^6} - \frac{1}{36}R(y_0) \frac{1}{|x|^4} \delta_{ij} + \frac{1}{9}R_(y_0)\frac{x^i x^j}{|x|^6} \\
& \ \ \ \ - \frac{16A}{|x|^6}x^i x^j + \frac{4A}{|x|^4} \delta_{ij} + O(|x|^{-5}).
\end{split}
\end{align}
\end{proposition}
\begin{proof}
This is proved in Appendix \ref{append}.
\end{proof}
Assuming the proposition, we see that
\begin{align*}
R_{ij}\frac{x^i x^j}{|x|} = \frac{1}{12} R(y_0) \frac{1}{|x|^3} -\frac{12A}{|x|^3} + O(|x|^{-4}).
\end{align*}
Therefore,
\begin{align*}
I_2 + I_7 &= -\frac{2}{3}R(z_0) \bint Ric(N, \omega_1)  + O(a^{2-\epsilon}) \\
&= -\frac{2}{3}R(z_0) \big[ \frac{1}{12}R(y_0) - 12A \big] \bint \frac{1}{|x|^3}  + O(a) \\
&= -\frac{2}{3} \omega_3 R(z_0) \big[ \frac{1}{12}R(y_0) -12A \big]  + O(a).
\end{align*}
Plugging this into (\ref{downtrace}) and (\ref{IBP3}), we arrive at (\ref{Straceterm}).
\end{proof}
Combining Lemmas \ref{T1Lemma}, \ref{T2Lemma}, and \ref{T3Lemma}, and using (\ref{3terms}), we have
\begin{align} \label{moddiv} \begin{split}
\int_{B} \langle S \tilde{H}_2, o_1 \rangle &= \int_{B} \langle \tilde{H}_2, \Delta^2 \kappa + 6t \big[ (\Delta^2 f)g - \nabla^2 (\Delta f)\big] \rangle   \\
& \hskip.25in + \frac{4}{9} \omega_3 \big[ W_{ikj\ell}(y_0) W_{ikj\ell}(z_0) + W_{ikj\ell}(y_0) W_{i\ell j k}(z_0) \big] \\
& \hskip.35in + 4t \omega_3 R(z_0) \big\{ 12 A - \frac{1}{12}R(y_0) \big\} + o(1).
\end{split}
\end{align}
By Proposition \ref{massprop}, the quantity is braces is exactly
the mass of the AF space. Proposition \ref{SadjProp} then follows from the next Lemma:

\begin{lemma} \label{Szed}  The cokernel element $o_1$ satisfies
\begin{align} \label{So1}
0 = S(o_1) = \Delta^2 \kappa + 6t \big[ (\Delta^2 f)g - \nabla^2 (\Delta f)\big] + O(|x|^{-8}).
\end{align}
\end{lemma}
\begin{proof} By the formula in (\ref{Sform}), we have
\begin{align}
S o_1 = \Delta^2 [o_1 - \frac{1}{4}(tr\ o_1)g] + \mathcal{K}\big[ d (\mathcal{D}_2(o_1)) \big] + \frac{3}{2}t \Big[ \Delta^2 (tr\ o_1) - \Delta (\delta^2 o_1) \Big]g + \cdots
\end{align}
Using the properties of $o_1$ in Theorem \ref{afker}, we have
\begin{align*}
o_1 - \frac{1}{4}(tr\ o_1)g &= \kappa, \\
tr\ o_1 &= 4f, \\
\delta^2 o_1 &= \Delta f, \\
\Delta (tr\ o_1) &= 4 \Delta f.
\end{align*}
It follows that
\begin{align*}
\mathcal{D}_2(o_1) &= \big( t + \frac{5}{6}\big) \delta^2 o_1 - \big( t + \frac{5}{24}\big) \Delta (tr\ o_1) \\
&= - 3 t \Delta f,
\end{align*}
and
\begin{align*}
\mathcal{K}\big[ d (\mathcal{D}_2(o_1)) \big] &= - 3t \mathcal{K}\big[ d(\Delta f) \big] \\
&= -6t \nabla^2 (\Delta f) + \frac{3}{2}t (\Delta^2 f)g
\end{align*}
Therefore,
\begin{align*}
S o_1 &= \Delta^2 \kappa + \mathcal{K}\big[ d (\mathcal{D}_2(o_1)) \big] + \frac{9}{2}t  (\Delta^2 f)g  + \cdots  \\
&= \Delta^2 \kappa -6t \nabla^2 (\Delta f) + \frac{3}{2}t (\Delta^2 f)g  + \frac{9}{2}t  (\Delta^2 f)g  + \cdots  \\
&= \Delta^2 \kappa + 6t \big[ (\Delta^2 f)g - \nabla^2 (\Delta f)\big] + \cdots,
\end{align*}
as claimed.
\end{proof}

\section{Na\"ive approximate metric}
\label{approx}

Let $(Z, g_Z)$ be a compact $B^t$-flat manifold.
In our application, $(Z,g_Z)$ will be taken to be
either $\CP^2$ with the Fubini-Study metric, or $S^2 \times S^2$
with the product metric, with the coordinate systems
described in Subsections \ref{fssub} and \ref{s2s2sub}.

We let $z_0$ denote the base point, which is $[1,0,0]$ in the
case of $\CP^2$, or $(n,n)$ in the case of $S^2 \times S^2$.
As seen above, we have a
Riemannian normal coordinate system $\{z^i\} , i = 1 \dots 4$, satisfying
\begin{align} \label{gZexp1}
g_Z = dz^2 + \eta_Z(z),
\end{align}
where $\eta_Z$ has the expansion $\eta_Z = (\eta_Z(z))_{ij} dz^i dz^j$ with
\begin{align}
\begin{split}
\label{fullne}
(\eta_Z(z))_{ij} &= - \frac{1}{3} R_{ikjl}(z_0) z^k z^l + O(|z|^4)
\end{split}
\end{align}
as $|z| \rightarrow 0$.

Furthermore, in the case of $\CP^2$ the metric is invariant under the
standard linear action of $U(2)$ in the $\{z\}$-coordinates,
and in the case of $S^2 \times S^2$
the metric is invariant under the standard diagonal torus action,
and also invariant under the diagonal symmetry, both in the $\{z\}$-coordinates.

Next, let $(N, g_N)$ be a $B^t$-flat AF space of order $2$.
In our application $(N, g_N)$ will be taken to the either the
Burns metric or Green's function metric of the product metric
with AF coordinate system as described in Subsections \ref{burnssub}
and \ref{greensub}. The Green's function here is with respect to
the basepoint which we will denote as $y_0$,  which is $[1,0,0]$ in the
case of $\CP^2$, or either point $(n,n)$ of $S^2 \times S^2$.

We denote the AF coordinates as $\{x^i\} , i = 1 \dots 4$, and write
\begin{align}
g_N  = dx^2 + \eta_N,
\end{align}
where the tensor $\eta_N$ admits the expansion
\begin{align}
\begin{split}
\label{fullneyn}
(\eta_N)_{ij}(x) &= - \frac{1}{3} R_{ikjl}(y_0) \frac{x^k x^l}{|x|^4} + 2 A \frac{1}{|x|^2} \delta_{ij} + O( |x|^{-4 + \epsilon})
\end{split}
\end{align}
as $|x| \rightarrow \infty$.

In the case of the Burns metric, the metric is invariant under the
standard linear action of $U(2)$ in the $\{x\}$-coordinates,
and in the case of the Green's function metric on $S^2 \times S^2$,
the metric is invariant under the standard diagonal torus action,
and also invariant under the diagonal symmetry, both in the $\{x\}$-coordinates.

Let $\phi$ be the cutoff function defined in (\ref{cutdef}):
\begin{align}
\phi(t) =
\begin{cases}
1 & t \leq 1 \\
0 & t \geq 2.\\
\end{cases}
\end{align}

For $b > 0$ denote the annulus
$A_Z(b,2b) = \{ b \leq |z| \leq 2 b \} \subset Z$,
and for $a > 0$ denote the annulus
$A_N(a^{-1},2a^{-1}) \equiv \{ a^{-1} \leq |x| \leq 2 a^{-1} \} \subset N$.
Let $\iota : A_N(a^{-1},2a^{-1}) \rightarrow A_Z(b,2b)$
denote the map $\iota(x) = ab x = z$.
Identify the annular region
$A_Z( b, 2b) \subset X$ with $A_N(a^{-1},2a^{-1}) \subset N$
using the map $\iota$ to define a new manifold $X_{a,b}$.
\begin{remark}{\em
With this choice of $\iota$, the manifold $X_{a,b}$ is diffeomorphic to
$X \# \overline{N_{c}}$, where ${N}_{c}$ is
the one-point compactification of $N$.
If we instead choose $\iota$ to be defined by, for example,
$\iota(x_1, x_2, x_3, x_4) = ab ( -x_1, x_2, x_3, x_4)$,
$X_{a,b}$ will be diffeomorphic to
$X \# {N}_{c}$, which can be different topologically.
}
\end{remark}

 In the case where $(Z, g_Z)$ is the Fubini-Study metric and $(N, g_N)$ is
the Burns metric, the $U(2)$ action extends to $X_{a,b}$, since the
actions agree in the coordinate systems. In all other cases,
the torus action as well as the diagonal symmetry
extend to actions on $X_{a,b}$. For convenience, we will now
refer to this action as ``the group action'', keeping in mind
that the group depends on the example.

We compute that
\begin{align}
\begin{split}
\iota^*( a^{-2} b^{-2} g_Z) &= a^{-2} b^{-2} \{ a^2 b^2 dx^2 + (\iota^*\eta_Z)(x) \}\\
&= dx^2 + \tilde{\eta}_Z(x),
\end{split}
\end{align}
where
\begin{align}
 \tilde{\eta}_Z(x) = a^{-2} b^{-2} (\iota^{*} \eta_Z)( x ).
\end{align}
Note that  $\tilde{\eta}_Z$ admits the expansion
$\tilde{\eta}_Z = (\tilde{\eta}_Z(x))_{ij} dx^i dx^j$ with
\begin{align}
\begin{split}
\label{fullne2}
(\tilde{\eta}_Z(x))_{ij} &= - \frac{1}{3} a^2 b^2 R_{ikjl}(z_0) x^k x^l + \mbox{ (higher) }
\end{split}
\end{align}
as $b \rightarrow 0$ and for $x \in A_N(a^{-1}, 2a^{-1})$.

Define a metric $g_{a,b}^{(0)}$ on $X_{a,b}$ by
\begin{align}
\label{idmt}
g_{a,b}^{(0)} =
\begin{cases}
a^{-2} b^{-2} g_Z & |z| \geq 2b \\
dx^2 + \phi( a|x|) \eta_N(x) + [ 1 - \phi(a|x|)] \tilde{\eta}_Z(x) & a^{-1} \leq |x| \leq 2 a^{-1}\\
g_N &   |x| \leq a^{-1}.\\
\end{cases}
\end{align}
The group action is linear in the $\{x\}$-coordinates, and is
contained in $SO(4)$. Since the cutoff function is radial,
it is clear that $g_{a.b}^{(0)}$ is invariant under the group action.

On the damage zone $A_N(a^{-1}, 2a^{-1})$, we will also write the metric as
\begin{align}
g_{a,b}^{(0)} = dx^2 + \eta_1 + \eta_2,
\end{align}
where
\begin{align}
\begin{split}
\eta_1(x) &=  \phi( a|x|) \eta_N(x)\\
\eta_2(x) &= [ 1 - \phi(a|x|)] \tilde{\eta}_Z(x).
\end{split}
\end{align}
Notice that after scaling and identifying, we have
\begin{align} \label{RmsizeD}
|Rm(\iota^*(a^{-2} b^{-2} g_Z^b))| =
\begin{cases}
a^2 b^2 (|R(g_Z)| \circ \iota) & |x| \geq 2a^{-1} \\
O(a^2 b^2) & a^{-1} \leq |x| \leq 2 a^{-1} \\
0 & |x| \leq   2a^{-1}.
\end{cases}
\end{align}
This implies that
\begin{align}
|B^t(\iota^*(a^{-2} b^{-2} g_Z^b))| =
\begin{cases}
0 & |x| \geq 2a^{-1} \\
O(a^4 b^2) &  a^{-1} \leq |x| \leq 2 a^{-1} \\
0 & |x| \leq   2a^{-1}.
\end{cases}
\end{align}

 Consequently,
\begin{align}
|B^t(g_{a,b}^{(0)})| =
\begin{cases}
0 & |x| \geq 2a^{-1} \\
O(a^4 b^2) + O( a^{6}) & a^{-1} \leq |x| \leq 2 a^{-1} \\
0 & |x| \leq   2a^{-1},
\end{cases}
\end{align}
which is proved by using the expansion
\begin{align}
B^t(g_{a,b}^{(0)}) =  B^t(g_0) + (B^t)'_{g_0} ( \eta_1 + \eta_2)
+ \mathcal{Q}( \eta_1 + \eta_2),
\end{align}
where $g_0 = dx^2$ in the damage zone; see Remark \ref{QRemark}.

This estimate will not suffice for our purposes, and in
Section \ref{better} we will construct a ``better'' approximate metric.

\subsection{Gluing with one basepoint}
To summarize, $(X_{a,b}, g_{a,b}^{(0)})$ is defined in the following cases:

\begin{itemize}
\item (i)
$\CP^2 \# \overline{\CP}^2$; the Fubini-Study metric with a Burns metric attached
at one fixed point. This case admits a $U(2)$-action.
\item(ii)
$S^2 \times S^2 \# \overline{\CP}^2 = \CP^2 \# 2  \overline{\CP}^2$;
the product metric on $S^2 \times S^2$ with a Burns metric attached at
one fixed point. Alternatively, we can view this as the Fubini-Study
metric on $\CP^2$, with a Green's function $S^2 \times S^2$ metric attached at
one fixed point. For this topology, we will therefore construct two
different critical metrics.
\item(iii)
$2 \# S^2 \times S^2$; the product metric on $S^2 \times S^2$ with
a Green's function $S^2 \times S^2$ metric attached at one fixed point.
\end{itemize}
All of these cases are invariant under the torus action,
and invariant under the diagonal symmetry.

As mentioned in the introduction,
the product metric on $S^2 \times S^2$ admits the Einstein
quotient $S^2 \times S^2 / \ZZ_2$,
where $\ZZ_2$ acts by the antipodal map on both factors,
and the quotient $\RP^2 \times \RP^2$.
The diagonal symmetry clearly extends to these metrics.
Using one of these metrics as the compact factor or
the Green's function metric of one of these as the AF space,
we obtain approximate metrics on the non-simply-connected topologies
listed in Table~\ref{table2}. Note that in this table, the first
special value of $t_0$ corresponds to the the first factor
being the compact factor, and the second factor being the AF
space, while the second value of $t_0$ corresponds to the reverse.
From this, the approximate metric is clear and we need not
detail every case here.

\subsection{Gluing with multiple bubbles}

 We first consider the case when $(Z,g_Z)$ is
$(S^2 \times S^2, g_{S^2 \times S^2})$.
We can glue on an AF space at both points $(n,n)$
and $(s, s)$, but we must take the same AF space for both points.
In this case, we impose an additional symmetry.
There is an orientation-preserving involution of $S^2 \times S^2$
consisting of the product of antipodal maps.
Since both AF spaces are the same, this involution obviously extends to
an involution of $X_{a,b}$ which is an isometry of $g_{a,b}^{(0)}$,
and which we will refer to as bilateral symmetry. As in the single
bubble case, the toric action extends to an isometry of the
approximate metric on the connect sum.
We then have the following cases with toric invariance,
diagonal symmetry, and bilateral symmetry:
\begin{itemize}
\item(iv)
$3 \# S^2 \times S^2$; the product metric on $S^2 \times S^2$ with
Green's function $S^2 \times S^2$ metrics attached at two fixed points.
\item (v) $S^2 \times S^2 \# 2 \overline{\CP}^2
= \CP^2 \# 3  \overline{\CP}^2$; the product metric on $S^2 \times S^2$ with Burns
metrics attached at two fixed points.
\end{itemize}

Next, we consider the case when $(Z,g_Z)$ is
$(\CP^2, g_{FS})$. Imposing trilateral symmetry (see Figure \ref{cp2fig}),
allows us to attach the same AF space at all $3$ fixed points.
We then have the following cases with toric invariance,
diagonal symmetry at each fixed point, and trilateral symmetry:
\begin{itemize}
\item (vi) $ \CP^2 \# 3  \overline{\CP}^2$; the Fubini-Study
metric with Burns metrics attached at all fixed points.
\item (vii) $ \CP^2 \# 3(S^2 \times S^2) = 4 \CP^2 \# 3 \overline{\CP}^2 $;
the Fubini-Study metric with Green's function $S^2 \times S^2$ metrics
attached at all fixed points.
\end{itemize}

Next, we return to the case that $(Z,g_Z)$ is
$(S^2 \times S^2, g_{S^2 \times S^2})$. Imposing quadrilateral symmetry
(see Figure~\ref{s2s2fig}), allows us to attach the same
AF space at all $4$ fixed points.
We then have the following cases with toric invariance,
diagonal symmetry at each fixed point, and quadrilateral symmetry:
\begin{itemize}
\item (viii) $S^2 \times S^2 \# 4 \overline{\CP}^2 =
 \CP^2 \# 5 \overline{\CP}^2$;
the product metric on $S^2 \times S^2$ with Burns metrics
attached at all fixed points.
\item (ix) $5 \# S^2 \times S^2$ viewed as the product metric on
$S^2 \times S^2$ with Greens function $S^2 \times S^2$ metrics
attached at all fixed points.
\end{itemize}

For multiple bubbles in the non-orientable case, see Appendix \ref{appb}.

\subsection{Weight function}
For the weighted norms, we define the weight
function on $X_{a,b}$ by
\begin{align} \label{wdef}
w =
\begin{cases}
a^{-1} b^{-1} &  |z| \geq 1 \\
a^{-1} b^{-1} |z| &  1 \geq |z| \geq 2b \\
|x| & 2 a^{-1} \geq |x| \geq  1 \\
1 & 1 \geq |x|,
\end{cases},
\end{align}
where for simplicity we have assumed that the $x$ and $z$ coordinates
contain the unit spheres.
We record the inequalities
\begin{align}
1 \leq w  \leq a^{-1} b^{-1}.
\end{align}

\section{Refined approximate metric}
\label{better}

\begin{remark} \label{epvsdel2} {\em
We will now choose $\delta < 0$ satisyfing
$-\epsilon < \delta < 0$, where $\epsilon$ was previously 
chosen (see Remark \ref{epvsdel}).
}
\end{remark}

As pointed out above,
the approximate metric defined in \eqref{idmt} is insufficient
for our purposes, and needs to be refined.
To define the new approximate metric, we replace
$g_N$ with $g_N^{(1)} = g_N + a^2 b^2 \tilde{H}_{2}(x)$, so that
\begin{align}
g_N^{(1)} = g_0 + \eta_N(x) +  a^2 b^2 \tilde{H}_{2}(x), \ \ |x| \leq a^{-1}.
\end{align}
By Proposition \ref{LamProp}, for $a,b$ sufficiently small
\begin{align*}
|a^2 b^2 \tilde{H}_{2}(x)| &\lesssim a^2 b^2 |x|^2 \\
&\lesssim b^2,
\end{align*}
so that $g_N^{(1)}$ is indeed a Riemannian metric when $|x| < a^{-1}$.

Next, replace the compact metric $ g_Z$ with ${g}_Z^{(1)} = g_Z + a^2 b^2 \tilde{H}_{-2}(z)$,
so that
\begin{align} \label{gZ1def}
g_Z^{(1)} = g_0 + \eta_Z(z) +  a^2 b^2 \tilde{H}_{-2}(z), \ \ |z| \geq b.
\end{align}
By Proposition \ref{Hminus2},
\begin{align*}
| a^2 b^2 \tilde{H}_{-2}(z)| &\lesssim a^2 b^2 b^{-2} \\
&\lesssim a^2,
\end{align*}
hence $g_Z^{(1)}$ is a Riemannian metric for $|z| \geq b$.

Using these metrics, we then define the refined approximate metric $g_{a,b}^{(1)}$  on $X_{a,b}$ by
\begin{align}
\label{idmtnew}
g_{a,b}^{(1)} =
\begin{cases}
a^{-2} b^{-2} (g_Z + a^2 b^2 \tilde{H}_{-2}(z))  & |z| \geq 2b, \\
g_N + a^2 b^2 \tilde{H}_2(x)  &   |x| \leq a^{-1},\\
\end{cases}
\end{align}
while in the damage zone  $a^{-1} \leq |x| \leq 2 a^{-1}$
the metric is given by
\begin{align} \label{g1DZ}
\begin{split}
g_{a,b}^{(1)} = dx^2 &+ \phi( a|x|) \big\{ \eta_N(x) + a^2b^2 \tilde{H}_2(x)\big\} \\
&+ [ 1 - \phi(a|x|)]
\iota^* \big\{ a^{-2} b^{-2} (\eta_Z(z) + a^2 b^2 \tilde{H}_{-2}(z)) \big\}.
\end{split}
\end{align}

\begin{remark}{\em
From Propositions \ref{LamProp} and \ref{Hminus2}, it is clear that $g_{a,b}^{(1)}$ is
invariant under the group action.
}
\end{remark}
\subsection{Damage zone estimate}
We compute that
\begin{align}
\begin{split}
\iota^* \big\{ a^{-2} b^{-2} (g_Z + a^2 b^2 \tilde{H}_{-2}(z)) \big\}
&= \iota^* \big\{ a^{-2} b^{-2} ( \delta_{ij} + (\eta_Z(z))_{ij} + a^2 b^2
\tilde{H}_{-2}(z)_{ij} ) dz^i dz^j \big\} \\
&= \big(  \eta_Z( ab x)_{ij} + a^2 b^2 \tilde{H}_{-2}(ab x)_{ij} \big) dx^i dx^j.
\end{split}
\end{align}
Consequently, in the damage zone, the metric is
\begin{align}
\begin{split}
g_{a,b}^{(1)} = dx^2 &+ \phi( a|x|) \big\{ \eta_N(x) + a^2b^2 \tilde{H}_2(x)\big\} \\
&+ [ 1 - \phi(a|x|)]
\big(  \eta_Z( ab x)_{ij} + a^2 b^2 \tilde{H}_{-2}(ab x)_{ij} \big) dx^i dx^j.
\end{split}
\end{align}
We next use the the expansions
\begin{align}
\label{e1}
 a^2b^2 \tilde{H}_2(x)_{ij} =
-a^2 b^2 \Big(\frac{1}{3} R_{ikjl}(z_0)x^k x^l + O^{(4)}(|x|^{\epsilon}) \Big),
\end{align}
\begin{align}
\eta_Z(abx)_{ij} = - a^2 b^2 \frac{1}{3} R_{ikjl}(z_0)x^k x^l
+ a^4 b^4 O^{(4)}(|x|^4),
\end{align}
\begin{align}
\begin{split}
a^2 b^2 \tilde{H}_{-2}(ab x)_{ij} &=
a^2 b^2 \Big( H_{-2}(ab x)_{ij} + (ab)^{- \epsilon} O( |x|^{-\epsilon}) \Big)\\
&= - \frac{1}{3} R_{ikjl}(y_0) \frac{x^k x^l}{|x|^4}
+ 2 A \frac{1}{|x|^2} \delta_{ij} + (ab)^{2 - \epsilon} O^{(4)}( |x|^{-\epsilon}),
\end{split}
\end{align}
and
\begin{align}
\label{e4}
\eta_N(x)_{ij} = - \frac{1}{3} R_{ikjl}(y_0) \frac{x^k x^l}{|x|^4}
+ 2 A \frac{1}{|x|^2} \delta_{ij} + O^{(4)}(|x|^{-4+\epsilon}).
\end{align}
Using \eqref{e1}-\eqref{e4}, we obtain in the damage zone:
\begin{align}
\label{dzest}
\begin{split}
(g_{a,b}^{(1)})_{ij} &=  \delta_{ij}
- a^2 b^2 \frac{1}{3} R_{ikjl}(z_0)x^k x^l
- \frac{1}{3} R_{ikjl}(y_0) \frac{x^k x^l}{|x|^4}
+ 2 A \frac{1}{|x|^2} \delta_{ij}\\
&+   a^2 b^2 O^{(4)}(|x|^{\epsilon}) +  a^4 b^4 O^{(4)}(|x|^4)
+(ab)^{2 - \epsilon} O^{(4)}( |x|^{-\epsilon})+ O^{(4)}(|x|^{-4+\epsilon}).
\end{split}
\end{align}

\begin{proposition} \label{NewBtsize}
The size of the $B^t$-tensor of the refined approximate metric in the
damage zone is given by
\begin{align}
\label{dsize}
|B^t(g_{a,b}^{(1)})|_{g_{a,b}^{(1)}} =  O( b^2 a^{6 - \epsilon}) + O ( a^4 b^4) + O( a^6 b^{2 - \epsilon})
+ O(a^{8-\epsilon}),
\end{align}
as $a, b \rightarrow 0$.
\end{proposition}
\begin{proof}
By (\ref{dzest}),
\begin{align*}
g_{a,b}^{(1)} = g_0 + a^2 b^2 H_2 + H_{-2} + \mathcal{E},
\end{align*}
where
\begin{align} \label{Edef}
 \mathcal{E} = a^2 b^2 O^{(4)}(|x|^{\epsilon}) +  a^4 b^4 O^{(4)}(|x|^4)
+(ab)^{2 - \epsilon} O^{(4)}( |x|^{-\epsilon})+ O^{(4)}(|x|^{-4+\epsilon}).
\end{align}
Let $\theta = g_{a,b}^{(1)} - g_0$; then using the expansion of the $B^t$-tensor in Proposition \ref{QRemark}
\begin{align} \label{P1} \begin{split}
B^t(g_{a,b}^{(1)} ) &= B^t(g_0) + (B^t_0)'(\theta) + \mathcal{Q}_0(\theta) \\
&= (B^t_0)'(\theta) + \mathcal{Q}_0(\theta).
\end{split}
\end{align}
By Lemma \ref{SBKer} and the fact that $(B^t_0)'$ is fourth order,
\begin{align*}
 (B^t_0)'(\theta) &= a^2 b^2 (B^t_0)' (H_2) + (B^t_0)'( H_{-2}) + (B^t_0)' (\mathcal{E}) \\
&= (B^t_0)'( \mathcal{E}),
\end{align*}
hence
\begin{align} \label{BexP}
B^t(g_{a,b}^{(1)} ) =  (B^t_0)'( \mathcal{E}) + \mathcal{Q}_0(\theta).
\end{align}
If we estimate the norm of $B^t(g_{a,b}^{(1)} )$ in the flat metric, then (\ref{dsize}) follows from the formula for $\mathcal{E}$
and the result of Proposition \ref{QRemark}.  However, by (\ref{dzest}) it is clear that the same estimate holds if we use the norm with respect to
$g_{a,b}^{(1)}$, since for any symmetric $(0,2)$-tensor $T = T_{ij}$
\begin{align} \label{norms} \begin{split}
| T |_{g_{a,b}^{(1)}}^2 &= \{ g_{a,b}^{(1)} \}^{ik} \{ g_{a,b}^{(1)} \}^{j \ell} T_{ij} T_{k \ell} \\
&=  \{ \delta_{ik} + O(a^2 + b^2) \} \{ \delta_{j \ell} + O(a^2 + b^2) \}^{j \ell} T_{ij} T_{k \ell} \\
& = \big( 1 + O(a^2 + b^2) \big)| T |_0^2.
\end{split}
\end{align}

\end{proof}

Next, on the asymptotically flat piece we have
\begin{proposition} \label{BtsizeN}
On $\{ |x| < a^{-1} \} \subset N$, the $B^t$-tensor satisfies
\begin{align}
\label{alesize}
B^t (g_{a,b}^{(1)}) = a^2 b^2 \lambda k_1^{(0)} - a^2 b^2 \mathcal{K}_{g_N} \delta_{g_N} \mathcal{K}_{g_N} \delta_{g_N} \overset{\circ}{\tilde{H}_2} + O(a^4b^4),
\end{align}
and
\begin{align} \label{alemodK}
B^t (g_{a,b}^{(1)}) =  a^2 b^2 \lambda k_1^{(0)} + a^2 b^2 O(|x|^{\epsilon - 4}) + O(a^4 b^4),
\end{align}
as $a,b \rightarrow 0$ and $|x| \rightarrow \infty$. 
\end{proposition}

\begin{proof} The proof proceeds along the same lines as the proof of Proposition \ref{NewBtsize}.  On $\{ |x| < a^{-1} \} \subset N$ we have
\begin{align} \label{gvsN}
g_{a,b}^{(1)} = g_N + a^2 b^2 \tilde{H}_2.
\end{align}
Let $\theta = g_{a,b}^{(1)} - g_N = a^2 b^2 \tilde{H}_2$.  Using the expansion of the $B^t$-tensor again,
\begin{align} \label{BNexp1} \begin{split}
B^t(g_{a,b}^{(1)} ) &= B^t(g_N) + (B^t_{g_N})'\theta + \mathcal{Q}_N(\theta) \\ 
&= (B^t_{g_N})'\theta + \mathcal{Q}_N(\theta),
\end{split}
\end{align}
since $g_N$ is $B^t$-flat. From the formula for the linearized operator $S$ in (\ref{linop}) and Proposition \ref{LamProp}
it follows that 
\begin{align} \label{BtvS} \begin{split}
(B^t_{g_N})'\theta &= S_{g_N}\theta  - \mathcal{K}_{g_N} \delta_{g_N} \mathcal{K}_{g_N} \delta_{g_N} \overset{\circ}{\theta} \\
&= a^2 b^2 S(  \tilde{H}_2) - a^2 b^2 \mathcal{K}_{g_N} \delta_{g_N} \mathcal{K}_{g_N} \delta_{g_N} \overset{\circ}{\tilde{H}_2} \\
&= a^2 b^2 \lambda k_1^{(0)} - a^2 b^2 \mathcal{K}_{g_N} \delta_{g_N} \mathcal{K}_{g_N} \delta_{g_N} \overset{\circ}{\tilde{H}_2}.
\end{split}
\end{align}
Substituting this into (\ref{BNexp1}) gives 
\begin{align} \label{BNexp2} 
B^t(g_{a,b}^{(1)} ) = a^2 b^2 \lambda k_1^{(0)} - a^2 b^2 \mathcal{K}_{g_N} \delta_{g_N} \mathcal{K}_{g_N} \delta_{g_N} \overset{\circ}{\tilde{H}_2} + \mathcal{Q}_N(\theta).
\end{align}

By part $(i)$ of Proposition \ref{quadest},
\begin{align} \label{QN}
|Q_N (  a^2 b^2 \tilde{H}_2)|_{g_N} = O(a^4 b^4),
\end{align}
and from our observations above the same estimate holds if we estimate with respect to the norm induced by $g_{a,b}^{(1)}$.  Therefore, (\ref{alesize}) follows 
from this estimate and (\ref{BNexp2}).

To estimate the gauge-fixing term in (\ref{BNexp2}), we first observe that $\delta_{g_N} \mathcal{K}_{g_N} \delta_{g_N} : S^2(T^{*}N) \rightarrow T^{*}N$  is a
third order differential operator, while
\begin{align*}
\nabla_N^3 = \partial^3 + \Gamma(g_N)  \partial^2 + [ \partial \Gamma(g_N) + \Gamma(g_N)^2 ] \partial + [\partial^2 \Gamma(g_N) + \Gamma(g_N) \partial \Gamma(g_N)],
\end{align*}
where $\nabla_N$ denotes the covariant derivative and $\Gamma(g_N)$ the Christoffel symbols in the $g_N$-metric. Since
\begin{align*}
\tilde{H}_2 &= H_2 + O^{(4)}(|x|^{\epsilon}), \\
\Gamma_N &= O^{(3)}(|x|^{-3}),
\end{align*}
it follows that
\begin{align*}
a^2 b^2 \mathcal{K}_{g_N} \delta_{g_N} \mathcal{K}_{g_N} \delta_{g_N} \overset{\circ}{\tilde{H}_2} 
= a^2 b^2  O(|x|^{\epsilon - 4}),
\end{align*}
and combining this with (\ref{QN}) and (\ref{BNexp2}) we obtain (\ref{alemodK}).
\end{proof}

Next, on the compact piece, we have
\begin{proposition}
On $\{ |z| \geq 2b \} \subset Z$, we have
\begin{align}
\label{alesize2}
|B^t (g_{a,b}^{(1)})|_{g_{a,b}^{(1)}} =  a^6 b^6 O(|z|^{-\epsilon - 4}).
\end{align}
as $a,b,|z| \rightarrow 0$.  
\end{proposition}
\begin{proof}
Recall the metric $g_Z^{(1)}$ defined in (\ref{gZ1def}):
\begin{align*}
g_Z^{(1)} = g_0 + \eta_Z(z) +  a^2 b^2 \tilde{H}_{-2}(z), \ \ |z| \geq 2b,
\end{align*}
so that on the compact piece $\{ |z| \geq 2b \} \cap Z$ the refined approximate metric $g_{a,b}^{(1)}$ is just a rescaling of $g_Z^{(1)}$:
\begin{align} \label{gzab}
g_{a,b}^{(1)} = a^{-2} b^{-2} g_Z^{(1)}.
\end{align}
We can then essentially repeat the arguments of the preceding propositions and write $g_Z^{(1)} = g_Z + \theta$, where $\theta = a^2 b^2 \tilde{H}_{-2}$,
then expand $B^t$:
\begin{align} \label{BZexp} \begin{split}
B^t(g_Z^{(1)}) &= B^t(g_Z) + (B^t_Z)'(\theta) + \mathcal{Q}_Z(\theta) \\
&= (B^t_Z)'(\theta) + \mathcal{Q}_Z(\theta),
\end{split}
\end{align}
where as usual the subscript $Z$ indicates that the tensor is with respect to the metric $g_Z$.

We first estimate the term involving the linearization of $B^t$, by comparing $(B^t_Z)'(\theta)$ and
$(B^t_0)'(\theta)$, i.e., the linearized operator with respect to the flat metric acting on $\theta$.  Recall from Lemma \ref{SBKer} and Proposition \ref{Hminus2} that
\begin{align} \label{BtpFlat} \begin{split}
(B_0^t)'(\theta) &=  a^2 b^2 (B_0^t)'(\tilde{H}_{-2}) \\
&= a^2 b^2 (B_0^t)'\big(H_{-2} + O(|z|^{-\epsilon})\big) \\
&= a^2 b^2 O(|z|^{-\epsilon - 4}).
\end{split}
\end{align}
Clearly, for any metric $g$ the operators $(B^t)'$ and $S$ have the same general form, as given in (\ref{Sgenform}):
\begin{align}
\begin{split} \label{Bpgenform}
(B^t)' \theta &= ( g^{-2} + g* g^{-3} ) *\nabla^4 \theta + g* g^{-3} *Rm * \nabla^2 \theta
+ g* g^{-3}* \nabla Rm * \nabla \theta \\
&+ ( g^{-2} + g * g^{-3})*( \nabla^2 Rm + Rm * Rm) * \theta.
\end{split}
\end{align}
Using this, we can estimate the difference
\begin{align*}
[(B^t_Z)' - (B^t_0)'](\theta).
\end{align*}
We will need the following estimates, which follow from (\ref{DTId}), (\ref{gZexp1}), and (\ref{fullne}):
\begin{align} \label{ZEucdiffs} \begin{split}
g_Z - g_0 &= O(b^2),  \\
\nabla_Z^m T - \nabla_0^m T &= O(b) * \nabla^{m-1}_0 T + O(1) * \{ \nabla_0^{m-2} T + \cdots  +  T \}. 
\end{split}
\end{align}
Then by (\ref{Bpgenform}),
\begin{align} \label{Btpdiffs}
|[(B^t_Z)' - (B^t_0)'](\theta)| = a^2 b^2 O(|z|^{-4}).
\end{align}
Combining with (\ref{BtpFlat}), we obtain
\begin{align} \label{BpZest}
|(B^t_Z)'(\theta)|_{0} = a^2 b^2 O(|z|^{-\epsilon - 4}).
\end{align}
It is easy to see that the same estimate holds if we estimate with respect to the norm induced by $g_Z^{(1)}$.

For the remainder term $\mathcal{Q}$ in (\ref{BZexp}) we use Proposition \ref{QRemark} to show
\begin{align} \label{QZed}
|\mathcal{Q}_Z(\theta)|_{g_Z} = a^4 b^4 O(|z|^{-8}),
\end{align}
with the same estimate in the $g_Z^{(1)}$-metric.  Combining (\ref{BpZest}) and (\ref{QZed}) gives
\begin{align} \label{BttEst}
| B^t(g_Z^{(1)})|_{g_Z^{(1)}} =  a^2 b^2 O(|z|^{-\epsilon - 4}) +  a^4 b^4 O(|z|^{-6}).
\end{align}
Since $|z| \geq b$, this implies
\begin{align} \label{BttEst2}
| B^t(g_Z^{(1)})|_{g_Z^{(1)}} =  a^2 b^2 O(|z|^{-\epsilon - 4}).
\end{align}

By the scaling properties of the $B^t$-tensor,
\begin{align*}
B^t (g_{a,b}^{(1)}) &=  B^t( a^{-2} b^{-2} g_Z^{(1)}) \\
&= a^2 b^2 B^t( g_Z^{(1)}).
\end{align*}
Using (\ref{BttEst2}) and (\ref{gzab}) we conclude
\begin{align*}
|B^t (g_{a,b}^{(1)})|_{g_{a,b}^{(1)}} &= | a^2 b^2 B^t (g_Z^{(1)} )|_{g_{a,b}^{(1)}} \\
&= a^2 b^2 | a^2 b^2 B^t(g_Z^{(1)})|_{g_Z^{(1)}} \\
&= a^4 b^4 | B^t(g_Z^{(1)})|_{g_Z^{(1)}} \\
&=  a^6 b^6 O(|z|^{-\epsilon - 4}).
\end{align*}
\end{proof}

Finally, we have
\begin{proposition}
\label{Btsize}
Choosing $a = b$, we have
\begin{align}
\Vert B^t (g_{a,a}^{(1)}) - a^4 \lambda k_1^{(0)} \Vert_{C^{0, \alpha}_{\delta - 4}(X_{a,a})}
 = O(a^{4 + \delta - \epsilon})
\end{align}
as $a \rightarrow 0$. 
\end{proposition}
\begin{proof}
We begin estimating the leading term in the $C^{0, \alpha}_{\delta - 4}$-norm.
On the damage zone, $k_1^{(0)}=0$, so by (\ref{dsize})
\begin{align} \label{wnDZ}
w^{4 - \delta} |  B^t (g_{a,a}^{(1)}) - a^4 \lambda k_1^{(0)}  | = O( a^{\delta-4 } a^{8 - \epsilon})
= O( a^{4 + \delta - \epsilon}).
\end{align}
By (\ref{alemodK}), on the AF piece
\begin{align} \label{wnALE1} \begin{split}
w^{4 - \delta} |   B^t (g_{a,a}^{(1)}) - a^4 \lambda k_1^{(0)}  | & = |x|^{4 - \delta}\{ a^4 O(|x|^{\epsilon - 4}) + O(a^8) \} \\
&= a^4 O(|x|^{\epsilon - \delta})+ a^8 O(|x|^{4 - \delta}).
\end{split}
\end{align}

Recall from Remark \ref{epvsdel2} that $-\epsilon < \delta < 0$; 
hence $\epsilon - \delta > 0$ and 
\begin{align} \label{wnALE} 
w^{4 - \delta} |   B^t (g_{a,a}^{(1)}) - a^4 \lambda k_1^{(0)}  | = O( a^{4 + \delta - \epsilon}).
\end{align}

On the compact piece, by (\ref{alesize2})
\begin{align}
\begin{split}
w^{4 - \delta} |   B^t (g_{a,a}^{(1)}) - a^4 \lambda k_1^{(0)} | &= a^{2\delta - 8} O(|z|^{4 - \delta}) \{ a^{12} O(|z|^{-\epsilon - 4})\} \\
&= a^{2\delta + 4} O(|z|^{ - \delta - \epsilon}) \\
&= O(a^{4 + \delta - \epsilon}),
\end{split}
\end{align}
since $|z| \geq 2a$ and $-\epsilon < \delta < 0$ (see Remark \ref{epvsdel2}).  

For estimating the H\"older part of the weighted norm, one must use the
formula \eqref{Qexp} in the proof of Proposition \ref{quadest}.
For example, the term with $h * \nabla^4 h$ is estimated like
\begin{align}
\begin{split}
&w^{ 4 - \delta} (x_0) \frac{ | (h * \nabla^4 h) (x_1) - (h * \nabla^4 h)(x_2)|}{|d(x_1, x_2)|^{\alpha}}\\
&\ \ \ \ \leq
 w^{-\delta} (x_0) |h(x_1)| \cdot  w^{4 - \delta}(x_0) \frac{ | \nabla^4 h (x_1) -
\nabla^4 h(x_2) |} { d(x_1, x_2)^{\alpha}} \\
& \ \ \ \ + w^{- \delta}(x_0) \frac{ | h(x_1) - h(x_2) |}{|d(x_1, x_2)|^{\alpha}}
\cdot w^{4 - \delta}(x_0) |\nabla^4 h (x_2)|,
\end{split}
\end{align}
and all other terms are estimated similarly, the complete computation is
lengthy but straightforward, so is omitted.
\end{proof}
\subsection{The approximate cokernel}
In this subsection, we define tensors $o_1, o_2, o_3$ and
$k_1, k_2, k_3$ which will be crucial in the Lyapunov-Schmidt
reduction in Section \ref{Lyap}.

\begin{remark}{\em
It is clear that all of the tensors in this section may be
chosen to be invariant under the group action, so we will
do this automatically without mention in every case.}
\end{remark}

Recall from Section \ref{afkersec} we denoted the cokernel of the asymptotically flat
manifold $(N,g_N)$ by $o_1$, and it is given by
\begin{align} \label{ALEcoke1}
o_1 = \mathcal{K}\omega_1 + f g_N.
\end{align}
In Section \ref{auxiliary} we defined a compactly supported symmetric $(0,2)$-tensor $k_1^{(0)} \in C_{\delta - 4}$ which satisfies
\begin{align} \label{k1props}
\| k_1^{(0)} \|_{C_{\delta - 4}} &\leq C_1, \\
\label{k1props2}
\int \langle k_1^{(0)} , o_1 \rangle_{g^{(0)}}\ dV_{g^{(0)}} &= 1,
\end{align}
where $C_1$ is independent of $a,b$.  Note that the quantities in (\ref{k1props}) and (\ref{k1props2}) are all computed with respect to the ``naive'' approximate metric $g^{(0)}$ defined
in Section \ref{approx}.  Since from now on we will be working in the refined approximate metric $g^{(1)}$ defined in Section \ref{better}, we will need to slightly rescale
$k_1^{(0)}$ so that (\ref{k1props}) and (\ref{k1props2}) hold with respect to $g^{(1)}$.  To this end, define
\begin{align} \label{k1redef}
k_1 = \Big\{ \int \langle k_1^{(0)} , o_1 \rangle_{g^{(1)}}\ dV_{g^{(1)}} \Big\}^{-1} k_1^{(0)}.
\end{align}
Then by (\ref{k1props2}) and (\ref{k1redef}),
\begin{align} \label{k1renorm}
\int \langle k_1, o_1 \rangle_{g^{(1)}}\ dV_{g^{(1)}} = 1.
\end{align}

\begin{claim} \label{k1claim}
We have
\begin{align} \label{k10}
k_1^{(0)} = \big( 1 + O(b^2) ) k_1.
\end{align}
\end{claim}

\begin{proof}
By (\ref{idmtnew}), on the support of $k_1^{(0)}$ we have
\begin{align}  \label{gNg} \begin{split}
g = g^{(1)} &= g_N + O(b^2) \\
&= g^{(0)} + O(b^2),
\end{split}
\end{align}
hence the volume forms satisfy
\begin{align}
dV_{g^{(1)}} &=   \big(1 + O( b^2) \big) dV_{g^{(0)}}.
\end{align}
Therefore,
\begin{align} \begin{split}
\int \langle k_1^{(0)} , o_1 \rangle_{g^{(1)}}\ dV_{g^{(1)}} &= \int [g^{(1)}]^{ik} [g^{(1)}]^{j \ell} [k_1^{(0)}]_{ij} [o_1]_{k \ell}\ dV_{g^{(1)}} \\
& \hskip-.5in = \int \big\{ [g^{(0)}]^{ik} + O(b^2) \big\} \big\{  [g^{(0)}]^{j \ell} + O(b^2) \big\}  [k_1^{(0)}]_{ij} [o_1]_{k \ell}\  \big(1 + O( b^2) \big) dV_{g^{(0)}} \\
& \hskip-.5in = \int \langle k_1^{(0)} , o_1 \rangle_{g^{(0)}}\ dV_{g^{(0)}} + O(b^2) \\
&\hskip-.5in = 1 + O(b^2).
\end{split}
\end{align}
Substituting this into (\ref{k1redef}) gives (\ref{k10}).
\end{proof}
\begin{remark}{\em
From now on, all metric-dependent quantities will be with respect to $g = g^{(1)}$.  To simplify the notation,
we will suppress the superscript.
}
\end{remark}
Let $o_3$ denote the cokernel element on the compact manifold $(Z,g_Z)$ given by scaling of the metric:
\begin{align} \label{Zcoke}
o_3 = (ab)^{-2 + \delta} g_Z.
\end{align}
(The reason for the scale factor will become apparent in a moment).  Fix a smooth positive cut-off function $\phi_3$ supported in
$Z \setminus \{ z_0 \}$ with
\begin{align} \label{p3normal}
\int \phi_3 | g_{Z}|^2\ dV = 1,
\end{align}
where we again emphasize that the volume form and inner product are with respect to $g = g^{(1)}$.
Define
\begin{align} \label{k2def}
k_3 =   (ab)^{2-\delta} \phi_3 g_Z.
\end{align}
We claim that there is a constant $C_2$  such that
\begin{align} \label{k2bound}
\| k_3 \|_{C_{\delta - 4}} \leq C_2.
\end{align}
To see this, first recall that by (\ref{idmtnew}),
(\ref{idmtnew}), on the support of $\phi_3$ we have
\begin{align*}
[g^{(1)}]_{ij} &= (ab)^{-2} \big[ (g_Z)_{ij} + O( b^2) \big], \\
[g^{(1)}]^{ij} &= (ab)^{2} \big[ (g_Z)^{ij} + O( b^2) \big].
\end{align*}
Also, on the support of $\phi_3$ the weight $w(x) \approx (ab)^{-2}$.  It follows that
\begin{align*}
\sup |k_3|_{g^{(1)}}  w^{4 - \delta}  &= c_2 (ab)^{2-\delta} \sup |\phi_3|| g_Z|_{g^{(1)}}w^{4 - \delta}  \\
&   = c_2 (ab)^{2-\delta} \sup |\phi_3| \big\{ [g^{(1)}]^{ik} [g^{(1)}]^{j \ell} (g_Z)_{ij} (g_Z)_{k \ell} \big\}^{1/2} w^{4 - \delta} \\
&  = c_2 (ab)^{2-\delta} \sup \Big\{  |\phi_3| \big\{ (ab)^{4} \big[ (g_Z)^{ik} + O(b^2) \big] \\
& \hskip.25in \times  \big[ (g_Z)^{j \ell} + O(b^2) \big] (g_Z)_{ij} (g_Z)_{k \ell}\big\}^{1/2} w^{4 - \delta}  \Big\} \\
&\leq  C_2.
\end{align*}
This estimate clarifies the choice of scaling in the definitions of $k_3$ and $o_3$: the scale factor $(ab)^{2-\delta}$ in (\ref{k2def}) is necessary to get the bound (\ref{k2bound}), while
the factor $(ab)^{-2 + \delta}$ in the definition of $o_3$ was chosen to give the normalization
\begin{align} \label{2normal} \begin{split}
\int \langle k_3, o_3 \rangle\ dV &= \int \langle (ab)^{2-\delta} \phi_3 g_Z, (ab)^{-2 + \delta} g_{Z} \rangle\ dV \\
& = \int \phi_3 | g_{Z}|^2\ dV = 1.
\end{split}
\end{align}

Next, denote
\begin{align} \label{o3def}
o_2  = g = g^{(1)}.
\end{align}
We claim that there is a tensor $k_2$, compactly supported in $N$, which satisfies the normalization
\begin{align} \label{3normal}
\int \langle k_2, o_2 \rangle\ dV = \int \mbox{tr}_{g} k_2\ dV = 1
\end{align}
and the orthogonality condition
\begin{align} \label{13ortho}
\int \langle k_2, o_1 \rangle\ dV = 0.
\end{align}
(Note that the integral in (\ref{13ortho}) makes sense, since $k_2$ is compactly supported in $N$, even though $o_1$ is not globally defined.)

To see that such a tensor exists, just take two smooth, positive cut-off functions
$\eta_1$, $\eta_2$ with compact support in $N$ and let
\begin{align*}
k_2 = (c_1 \eta_1  + c_2 \eta_2)o_2 = (c_1 \eta_1  + c_2 \eta_2)g,
\end{align*}
where $c_1$ and $c_2$ are constants to be determined.  Then
\begin{align} \label{intfacts1}  \begin{split}
\int \langle k_2, o_1 \rangle\ dV &= c_1 \int \eta_1 (\mbox{tr}_{g} o_1)\ dV + c_2 \int \eta_2  (\mbox{tr}_{g} o_1)\ dV, \\
\int \langle k_2, o_2 \rangle\ dV &= 4 c_1 \int \eta_1\ dV + 4 c_2 \int \eta_2\ dV. \\
\end{split}
\end{align}
By (\ref{gNg}), on the support of $k_2$
\begin{align*}
g = g^{(1)} = g_N + O(b^2).
\end{align*}
Therefore,
\begin{align*}
\mbox{tr}_{g} o_1  &= [g^{(1)}]^{k \ell}       (o_1)_{k \ell} \\
&= [ g_N + O(b^2) ]^{k \ell}   [\mathcal{K}\omega_1 + f g_N]_{k \ell}  \\
&= 4 f + O(b^2).
\end{align*}
Therefore, we can estimate the integrals in (\ref{intfacts1}) by
\begin{align} \label{intfacts2}  \begin{split}
\int \langle k_2, o_1 \rangle\ dV &=  c_1 \Big\{ 4 \int \eta_1 f\ dV  + O(b^2) \Big\} + c_2 \Big\{ 4 \int \eta_2 f\ dV + O(b^2) \Big\}, \\
\int \langle k_2, o_2 \rangle\ dV &=  c_1 \Big\{ 4 \int \eta_1\ dV   \Big\} +  c_2 \Big\{ 4 \int \eta_2\ dV   \Big\}. \\
\end{split}
\end{align}
Consequently, once $a,b$ are small enough it is possible to choose the cut-off functions $\eta_1, \eta_2$ and the
constants $c_1, c_2$ so that (\ref{3normal}) and (\ref{13ortho}) hold.

\section{Lyapunov-Schmidt reduction}
\label{Lyap}
In this section, we perform the main reduction of the problem
from an infinite-dimensional problem to a finite-dimensional problem.

\begin{remark} \label{equiRemark} {\em Since we are carrying out an equivariant gluing construction,
from now on all operators are understood to act on sections of the relevant bundle which are invariant
under the group actions described above.}
\end{remark}
\subsection{The modified nonlinear map} Let
\begin{align}
\mathcal{D} = \Big\{ h \in C^{4,\alpha}_{\delta}(X_{a,b}) \ :\ \int \langle h, k_1 \rangle = 0,\  \int \langle h, k_3 \rangle = 0\ \Big\}.
\end{align}
Define the mapping $H^t : \mathbb{R} \times \mathbb{R} \times \mathcal{D} \rightarrow C^{0,\alpha}_{\delta - 4}(X_{a,b})$ by
\begin{align} \label{Htdef}
H^t(\lambda_1, \lambda_2, \theta) = P_{g^{(1)}}(\theta) - \lambda_1 k_1 - \lambda_2 k_2.
\end{align}
Let $(H^t)' : \mathbb{R} \times \mathbb{R} \times \mathcal{D} \rightarrow C^{0,\alpha}_{\delta - 4}$ denote the linearization of $H$ at $(0,0,0)$:
\begin{align} \label{Htp}
(H^t)'(\lambda_1, \lambda_2, h) = \frac{d}{ds} H^t(s\lambda_1, s\lambda_2, s h )\Big|_{s=0}.
\end{align}
Then $(H^t)'$ is given by
\begin{align} \label{Htpform}
(H^t)'(\lambda_1, \lambda_2, h) = S(h) - \lambda_1 k_1 - \lambda_2 k_2,
\end{align}
where $S = S^t$ is the linearization of $P$ at $g^{(1)}$.

\begin{proposition}
\label{unifinj}
For $a,b$ sufficiently small,
the map $(H^t)' : \mathbb{R} \times \mathbb{R} \times \mathcal{D} \rightarrow C^{0,\alpha}_{\delta - 4}$ is uniformly injective: i.e., there is a constant $\delta_0 > 0$
which is independent of $a,b$ such that
\begin{align} \label{Hbelow}
\| (H^t)'(\lambda_1, \lambda_2, h) \|_{C^{\alpha}_{\delta - 4}} \geq \delta_0 \big(  |\lambda_1| + |\lambda_2| + \| h \|_{C^{4, \alpha}_{\delta}} \big).
\end{align}
\end{proposition}

\begin{proof}  We argue via contradiction: if (\ref{Hbelow}) does not hold, then there is a sequence $(\lambda_1^i, \lambda_2^i, h_i) \in \mathcal{D}$ with
\begin{align} \label{assumption} \begin{split}
 |\lambda_1^i| + |\lambda_2^i| + \| h_i \|_{C^{4, \alpha}_{\delta}} &= 1 \ \ \forall i, \\
 \epsilon_i \equiv (H^t)'(\lambda_1^i, \lambda_2^i, h_i)  & \rightarrow 0 \ \ \mbox{in } C^{0,\alpha}_{\delta - 4}.
 \end{split}
 \end{align}
If we pair $\epsilon_i$ with $ \eta o_1$ and integrate, where $\eta$ is a cut-off function with
\begin{align} \label{etad}
\eta(x) =
\begin{cases}
1  &  |x| \leq a^{-1} \\
0  &  |x| > 2a^{-1},
\end{cases}
\end{align}
then
\begin{align} \label{pairo1}
\int \langle \epsilon_i, \eta o_1 \rangle\ dV = \int_{B} \langle Sh_i,\eta  o_1 \rangle\ dV  - \lambda_1^i \int \langle k_1, \eta o_1 \rangle\ dV - \lambda_2^i \int \langle k_2, \eta o_1 \rangle\ dV.
\end{align}
Since $\eta \equiv 1$ on the support of $k_1$ and $k_2$, by the normalization (\ref{k1props}) and the orthogonality condition (\ref{13ortho}) we can rewrite this as
\begin{align} \label{lam1form}
\lambda_1^i = - \int \langle \epsilon_i, \eta o_1 \rangle\ dV + \int \langle Sh_i,\eta o_1 \rangle\ dV.
\end{align}
For the first term on the right-hand side, note that
\begin{align} \label{eisup}
\| \epsilon_i \|_{C^{0,\alpha}_{\delta - 4}} \geq \sup \{ |\epsilon_i| w^{4 - \delta} \},
\end{align}
where $w$ is the weight function.  According to (\ref{wdef}), on the support of $\eta$ the weight function is $w(x) = |x|$ (for $|x|$ large).  Also, by Theorem \ref{ALEcokethm} the cokernel $o_1$ satisfies
\begin{align*}
|o_1| \leq C |x|^{-2}.
\end{align*}
Therefore,
\begin{align} \label{epest} \begin{split}
\big| \int \langle \epsilon_i, \eta o_1 \rangle\ dV \big| &\leq \int |\epsilon_i| |\eta o_1|\ dV  \\
&\leq C \| \epsilon_i \|_{C^{0,\alpha}_{\delta - 4}} \int_{1 << |x| \leq a^{-1}/2 } |x|^{\delta - 4} |x|^{-2}\ dV \\
&\leq C \| \epsilon_i \|_{C^{0,\alpha}_{\delta - 4}} \rightarrow 0
\end{split}
\end{align}
as $i \rightarrow \infty$.

For the second term on the right-hand side of (\ref{lam1form}) we integrate by parts, using the fact that $S$ is self-adjoint:
\begin{align} \label{Sadjab}
\int \langle Sh_i,\eta o_1 \rangle\ dV = \int \langle h_i , S(\eta o_1) \rangle\ dV.
\end{align}
Using the formula for $S$ in (\ref{Sgenform}) and the Leibniz rule, write
\begin{align} \label{SLeibniz} \begin{split}
S(\eta o_1) &= ( g^{-2} + g* g^{-3} ) *\nabla^4 (\eta o_1)+ g* g^{-3} *Rm * \nabla^2 (\eta o_1) \\
& + g* g^{-3}* \nabla Rm * \nabla (\eta o_1) + ( g^{-2} + g * g^{-3})*( \nabla^2 Rm + Rm * Rm) *(\eta o_1) \\
&= \eta S o_1 + ( g^{-2} + g* g^{-3} ) * \sum_{j \geq 1}^4 \nabla^{4-j} o_1 * \nabla^j \eta \\
&\ \ \ \ + g* g^{-3} *Rm * \sum_{j \geq 1}^2 \nabla^{2-j} o_1 * \nabla^j \eta
+ g* g^{-3}* \nabla Rm * o_1 * \nabla \eta.
\end{split}
\end{align}
By (\ref{fdecay}) and Theorem \ref{ALEcokethm}, on the support of $|\nabla \eta|$
\begin{align*}
|\nabla^{m} o_1| &= O(a^{m + 2}), \\
|\nabla^m \eta| &= O(a^{m}), \\
|\nabla^m Rm| &= O(a^{m+4}).
\end{align*}
Therefore, from (\ref{SLeibniz}) we have
\begin{align} \label{Srem}
S(\eta o_1) = \eta S o_1 + \{ \mbox{Error} \},
\end{align}
where the error is supported on $\{ a^{-1}/4 \leq |x| \leq a^{-1}/2 \}$ and satisfies
\begin{align} \label{Err}
|\{ \mbox{Error}\}| = O(a^6).
\end{align}
It follows from (\ref{Sadjab})
\begin{align} \label{Sadjabc}
\int \langle Sh_i,\eta o_1 \rangle\ dV = \int \langle h_i , \eta S(o_1) \rangle\ dV + \int \langle h_i, \{ \mbox{Error} \} \rangle\ dV.
\end{align}
Since $\| h_i \|_{C_{\delta}^{4,\alpha}} \leq 1$,
\begin{align} \label{hiwt}
|h_i| \leq C w^{\delta},
\end{align}
hence on the support of $\mbox{Error}$
\begin{align} \label{hisize}
|h_i| \leq C a^{-\delta},
\end{align}
hence by (\ref{Err})
\begin{align}
\Big| \int_{ \{ a^{-1}/4 \leq |x| \leq a^{-1}/2 \} } \langle h_i, \{ \mbox{Error} \} \rangle\ dV \Big| \leq C a^{2-\delta}.
\end{align}
Therefore,
\begin{align} \label{Sadjr} \begin{split}
 \int \langle Sh_i,\eta o_1 \rangle\ dV = \int \langle h_i , \eta S o_1 \rangle\ dV +  O(a^{2-\delta}).
\end{split}
\end{align}

Let $S_N$ denote the linearized operator with respect to the metric $g_N$.  Then $S_N o_1 = 0$, hence
\begin{align} \label{SNdiff} \begin{split}
S o_1 &= (S - S_N)o_1 + S_N o_1 \\
&= (S-S_N) o_1.
\end{split}
\end{align}
Using (\ref{Sgenform}) with (\ref{gNg}), we can estimate
\begin{align} \label{SNS1} \begin{split}
|(S-S_N) o_1| &\lesssim b^2 |\nabla^4 o_1| + b^2 a |\nabla^3 o_1| + a^2 b^2 |\nabla^2 o_1|+  a^3 b^2 |\nabla o_1|  \\
& \ + a^4 b^2 |o_1| + b^2 |Rm| |\nabla^2 o_1| + b^2 |\nabla Rm||\nabla o_1| + b^2 \big( |\nabla^2 Rm| + |Rm|^2 \big)|o_1|.
\end{split}
\end{align}
Therefore, by (\ref{hisize}),
\begin{align} \label{SNSdiff} \begin{split}
\Big| \int \langle h_i , \eta S o_1 \rangle\ dV  \Big| &= \Big| \int \langle h_i , \eta (S-S_N) o_1 \rangle\ dV \Big| \\
&\leq C \int w^{\delta} |(S-S_N) o_1|\ dV \\
&\leq C b^2 a^{2-\delta}.
\end{split}
\end{align}
Combining the above, we conclude
\begin{align} \label{lamone}
\int \langle Sh_i,\eta o_1 \rangle\ dV =  O(a^{2-\delta}),
\end{align}
hence by (\ref{lam1form}),
\begin{align} \label{Lam1zed}
\lambda_1^i \rightarrow 0 \ \ \mbox{as } i \rightarrow \infty.
\end{align}

Next, pair $\epsilon_i$ with $\eta o_2$ and integrate:
\begin{align} \label{pairo3}
\int \langle \epsilon_i, \eta o_2 \rangle\ dV = \int \langle Sh_i, \eta o_2 \rangle\ dV  - \lambda_1^i \int \langle k_1, \eta o_2 \rangle\ dV - \lambda_2^i \int \langle k_2, \eta o_2 \rangle\ dV.
\end{align}
By the normalization (\ref{3normal}), we can rewrite this as
\begin{align} \label{lam2form}
\lambda_2^i = - \int  \langle \epsilon_i, \eta o_2 \rangle\ dV + \int  \langle Sh_i, \eta o_2 \rangle\ dV  - \lambda_1^i \int  \langle k_1, \eta o_2 \rangle\ dV
\end{align}
As in (\ref{epest}), we can estimate the first integral on the right as
\begin{align} \label{epest2} \begin{split}
\big| \int \langle \epsilon_i, \eta o_2 \rangle\ dV \big| &\leq \int |\epsilon_i| |\eta o_2|\ dV  \\
&\leq C \| \epsilon_i \|_{C^{0,\alpha}_{\delta - 4}} \int_{1 < |x| \leq a^{-1}/2 } |x|^{\delta - 4}  \ dV \\
&\leq C_{\delta} \| \epsilon_i \|_{C^{0,\alpha}_{\delta - 4}} \ \ \mbox{(since $\delta < 0$)},
\end{split}
\end{align}
which limits to $0$ as $i \rightarrow \infty.$  The second term on the right we estimate as we did above; namely,
\begin{align} \label{Sadj2}
\int \langle Sh_i,\eta o_2 \rangle\ dV = \int \langle h_i , S(\eta o_2) \rangle\ dV.
\end{align}
Using the fact that $\nabla o_2 = 0$, we can estimate as in (\ref{SLeibniz}):
\begin{align*}
S(\eta o_2) &= ( g^{-2} + g* g^{-3} ) *\nabla^4 (\eta o_2)+ g* g^{-3} *Rm * \nabla^2 (\eta o_2) \\
& + g* g^{-3}* \nabla Rm * \nabla (\eta o_2) + ( g^{-2} + g * g^{-3})*( \nabla^2 Rm + Rm * Rm) *(\eta o_2) \\
&= \eta S o_2 + ( g^{-2} + g* g^{-3} )  * \nabla^4 \eta * o_2 + g* g^{-3} * Rm  * \nabla^2 \eta * o_2 \\
& \ \ \ \ + g* g^{-3}* \nabla Rm * o_2 * \nabla \eta.
\end{align*}
Since $S o_2 = S g = 0$,
\begin{align} \label{Srem2}  \begin{split}
S(\eta o_2) &= \eta S o_2 + \{ \mbox{Error} \}, \\
&= \{ \mbox{Error} \},
\end{split}
\end{align}
where the error is supported on $\{ a^{-1}/4 \leq |x| \leq a^{-1}/2 \}$ and satisfies
\begin{align} \label{Err2}
|\{ \mbox{Error}\}| = O(a^4).
\end{align}
Using (\ref{hisize}), we can therefore estimate
\begin{align*}
\Big| \int_{ \{ a^{-1}/4 \leq |x| \leq a^{-1}/2 \} } \langle h_i, \{ \mbox{Error} \} \rangle\ dV \Big| \leq C  a^{-\delta}.
\end{align*}
Hence,
\begin{align} \label{Sadjr2} \begin{split}
 \int \langle Sh_i,\eta o_2 \rangle\ dV =  O( a^{-\delta}).
\end{split}
\end{align}

For the last term in (\ref{lam2form}), we use the fact that $\lambda_1^i \rightarrow 0$, and that $k_1$ is compactly supported:
\begin{align*}
\Big| \lambda_1^i \int  \langle k_1, \eta o_2 \rangle\ dV \Big| &\leq   | \lambda_1^i | \int \eta |k_1| |o_2|\ dV \\
&\leq C | \lambda_1^i | \int_{\mathrm{supp}\{ k_1\}}  \ dV \\
&\leq C | \lambda_1^i |   \rightarrow 0,
\end{align*}
as $i \rightarrow 0$.
Combining with (\ref{Sadjr2}), (\ref{epest2}), and (\ref{lam2form}), we see that
\begin{align} \label{l2zed}
\lambda_2^i \rightarrow 0
\end{align}
as $i \rightarrow \infty$.

Consequently, by (\ref{assumption}) we now know
\begin{align}
\label{hlimit}
  \| h_i \|_{C^{4, \alpha}_{\delta}} &\rightarrow 1,  \\
\label{hlimit2}
\| Sh_i \|_{C^{0,\alpha}_{\delta - 4}} & \rightarrow 0,
\end{align}
as $i \rightarrow \infty$.

 The remainder of the proof is a standard ``blow-up'' argument, which
we only briefly outline. Let $(a_i, b_i)$ be a sequence of gluing parameters with  $(a_i, b_i) \rightarrow (0,0)$ as $i \rightarrow \infty$,
and let $p_i \in X_{a_i,b_i}$ a sequence of points at which the supremum
in \eqref{hlimit} is attained.  We have the three possibilities:
\begin{itemize}
\item (1) $p_i \rightarrow p \in N$. In this case, standard elliptic
estimates produce a nontrivial solution of the limiting equation
$S h_{\infty} = 0$ on $(N, g_N)$ with $h \in C^{4, \alpha}_{\delta}$.
By Theorem \ref{afker}, $h_{\infty} = c \cdot o_1$ for some $c \in \RR$.
Since
\begin{align}
\int \langle h_i, k_1 \rangle dV_{(a_i,b_i)} = 0,
\end{align}
$k_1$ has compact support on $N$, and $h_i \leq C w^{\delta}$,
the integrand is bounded.  Therefore,
\begin{align}
\int \langle h_\infty, k_1 \rangle dV_{g_N} = 0,
\end{align}
which implies that $c = 0$, a contradiction.

\item (2) $p_i \rightarrow p \in Z \setminus \{z_0\}$.  In this case, define
$\tilde{h}_i =  (ab)^{2 + \delta} h_i$. It is easy to see that
this scaling preserves the $C^{4,\alpha}_{\delta}$ norm, with
respect to the metric $\tilde{g}^{(1)}_{a,b} = (ab)^2 {g}^{(1)}_{a,b}$.
Standard elliptic estimates produce a nontrivial solution of the limiting equation
$S \tilde{h}_{\infty} = 0$ on $(Z, g_Z)$ with $h \in C^{4, \alpha}_{\delta}$.
By Theorem \ref{comker}, $h_{\infty} = c \cdot g_Z$ for some $c \in \RR$.
Since
\begin{align}
\int \langle h_i, k_3 \rangle dV_{(a_i,b_i)} = 0,
\end{align}
scaling shows that
\begin{align}
\int \langle \tilde{h}_i , \phi_3 g_Z \rangle_{\tilde{g}} dV_{\tilde{g}}
= (ab)^{2 \delta} \cdot 0 = 0.
\end{align}
Since $\phi_3$ has compact support on $Z \setminus \{z_0\}$,
and $\tilde{h}_i \leq C \tilde{w}^{\delta} = C (ab w)^{\delta}$,
the integrand is bounded, which implies that
\begin{align}
\int \langle \tilde{h}_\infty, k_3 \rangle dV_{g_Z} = 0,
\end{align}
which implies that $c = 0$, a contradiction.

\item (3) If neither of the above cases happen, then there are two possibilities:
a subsequence can approach the damage zone from the AF side, or
from the compact side. We give the argument in the former case,
the proof of the latter case is similar.
Fix a point $O \in N$ and let
\begin{align*}
d_i = \mbox{dist}_{g^{(1)}_{a_i,b_i}}(O,p_i).
\end{align*}
Clearly, $d_i \rightarrow \infty$ as $i \rightarrow \infty$ (otherwise we are in case (1) above).
For $i >> 1$ we can view the sequence $\{ p_i \} \subset N_i = N \cap A_i$, where $A_i$ is the
annulus $\{ R_0 < |x| < 2a_i^{-1} \}$ and $N_i$ is equipped with the metric
$g_i = g^{(1)}_{a_i,b_i}$.  Let $\psi_i : A_i  \rightarrow N$ denote dilation,
\begin{align*}
\psi_i : x \mapsto d_i x,
\end{align*}
and define
\begin{align} \label{tgit}
\tilde{g}_i = d_i^{-2} \psi_i^{\ast} g_i \Big|_{ \{ m_i  d_i \leq |x| \leq M_i   d_i \}},
\end{align}
where $m_i \to 0$ and $M_i \rightarrow \infty$ are chosen so that the
annulus $\{ m_i  d_i \leq |x| \leq M_i   d_i \} \subset N_i$.  Denote the dilated
coordinates by $\tilde{x}^i$; then $\tilde{g}_i$ is defined on the annulus $\{ m_i \leq |\tilde{x}^i| \leq M_i \}$.
Finally, define
\begin{align*}
\tilde{h}_i = d_i^{-2 + \delta} \psi_i^{\ast}h_i,
\end{align*}
which preserves the  $C^{4,\alpha}_{\delta}$-norm.  Taking the limit as $i \rightarrow \infty$ we have $\tilde{g}_i \rightarrow ds^2$,
the flat metric on $\mathbb{R}^4 \setminus \{ 0 \}$,  $\tilde{h}_i \rightarrow h_{\infty} $, where $h_{\infty}$ satisfies
\begin{align} \label{S0facts} \begin{split}
S_0 h_{\infty} &= 0\ \ \mbox{on } \RR^4 \setminus \{0\}, \\
h_{\infty} &\in C^{4,\alpha}_{\delta}(\RR^4 \setminus \{0\}),
\end{split}
\end{align}
and $S_0$ is the linearized operator with respect to the flat metric (see (\ref{Sform1})).  Note the weight function in the limit is given by $w = |x|$.
Since $-1 < \delta < 0$, $\delta$ is not an indicial root by Proposition
\ref{indclaim}. This implies that $S_0 : C^{4,\alpha}_{\delta}
\rightarrow C^{0, \alpha}_{\delta -4}$ is an isomorphism (see
\cite{Bartnik, LockhartMcOwen}),
so $h_{\infty} \equiv 0$, which is a contradiction.
\end{itemize}
This contradiction argument finishes the proof of Proposition \ref{unifinj}.
\end{proof}

We next quote without proof the following standard implicit
function theorem:
\begin{lemma}
\label{ift}
Let $H : E \rightarrow F$ be a smooth map between Banach spaces.
Define $Q = H - H(0) - H'(0)$. Assume that there are positive constants
$C_1, s_0, C_2$ so that the following are satisfied:
\begin{itemize}
\item $(1)$ The nonlinear term $Q$ satisfies
\begin{align}
\label{iquad}
\Vert Q(x) - Q(y) \Vert_F \leq C_1 (\Vert x \Vert_E + \Vert y \Vert_E)
\Vert x - y \Vert_E
\end{align}
for every $x, y \in B_E (0, s_0)$.
\item $(2)$ The linearized operator at $0$, $H'(0) : E \rightarrow F$
is an isomorphism with inverse bounded by $C_2$.
\end{itemize}
If
\begin{align}
s < \min \Big( s_0, \frac{1}{2 C_1 C_2} \Big)
\end{align}
and
\begin{align}
\Vert H(0) \Vert_F < \frac{ s}{2 C_2},
\end{align}
Then there is a unique solution $x \in B_E(0,s)$ of the
equation $H(x) = 0$.
\end{lemma}

We end this section with the following existence theorem:
\begin{theorem}Let $a = b$. Then for all $a$ sufficiently small,
there exist constants $\lambda_1, \lambda_2 \in \RR$
and $\theta \in \mathcal{D}$ satisyfying
\begin{align}
\label{thetanorm}
\Vert \theta \Vert_{C^{4, \alpha}_{\delta}} < C a^{4 + \delta - \epsilon}
\end{align}
so that
\begin{align}
\label{bigeqn}
P_{g^{(1)}}(\theta) = \lambda_1 k_1 + \lambda_2 k_2.
\end{align}
\end{theorem}
\begin{proof}
We denote the refined approximate metric by $g^{(1)} = g^{(1)}_{a}$, or by $g$ if the context is clear.

We will find a zero of $H$, so we need to verify the assumptions
in Lemma \ref{ift} with $E = \RR \times \RR \times \mathcal{D}$
and $F = C^{0, \alpha}_{\delta}$, beginning with $(1)$:

\begin{lemma} \label{checkQuad} The quadratic estimate \eqref{iquad}
holds for $H : \RR \times \RR \times \mathcal{D} \rightarrow C^{0, \alpha}_{\delta}$.
\end{lemma}

\begin{proof}

This follows from Proposition \ref{quadest}, once we verify the assumptions (\ref{Rmw})
(the assumptions (\ref{wdoes}) clearly hold).  We need to verify the estimate on each of
the three regions: the asymptotically flat piece, the damage zone, and the compact piece.
Recall that the weight is given by (\ref{wdef}).

On the asymptotically flat piece, i.e., for $|x| \leq a^{-1}$,
\begin{align*}
g^{(1)} = g_N + a^4 \tilde{H}_2.
\end{align*}
Let $h_1 = g^{(1)} - g_N = a^4 \tilde{H}_2$, then using the formula (\ref{e1rm}) we have
\begin{align*}
Rm_{g^{(1)}} = Rm_{g_N} + (g^{(1)})^{-1} * \nabla_N^2 h_1 +  (g^{(1)})^{-2} * \nabla_N h_1 * \nabla h_1.
\end{align*}
By Proposition \ref{LamProp} and the fact that $g_N$ is asymptotically flat of order $2$,
\begin{align} \label{Rmg1N} \begin{split}
Rm_{g^{(1)}} &= O(|x|^{-4}) + O(a^4) + O(a^8 |x|^2) \\
&= O(|x|^{-4}) + O(a^4).
\end{split}
\end{align}
Since $w(x) = |x|$ for $|x| >> 1$, it follows that
\begin{align} \label{wRmg1N}
w(x)^2 | Rm_{g^{(1)}}| \leq C_0.
\end{align}
Similarly, using (\ref{DRmh})
\begin{align} \label{DRm1}
\nabla_{g^{(1)}} Rm_{g^{(1)}} = O(|x|^{-5}) + O(a^4 |x|^{-3}) + O(a^4 |x|^{ \epsilon - 3}) + O(a^8 |x|) + O(a^{12}|x|^3),
\end{align}
hence
\begin{align} \label{wDRm1}
w(x)^3 |\nabla_{g^{(1)}} Rm_{g^{(1)}}| \leq C_1.
\end{align}
Finally, (\ref{D2Rmh})
\begin{align} \label{DRm2} \begin{split}
\nabla_{g^{(1)}}^2 Rm_{g^{(1)}} &= O(|x|^{-6}) + O(a^4 |x|^{-4}) + O(a^8 |x|^{-2}) + O(a^4 |x|^{-4 + \epsilon}) \\
& + O(a^8 |x|^{\epsilon -2}) + O(a^8 |x|^{2\epsilon - 4}) + O(a^{12} |x|^2)  + O(a^{16}|x|^4).
\end{split}
\end{align}
Therefore,
\begin{align} \label{wDRm2}
w(x)^4 |\nabla_{g^{(1)}}^2 Rm_{g^{(1)}}| \leq C_2.
\end{align}

The estimates for the other regions are verified in a similar manner, so we omit the details.
\end{proof}

It remains to show that
\begin{align}
H' : \RR \times \RR \times \mathcal{D} \rightarrow C^{0, \alpha}_{\delta - 4}
\end{align}
is an isomorphism with bounded inverse.
This will follow once we prove surjectivity; the bound on the
inverse will then follow immediately from Proposition \ref{unifinj}.

In the following, let us view $H'$ as a map
\begin{align}
H' : \RR \times \RR \times C^{4, \alpha}_{\delta}  \rightarrow C^{0, \alpha}_{\delta - 4}.
\end{align}
Then the formal adjoint of $H'$ maps from
\begin{align}
(H')^* : C^{4, \alpha}_{-\delta} \rightarrow  \RR \times \RR \times C^{0, \alpha}_{-\delta - 4}
\end{align}
and is given by
\begin{align}
(H')^* (h) = \Big(  \int \langle h, k_1 \rangle dV, \int \langle h, k_2 \rangle dV ,
S h \Big)
\end{align}
since $S$ is self-adjoint (the duals of H\"older spaces are not
H\"older spaces, but this slight abuse of notation should not
cause confusion). We claim that for $a$ sufficiently small,
$Ker((H')^*) = 0$. To see this argue by contradiction: let
$h_i$ be a sequence of kernel elements corresponding to
a sequence $a_i \rightarrow 0$ as $i \rightarrow \infty$.
Normalize $h_i$ so that $\Vert h_i \Vert_{C^{4,\alpha}_{- \delta}} = 1$.
We then have a sequence $h_i$ satisfying
\begin{align}
Sh_i = 0,\ \ \  \int \langle h, k_1 \rangle dV &= 0,\ \ \  \int \langle h, k_2 \rangle dV =0,\\
\label{hinorm}
\Vert h_i \Vert_{C^{4,\alpha}_{- \delta}} &= 1.
\end{align}
The limiting argument in the proof of Proposition \ref{unifinj} is then
modified as follows. Let $p_i$ be a sequence of points
in $X_{a_i,b_i}$ for a sequence
$a_i \rightarrow 0$ as $i \rightarrow \infty$
at which the supremum in the norm \eqref{hinorm} is attained.
We have the three possibilities.

\begin{itemize}
\item (1) $p_i \rightarrow p \in N$. In this case, standard elliptic
estimates produce a nontrivial solution of the limiting equation
$S h_{\infty} = 0$ on $(N, g_N)$ with $h \in C^{4, \alpha}_{-\delta}$.
By Theorem \ref{afker}, $h_{\infty} = c_1 \cdot o_1 + c_2 g_N$ for some $c_1, c_2 \in \RR$.
Since
\begin{align}
\int \langle h_i, k_1 \rangle dV_{g^{(1)}_{a_i}} = 0, \ \ \
\int \langle h_i, k_2 \rangle dV_{g^{(1)}_{a_i}} = 0,
\end{align}
and $k_1, k_2$ both have compact support on $N$, and $h_i \leq C w^{-\delta}$,
the integrand is bounded, which implies that
\begin{align}
\int \langle h_\infty, k_1 \rangle dV_{g_N} = 0, \ \ \
\int \langle h_\infty, k_2 \rangle dV_{g_N} = 0,
\end{align}
which implies that $c_1 = c_2 = 0$, a contradiction.
\item (2) $p_i \rightarrow p \in Z \setminus \{z_0\}$.  In this case, defined
$\tilde{h}_i =  (a_i)^{4 + 2 \delta} h_i$. It is easy to see that
this scaling preserves the $C^{4,\alpha}_{\delta}$ norm, with
respect to the metric $\tilde{g}^{(1)}_{a_i} = a^4 {g}^{(1)}_{a_i}$.
Standard elliptic estimates produce a nontrivial solution of the limiting equation
$S \tilde{h}_{\infty} = 0$ on $(Z, g_Z)$ with $h \in C^{4, \alpha}_{-\delta}$.
By Theorem \ref{comker}, $h_{\infty} = 0$ which is a contradiction.
\item (3) If neither of the above cases happen, then as above one
can rescale both the metric and $h_i$ to find a solution $h_{\infty} \in
C^{4, \alpha}_{-\delta}$ of the equation
$S h_{\infty} = 0$ on $\RR^4 \setminus \{0\}$ with weight
function $w = r$. Since $0< - \delta < 1$, $\delta$
is not an indicial root so $S : C^{4,\alpha}_{\delta} \rightarrow C^{0, \alpha}_{\delta -4}$
is an isomorphism, therefore $h_{\infty} = 0$.
\end{itemize}
This contradiction proves that $Ker((H')^*) = \{0\}$, and by
standard Fredholm Theory, we conclude that
\begin{align*}
H' : \RR \times \RR \times C^{4, \alpha}_{\delta}  \rightarrow C^{0, \alpha}_{\delta - 4}.
\end{align*}
is surjective.

\begin{claim}
\label{hclm}
For $a$ sufficiently small, the dimension of the kernel of $H' : \RR \times \RR \times C^{4, \alpha}_{\delta}  \rightarrow C^{0, \alpha}_{\delta - 4}$
is at least 2.
\end{claim}
\begin{proof}
To see this, we claim that $k_1$ and $k_2$ are not in the image
of $S$. If, for example $Sh_i = k_1$, then a limiting argument
as above would produce a solution of $S (h_{\infty}) = k_1$ on $(N, g_N)$,
which is a contradiction. Similarly, if $S h_i = k_2$, the same
argument yields a contradiction. We have found $2$ linearly
independent elements not in the image of $S$; by Fredholm theory the cokernel of $S$ must
be at least two-dimensional. Since $S$ is a self-adjoint
operator, we must have $\dim(Ker(S)) \geq 2$.
Obviously $\{ 0 \} \times \{ 0 \} \times Ker(S) \subset Ker(H')$, so the claim follows.
\end{proof}

To finish, by standard $L^2$-decomposition
\begin{align}
\label{l2d}
L^2 \cap C^{4,\alpha}_{\delta} = \mathrm{span}\{k_1, k_3\} \oplus \mathcal{D},
\end{align}
where $\mathcal{D} =  C^{4,\alpha}_{\delta} \cap (\mathrm{span}\{k_1, k_3\})^{\perp}$.
Let $h_1, h_2, \dots h_j$ be a basis for $Ker(H')$, where $j = \dim(Ker(H'))$.
Then we can write
\begin{align}
h_i = c_{i1} k_1 + c_{i2} k_3 + m_i,
\end{align}
where $m_i \in \mathcal{D}$.
If $j > 2$, then obviously we can take a nontrivial linear combination to obtain
\begin{align}
\label{lineqn}
\sum_i c_i h_i = \sum_i c_i m_i
\end{align}
for some constants $c_i$. The left hand side is in the
kernel of $H'$, but Proposition \ref{unifinj} shows
that the left hand side cannot be, which is a contradiction.
Consequently, from Claim \ref{hclm} we conclude that $\dim(Ker (H')) = 2$.
So we have the equations
\begin{align}
h_1 &= c_{11} k_1 + c_{12} k_3 + m_1\\
h_2 &= c_{21} k_1 + c_{22} k_3 + m_2.
\end{align}
The matrix of coefficients must be an invertible $2 \times 2$ matrix,
since otherwise we could again find a nontrivial solution of \eqref{lineqn}.
Consequently, we can solve
\begin{align}
k_1 &= c_{11}' h_1 + c_{12}' h_2 + m_1'\\
k_3 &= c_{21}' h_1 + c_{22}' h_2 + m_2'.
\end{align}
which, together with \eqref{l2d}, proves the vector space decomposition
\begin{align}
L^2 \cap C^{4,\alpha}_{\delta} = Ker(H') \oplus \mathcal{D}.
\end{align}
Clearly, this proves that $H' : \RR \times \RR \times \mathcal{D} \rightarrow C^{0,\alpha}_{\delta-4}$ is also
surjective.

Finally, the estimate on the size of $\theta$ follows from
Proposition \ref{Btsize}:
\begin{align}
\begin{split}
\Vert H(0,0,0) \Vert_{C^{0,\alpha}_{\delta - 4}}
&= \Vert P(0) \Vert_{C^{0,\alpha}_{\delta - 4}}
= \Vert B^t ( g^{(1)}_{a,b}) \Vert_{C^{0,\alpha}_{\delta - 4}}\\
&\leq \lambda a^4  \Vert k_1  \Vert_{C^{0,\alpha}_{\delta - 4}} +  C a^{4 + \delta - \epsilon}
\leq C a^{4 + \delta - \epsilon}.
\end{split}
\end{align}
\end{proof}

\section{Completion of proofs}
\label{Kuranishi}
The following result immediately implies Theorem \ref{SadjThm}:
\begin{theorem}
\label{Kurthm}
Let  $a = b$ and $\theta \in \mathcal{D}$ be the unique solution of (\ref{bigeqn}):
\begin{align*}
P_{g^{(1)}}(\theta) = \lambda_1 k_1 + \lambda_2 k_2.
\end{align*}
Then
\begin{align}
\label{l1form30}
\lambda_1 =  \lambda a^4 + O(a^{6 - \epsilon})
\end{align}
as $a \rightarrow 0$, where
\begin{align}
\lambda =
 \Big( \frac{2}{3} W(y_0) \circledast W(z_0)  + 4t R(z_0) \mathrm{mass}(g_N)  \Big) \omega_3.
\end{align}
\end{theorem}
\begin{proof}
Let $\theta \in \mathcal{D}$ be a solution of (\ref{bigeqn}):
\begin{align*}
P_{g^{(1)}}(\theta) = \lambda_1 k_1 + \lambda_2 k_2.
\end{align*}
Pairing both sides with $\eta o_1$, where $\eta$ is given in (\ref{etad}), and
integrating (all with respect to the metric $g = g^{(1)}_{a}$) gives
\begin{align}
\int \langle P(\theta), \eta o_1 \rangle dV
 = \lambda_1 \int \langle k_1 , \eta o_1 \rangle dV
+ \lambda_2 \int \langle k_2 , \eta o_1 \rangle dV.
\end{align}
The last integral is identically zero by \eqref{13ortho},
and by \eqref{k1props2}, we obtain
\begin{align}
\label{l1form000}
\lambda_1 = \int \langle P(\theta), \eta o_1 \rangle dV.
\end{align}
Using Proposition \ref{quadest} we expand $P(\theta)$ as
\begin{align*}
P(\theta) &= P(0) + S(\theta) + Q(\theta) \\
&= B^t( g^{(1)}) + S(\theta) + Q(\theta).
\end{align*}
Substituting this into (\ref{l1form000}),
\begin{align}
\label{L1terms}
\lambda_1 = \int \langle B^t( g^{(1)}), \eta o_1 \rangle dV
+ \int \langle S(\theta), \eta o_1 \rangle dV
+ \int \langle Q(\theta), \eta o_1 \rangle dV.
\end{align}
Using \eqref{lamone} (replacing $\h_i$ with $\theta$ in that computation),
we estimate
\begin{align}
\int \langle S(\theta), \eta o_1 \rangle dV
= O(a^{2 - \delta}) O(a^{4 + \delta - \epsilon})
= O(a^{6 - \epsilon})
\end{align}
as $a \rightarrow 0$.
The estimate \eqref{thetanorm} implies the pointwise estimates:
\begin{align}
\label{pointwise}
| \nabla^m \theta | &\leq C a^{4 + \delta - \epsilon} w^{\delta-m},
\end{align}
for $0 \leq m \leq 4$.  Using Proposition \ref{quadest}, the nonlinear term in (\ref{L1terms}) is then estimated
\begin{align}
\int \langle Q(\theta), \eta o_1 \rangle dV = O(a^{8 + 2 \delta - 2 \epsilon}).
\end{align}
We conclude
\begin{align}
\label{l1form20}
\lambda_1 =  \int \langle B^t( g^{(1)}), \eta o_1 \rangle dV + O(a^{6 - \epsilon}),
\end{align}
as $a \rightarrow 0$.

Notice that from (\ref{dsize}),
\begin{align}
\int_{DZ} \langle B^t( g^{(1)}), \eta o_1 \rangle dV = O(a^{6 - \epsilon}),
\end{align}
so we can rewrite \eqref{l1form20} as
\begin{align}
\label{l1form2}
\lambda_1 =  \int_{B } \langle B^t( g^{(1)}), \eta o_1 \rangle dV + O(a^{6 - \epsilon}),
\end{align}
as $a \rightarrow 0$, where $B$ is the same as in \eqref{balldef}.

On $B \subset N$, $\eta \equiv 1$, and from Proposition \ref{BtsizeN} we have
\begin{align} \label{BtS1} \begin{split}
 \int_{B }  \langle B^t( g^{(1)}),   o_1 \rangle\ dV 
&= \int_{B } \langle a^4 \lambda k_1^{(0)} - a^4 \mathcal{K}_{g_N} \delta_{g_N} \mathcal{K}_{g_N} \delta_{g_N} \overset{\circ}{\tilde{H}_2} + O(a^8), o_1 \rangle dV \\
&=  a^4 \lambda \int_{B } \langle k_1^{(0)}, o_1 \rangle\ dV  - a^4  \int_{B } \langle \mathcal{K}_{g_N} \delta_{g_N} \mathcal{K}_{g_N} \delta_{g_N} \overset{\circ}{\tilde{H}_2}, o_1 \rangle\ dV  \\ 
& \hskip.25in + \int_{B } \langle O(a^8), o_1 \rangle\ dV. \\
\end{split}
\end{align}
By Claim \ref{k1claim}, the first integral in (\ref{BtS1}) is 
\begin{align} \begin{split} \label{BtS2}
a^4 \lambda \int_{B } \langle k_1^{(0)}, o_1 \rangle\ dV 
&= a^4 \lambda \int_{B } \langle  \big( 1 + O(a^2) \big) k_1,   o_1 \rangle dV \\
&= a^4 \lambda \int \langle k_1, o_1 \rangle\ dV + O(a^6) \\
&= a^4 \lambda + O(a^6). 
\end{split}
\end{align}

To estimate the second integral in (\ref{BtS1}), we use the fact that on $B \subset N$,
\begin{align*}
g &= g_N + O(a^2).
\end{align*}
In particular, for tensors $T_1, T_2$ we have 
\begin{align*}
\langle T_1, T_2 \rangle &= \big( 1 + O(a^2)\big) \langle T_1 , T_2 \rangle_{g_N},\\
dV &= \big( 1 + O(a^2)\big)dV_{g_N}.
\end{align*}
Therefore, 
\begin{align} \label{reftoN1} \begin{split}
- a^4 \int_{B } & \langle \mathcal{K}_{g_N} \delta_{g_N} \mathcal{K}_{g_N} \delta_{g_N} \overset{\circ}{\tilde{H}_2}, o_1 \rangle\ dV  \\
&= - a^4 \int_{B } \langle \mathcal{K}_{g_N} \delta_{g_N} \mathcal{K}_{g_N} \delta_{g_N} \overset{\circ}{\tilde{H}_2}, o_1 \rangle_{g_N}\ dV_{g_N}  \\
& \hskip.25in +  O(a^6) \int_{B } \big|\mathcal{K}_{g_N} \delta_{g_N} \mathcal{K}_{g_N} \delta_{g_N} \overset{\circ}{\tilde{H}_2}\big|_{g_N} \big| o_1 \big|_{g_N}\ dV_{g_N}.
\end{split}
\end{align}
For the second integral above on the right-hand side of (\ref{reftoN1}), we note that 
\begin{align*}
\big|\mathcal{K}_{g_N} \delta_{g_N} \mathcal{K}_{g_N} \delta_{g_N} \overset{\circ}{\tilde{H}_2}\big|_{g_N} = O(|x|^{\epsilon - 4}),
\end{align*}
see the proof of Proposition \ref{BtsizeN}.  Also, by Theorem \ref{ALEcokethm}, $o_1$ decays quadratically, hence 
\begin{align*}
\int_{B } \big|\mathcal{K}_{g_N} \delta_{g_N} \mathcal{K}_{g_N} \delta_{g_N} \overset{\circ}{\tilde{H}_2}\big|_{g_N} \big| o_1 \big|_{g_N}\ dV_{g_N} = O(1).
\end{align*}

For the first integral on the right in (\ref{reftoN1}) we recall from Theorems \ref{afker} and \ref{ALEcokethm} that the trace-free part of $o_1$ is given by 
\begin{align} \label{Kom1} \begin{split}
\mathcal{K}_{g_N}\omega_1 &= \frac{2}{3} W_{ikj\ell}(y_0) \frac{x^k x^{\ell}}{|x|^4} + O(|x|^{-4 + \epsilon}) \\
&= O(|x|^{-2}) 
\end{split}
\end{align}
as $|x| \rightarrow \infty$.
Moreover, 
\begin{align} \label{Boxkill}
\Box_{g_N} \omega_1 = \delta_{g_N} \mathcal{K}_{g_N} \omega_1 = 0.
\end{align}
Therefore, integration by parts gives 
\begin{align} \label{gaugetail} \begin{split}
- a^4 \int_{B } & \langle \mathcal{K}_{g_N} \delta_{g_N} \mathcal{K}_{g_N} \delta_{g_N} \overset{\circ}{\tilde{H}_2}, o_1 \rangle_{g_N}\ dV_{g_N}  \\
&= - a^4 \int_{B } \langle \mathcal{K}_{g_N} \delta_{g_N} \mathcal{K}_{g_N} \delta_{g_N} \overset{\circ}{\tilde{H}_2}, \mathcal{K}_{g_N}\omega_1 \rangle_{g_N}\ dV_{g_N}  \\
&=  2 a^4 \int_{B } \langle \delta_{g_N} \mathcal{K}_{g_N} \delta_{g_N} \overset{\circ}{\tilde{H}_2}, \delta_{g_N} \mathcal{K}_{g_N}\omega_1 \rangle_{g_N}\ dV_{g_N}  \\
& \hskip.25in - 2 a^4 \oint_{\partial B} \mathcal{K}_{g_N}\omega_1 \big( N, \delta_{g_N} \mathcal{K}_{g_N} \delta_{g_N} \overset{\circ}{\tilde{H}_2} \big)\ dS \\
&= - 2 a^4 \oint_{\partial B} \mathcal{K}_{g_N}\omega_1 \big( N, \delta_{g_N} \mathcal{K}_{g_N} \delta_{g_N} \overset{\circ}{\tilde{H}_2} \big)\ dS. \\
\end{split}
\end{align}
Using (\ref{Kom1}), the integrand of the boundary integral above is 
\begin{align*}
\big| \mathcal{K}_{g_N}\omega_1 \big( N, \delta_{g_N} \mathcal{K}_{g_N} \delta_{g_N} \overset{\circ}{\tilde{H}_2} \big) \big|
= O(|x|^{-2}) \cdot O(|x|^{\epsilon - 3}) = O(|x|^{\epsilon - 5}),
\end{align*}
and it follows that the boundary integral in (\ref{gaugetail}) is of the order 
\begin{align*}
- 2 a^4 \oint_{\partial B} \mathcal{K}_{g_N}\omega_1 \big( N, \delta_{g_N} \mathcal{K}_{g_N} \delta_{g_N} \overset{\circ}{\tilde{H}_2} \big)\ dS = O(a^{6 - \epsilon}). 
\end{align*}
Consequently, 
\begin{align} \label{reftoN2} 
- a^4 \int_{B } & \langle \mathcal{K}_{g_N} \delta_{g_N} \mathcal{K}_{g_N} \delta_{g_N} \overset{\circ}{\tilde{H}_2}, o_1 \rangle\ dV = O(a^{6- \epsilon}).
\end{align}

We can also use the fact that $o_1$ decays quadratically to estimate the last term in (\ref{BtS1}) as 
\begin{align} \label{lastT}
\int_{B } \langle O(a^8), o_1 \rangle\ dV = O(a^6).
\end{align}
Combining (\ref{BtS1}), (\ref{BtS2}), (\ref{reftoN2})  , and (\ref{lastT}) we obtain
\begin{align}
 \int_{B } \langle B^t( g^{(1)}), \eta o_1 \rangle dV
= \lambda a^4  + O(a^{6-\epsilon}).
\end{align}
Proposition \ref{SadjProp} then completes the proof.
\end{proof}
\begin{proof}[Proof of Theorem \ref{thm2}]
From Theorem \ref{Kurthm}, it is clear that for $a$ sufficiently small,
there are two possibilities.
The first is that the remainder term in \eqref{l1form30} is
identically zero for all $a$ sufficiently small.
Choosing $t_0$ as in \eqref{t0}, we
have that $\lambda_1 = 0$. The second possibility is that
the remainder term in \eqref{l1form30} is not zero.
In this case, by an application of the intermediate value
theorem, we may perturb $t$ slightly to again conclude that
$\lambda_1 = 0$. We now have a solution of the equation
\begin{align}
\label{bigeqn2}
P_{g^{(1)}}(\theta) = \lambda_2 k_2.
\end{align}
Recalling the definition of $P$, this is
\begin{align}
B^t ( g^{(1)} + \theta) + \mathcal{K}_{ g^{(1)} + \theta}
\delta_{g^{(1)}} \mathcal{K}_{g^{(1)}} \delta_{g^{(1)}} \overset{\circ}{\theta}
= \lambda_2 k_2.
\end{align}
With respect to the metric  $g^{(1)} + \theta$,
the trace of the left hand side of this equation has mean value zero,
so we have
\begin{align}
0 = \lambda_2 \int tr_{g^{(1)} + \theta} k_2 dV_{ g^{(1)} + \theta}.
\end{align}
Since $k_2$ has compact support in the region where the weight
function is bounded, expanding the trace and volume element 
and using \eqref{pointwise}, we have
\begin{align}
0= \lambda_2  \Big( \int tr_{g^{(1)}} k_2 dV_{ g^{(1)}} + O(a^{4 + \delta - \epsilon}) \big))
= \lambda_2  (1 + O(a^{4 + \delta - \epsilon})),
\end{align}
as $ a \rightarrow 0$, by \eqref{3normal},
which implies that $\lambda_2 = 0$.  We have therefore found
a solution of
\begin{align}
P_{g^{(1)}}(\theta) = 0,
\end{align}
which is a smooth $B^t$-flat metric from Proposition \ref{smoothprop}.

 In the cases of multiple gluing points, imposing the bilateral, trilateral,
or quadrilateral symmetries in the respective cases, reduces
the argument to that of a single gluing point, so the argument is
the same as above.
\end{proof}

\begin{proof}[Proof of Theorem \ref{nonexist}]
In the Bach-flat case, we may restrict all above arguments to pointwise
traceless tensors. The pure-trace kernel and cokernel
elements are then not required in the Lyapunov-Schmidt
reduction in Section \ref{Lyap}.
We then add a $1$-dimensional kernel parameter to the map $H$.
That is, we let
\begin{align}
\mathcal{D} = \Big\{ h \in C^{4,\alpha}_{\delta}\ :\ \int \langle h, k_1 \rangle = 0 \Big\}.
\end{align}
where $k_1$ is of compact support chosen to pair non-trivially
with  $\overset{\circ}{o_1}$, and define the mapping $H :
\mathbb{R} \times \mathbb{R} \times \mathcal{D} \rightarrow C^{\alpha}_{\delta - 4}$ by
\begin{align} \label{Htsdef}
H(s, \lambda_1, \theta) =
P_{g^{(1)}}(\theta + s \eta \overset{\circ}{o_1}) - \lambda_1 k_1.
\end{align}
For gluing parameter $a$ sufficiently small,
the Kuranishi map is then the map
\begin{align}
\Psi: s \mapsto \lambda_1(s)
\end{align}
Using the gauging
argument from \cite[Section 2.3]{GV11},
the fixed point argument in Section~\ref{Lyap} is easily extended to show that
{\em{any}} equivariant Bach-flat metric in a sufficiently small
$C^{4,\alpha}$-neighborhood of the approximate metric will correspond
to a zero of $\Psi$ for some $s$.
If $W(y_0) \circledast W(z_0) \neq 0$, then the leading term of
$\Psi$ is non-zero, so obviously there
can be no equivariant Bach-flat metric in a sufficiently small
$C^{4,\alpha}$ neighborhood of the approximate metric.
\end{proof}
\subsection{Computation of values in Table \ref{table}}

Assume $W^{\pm}, \tW^{\pm}$ are trace-free endomorphisms of $\Lambda^{\pm}_2(V^{*})$, where $V$ is a real, oriented, four-dimensional inner product space,
and write
\begin{align*}
W &= W^{+} + W^{-} : \Lambda^2(V^{*}) \rightarrow \Lambda^2(V^{*}), \\
\tW &= \tW^{+} + \tW^{-} : \Lambda^2(V^{*}) \rightarrow \Lambda^2(V^{*}).
\end{align*}
Assume further that $W$ and $\tW$ can be simultaneously
diagonalized: that is, there is an orthogonal basis of eigenvectors (two-forms) for $W$ \underline{and} $\tW$ denoted
\begin{align} \label{Wbasis}
\omega, \eta, \theta, \omega^{-}, \eta^{-}, \theta^{-},
\end{align}
where the first three are a basis of $\Lambda_2^{+}(V^{\ast})$ and the last three a basis of $\Lambda_2^{-}(V^{\ast})$.
Denote the
eigenvalues of $W$ and $\tW$ as
\begin{align} \label{Wevalues} \begin{split}
\mbox{spec}(W) &= \{ \lambda, \mu, \nu, \lambda^{-}, \mu^{-}, \nu^{-} \}, \\
\mbox{spec}(\tW) &= \{ \tlambda, \tmu, \tnu, \tlambda^{-}, \tmu^{-}, \tnu^{-} \}. \end{split}
\end{align}
We will further assume  that $W^{\pm}$ and $\tW^{\pm}$ are trace-free; i.e.,   \begin{align} \label{WtWtrace} \begin{split}
\lambda + \mu + \nu &= 0, \ \lambda^{-} + \mu^{-} + \nu^{-} = 0,  \\
\tlambda + \tmu + \tnu &= 0, \ \tlambda^{-} + \tmu^{-} + \tnu^{-} = 0.
\end{split}
\end{align}
Using this basis in (\ref{Wbasis}) we can write
\begin{align} \label{WtWforms} \begin{split}
W &=  \frac{1}{2} \big\{ \lambda \omega \otimes \omega + \mu \eta \otimes \eta + \nu \theta \otimes \theta  \\
&\hskip.25in + \lambda^{-} \omega^{-} \otimes \omega^{-} + \mu^{-} \eta^{-} \otimes \eta^{-} + \nu^{-} \theta^{-} \otimes \theta^{-} \big\}, \\
\tW &=  \frac{1}{2} \big\{ \tlambda \omega \otimes \omega + \tmu \eta \otimes \eta + \tnu \theta \otimes \theta  \\
&\hskip.25in + \tlambda^{-} \omega^{-} \otimes \omega^{-} + \tmu^{-} \eta^{-} \otimes \eta^{-} + \tnu^{-} \theta^{-} \otimes \theta^{-} \big\}.
\end{split}
\end{align}

We normalize the eigenforms to have length $\sqrt{2}$; this convention gives the identities
\begin{align} \label{multrules1}
\omega^2 = \eta^2 = \cdots = (\theta^{-})^2 = Id.
\end{align}

We also point out two more important algebraic facts: first, the product of any SD basis element with any ASD basis element gives
a symmetric trace-free two-tensor, whose square is the identity.  Thus, for example,
\begin{align} \label{multrules2}
h = \omega \omega^{-} \Rightarrow \ tr\ h = 0,\ h^2 = Id,\ |h|^2 = 4.
\end{align}
Also, the bases of $\Lambda_2^{\pm}$ give a quaternionic structure satisfying the following multiplication rules:
\begin{align}
\omega \eta = \theta, \ \eta \theta = \omega, \ \theta \omega = \eta.
\end{align}

\begin{lemma} \label{WIdLemma} Fix an orthonormal basis $\{ v_1,\dots,v_4 \}$  of $V$, and let $W_{i  k j \ell}$ (resp., $\tW_{i \ell j k}$) denote the components of $W$ (resp., $\tW$) with respect to this basis.  Then
\begin{align} \label{WW} \begin{split}
W_{ikj\ell} \tW_{i\ell j k} &=  2 \big\{ \lambda \tlambda + \mu \tmu + \nu \tnu + \lambda^{-} \tlambda^{-} + \mu^{-} \tmu^{-} +  \nu^{-} \tnu^{-}\big\} \\
     &=      2 \langle W, \tW \rangle
= \frac{1}{2} W_{ikj\ell} \tW_{ikj\ell}.
     \end{split}
\end{align}
In particular, the answer is independent of the choice of basis.
\end{lemma}

\begin{proof}
By (\ref{WtWforms}),
\begin{align} \label{Wprod1} \begin{split}
4 W_{ikj\ell} \tW_{i \ell j k} &= \big\{ \lambda \omega_{ik} \omega_{j\ell} + \mu \eta_{ik} \eta_{j\ell} + \nu \theta_{ik} \theta_{j\ell}  \\
&\hskip.25in + \lambda^{-} \omega_{ik}^{-} \omega_{j\ell}^{-} + \mu^{-} \eta_{ik}^{-} \eta_{j\ell}^{-} + \nu^{-} \theta_{ik}^{-} \theta_{j\ell}^{-} \big\} \\
& \times \big\{ \tlambda \omega_{i\ell} \omega_{jk} + \tmu \eta_{i\ell} \eta_{jk} + \tnu \theta_{i\ell} \theta_{jk}  \\
&\hskip.25in + \tlambda^{-} \omega_{i\ell}^{-} \omega_{jk}^{-} + \tmu^{-} \eta_{i\ell}^{-} \eta_{jk}^{-} + \tnu^{-} \theta_{i\ell}^{-} \theta_{jk}^{-} \big\}.
\end{split}
\end{align}
As we multiply and distribute we see that there are six kinds of terms, which we represent schematically as
\begin{align} \label{pairs} \begin{split}
(A_{ik}  A_{j\ell}) \cdot (A_{i \ell} A_{jk}), & \ \ (A^{-}_{ik} A^{-}_{j \ell} \cdot (A^{-}_{i\ell} A_{jk}^{-}) \\
(A_{ik} A_{j\ell}) \cdot (B_{i\ell} B_{jk}), & \ \ (A^{-}_{ik} A^{-}_{j\ell}) \cdot (B^{-}_{i\ell} B^{-}_{jk}) \\
(A_{ik} A_{j\ell} ) \cdot ( C^{-}_{i\ell} C^{-}_{jk}), &\ \ (A^{-}_{ik} A^{-}_{j\ell}) \cdot (C_{i\ell} C_{jk}),
\end{split}
\end{align}
where $A,B,C$ are self-dual and $A^{-}, B^{-}$, and $C^{-}$ are anti-self-dual. Using the multiplication rules in (\ref{multrules1}) and
(\ref{multrules2}), we find
\begin{align} \label{pairrules} \begin{split}
(A_{ik}  A_{j\ell}) \cdot (A_{i \ell} A_{jk}) = 4, & \ \ (A^{-}_{ik} A^{-}_{j \ell} \cdot (A^{-}_{i\ell} A_{jk}^{-}) = 4,\\
(A_{ik} A_{j\ell}) \cdot (B_{i\ell} B_{jk}) = -4, & \ \ (A^{-}_{ik} A^{-}_{j\ell}) \cdot (B^{-}_{i\ell} B^{-}_{jk}) = -4, \\
(A_{ik} A_{j\ell} ) \cdot ( C^{-}_{i\ell} C^{-}_{jk}) = -4, &\ \ (A^{-}_{ik} A^{-}_{j\ell}) \cdot (C_{i\ell} C_{jk}) = -4.
\end{split}
\end{align}
Therefore, after multiplying out and collecting all the terms in (\ref{Wprod1}), we find
\begin{align} \label{Wprod2} \begin{split}
4 W_{ikj\ell} \tW_{i \ell j k} &= 4\lambda ( \tlambda -  \tmu - \tnu) + 4 \lambda (-\tlambda^{-} -  \tmu^{-} - \tnu^{-} ) \\
& + 4 \mu (-\tlambda + \tmu - \tnu) + 4\mu (-\tlambda^{-}  - \tmu^{-} -  \tnu^{-} ) \\
& + 4\nu ( -\tlambda - \tmu + \tnu) + 4 \nu( -\tlambda^{-} - \tmu^{-} - \tnu^{-} ) \\
& + 4 \lambda^{-} ( -\tlambda - \tmu - \tnu) + 4 \lambda^{-} (\tlambda^{-} -  \tmu^{-} - \tnu^{-} ) \\
& + 4 \mu^{-} (- \tlambda - \tmu - \tnu) + 4\mu^{-} (-\lambda^{-} + \mu^{-} - \nu^{-} ) \\
& + 4\nu^{-} ( - \tlambda - \tmu - \tnu) + 4 \nu^{-}( -\tlambda^{-} - \tmu^{-} +  \tnu^{-} ).
\end{split}
\end{align}
By (\ref{WtWtrace}), this gives \begin{align} \label{Wnorm1} \begin{split}
4 W_{ikj\ell} \tW_{i \ell j k} &= 8 \big\{ \lambda \tlambda + \mu \tmu + \nu \tnu + \lambda^{-} \tlambda^{-} + \mu^{-} \tmu^{-} +  \nu^{-} \tnu^{-}\big\} \\
&=  8 \langle W, \tW \rangle,
\end{split}
\end{align}
and (\ref{WW}) follows.
\end{proof}

To compute the values of $t_0$, we note that
in the coordinate system $\{z^i\}$ given in Section \ref{bb}, letting
\begin{align*}
\omega^{\pm} &= e^1 \wedge e^2 \pm  e^3 \wedge e^4,\\
\eta^{\pm} & = e^1 \wedge e^3 \mp  e^2 \wedge e^4,\\
\theta^{\pm} & = e^1 \wedge e^4 \pm  e^2 \wedge e^3,
\end{align*}
with $(e^1, e^3, e^3, e^4) = (dz^1, dz^2, dz^3, dz^4)$,
we have
\begin{align*}
W^+(g_{FS}) &= \mathrm{diag} ( R/6, - R/12, - R/12) = \mathrm{diag}( 4, -2, -2),\\
W^-(g_{FS}) &= \mathrm{diag}(0,0,0),\\
W^+(g_{S^2 \times S^2}) &= \mathrm{diag} ( R/6, - R/12, - R/12)
= \mathrm{diag}( 2/3, -1/3, -1/3),\\
W^-(g_{S^2 \times S^2}) & = \mathrm{diag}( 2/3, -1/3, -1/3).
\end{align*}
In case (i), since mass$(\hat{g}_{FS}) =2$,
\begin{align}
t_0 = \frac{-1}{6 \cdot 24 \cdot 2} 4 ( 16 + 4 + 4) = - \frac{1}{3}.
\end{align}
In case (ii) with a Burns metric attached, we have
\begin{align}
t_0 = \frac{-1}{6 \cdot 4 \cdot 2} 4 ( 4\cdot (2/3) + 2 \cdot(1/3) + 2 \cdot (1/3))
= - \frac{1}{3}.
\end{align}
Case (v) has the same value as this.

In case (ii) with a Green's function $S^2 \times S^2$ attached,
\begin{align}
\begin{split}
t_0 &= \frac{-1}{6 \cdot 24 \cdot \mathrm{mass}(\hat{g}_{S^2 \times S^2}) }
4 ( (2/3) \cdot 4 + (1/3) \cdot 2 + (1/3) \cdot 2) \\
&=
- \frac{1}{9 \cdot \mathrm{mass}(\hat{g}_{S^2 \times S^2})}.
\end{split}
\end{align}
In case (iii),
\begin{align}
\begin{split}
t_0 & =  \frac{-1}{6 \cdot 4 \cdot   \mathrm{mass}(\hat{g}_{S^2 \times S^2})}
 ( (2/3)^2 + (1/3)^2 + (1/3)^2 + (2/3)^2 + (1/3)^2 + (1/3)^2) \\
& = - \frac{2}{9 \cdot \mathrm{mass}(\hat{g}_{S^2 \times S^2})}.
\end{split}
\end{align}
Case (iv) has the same value of $t_0$ as does case (iii).

 All other cases are computed similarly as the above cases, so it
is not necessary to write every case here. We only need mention the
fact that in all non-orientable cases,
the answer does not depend on choice of local orientation.
\appendix
\section{Proof of Proposition \ref{Ricprop}}
\label{append}
To prove the proposition,
we use the expansion of the metric in AF coordinates,
\begin{align} \label{gALEexp}
g_{\mu \nu} = \delta_{\mu \nu} - \frac{1}{3} R_{\mu k \nu \ell}(y_0)\frac{x^k x^{\ell}}{|x|^4} + \frac{2A}{|x|^2}\delta_{\mu \nu} + O(|x|^{-3}).
\end{align}
In terms of the Christoffel symbols, the Ricci tensor is given by
\begin{align} \label{ChrtoRic} \begin{split}
R_{ij} &= \partial_{m} \Gamma_{ij}^{m} - \partial_i \Gamma_{j m}^{m} + \Gamma * \Gamma \\
&= -\frac{1}{2} g^{m p} \big\{ \partial_{m} \partial_{p} g_{ij} - \partial_{m} \partial_i g_{ j p} - \partial_{m} \partial_j g_{i p} + \partial_{i} \partial_j g_{m p} \big\} + \partial g * \partial g.
\end{split}
\end{align}
By the expansion (\ref{gALEexp}),
\begin{align} \label{gALEup}
g^{m p} = \delta_{m p}  + O(|x|^{-2}),
\end{align}
hence
\begin{align} \label{ChrtoRic2} \begin{split}
R_{ij} &= -\frac{1}{2} \big\{ \delta_{m p} + O(|x|^{-2})\big\} \big\{ \partial_{m} \partial_{p} g_{ij} - \partial_{m} \partial_i g_{ j p} - \partial_{m} \partial_j g_{i p} + \partial_{i} \partial_j g_{m p} \big\} + \partial g * \partial g \\
&= -\frac{1}{2}  \big\{ \partial_{m} \partial_{m} g_{ij} - \partial_{m} \partial_i g_{ jm} - \partial_{m} \partial_j g_{i m} + \partial_{i} \partial_j g_{m m} \big\} + (\partial^2 g)*(|x|^{-2}) +  \partial g * \partial g.
\end{split}
\end{align}
In AF coordinates, $\partial g = O(|x|^{-3}), \partial^2 g = O(|x|^{-4})$; hence
\begin{align} \label{ChrtoRic3} \begin{split}
R_{ij} &= -\frac{1}{2}  \big\{ \partial_{m} \partial_{m} g_{ij} - \partial_{m} \partial_i g_{ jm} - \partial_{m} \partial_j g_{i m} + \partial_{i} \partial_j g_{m m} \big\} +  O(|x|^{-6}).
\end{split}
\end{align}

By (\ref{gALEexp}),
\begin{align} \label{DgALE} \begin{split}
\partial_{\beta} g_{\mu \nu} &= -\frac{1}{3} R_{\mu \beta \nu k}(y_0)\frac{x^k}{|x|^4} - \frac{1}{3} R_{\mu k \nu \beta}(y_0)\frac{x^k}{|x|^4} \\
&\hskip.25in + \frac{4}{3} R_{\mu k \nu \ell}(y_0) \frac{ x^k x^{\ell} x^{\beta}}{|x|^6} - \frac{4A}{|x|^4} x^{\beta} \delta_{\mu \nu} + O(|x|^{-4}),
\end{split}
\end{align}
\begin{align} \label{D2gALE} \begin{split}
\partial_{\alpha} \partial_{\beta} g_{\mu \nu} &= -\frac{1}{3}R_{\mu \beta \nu \alpha}(y_0)\frac{1}{|x|^4} - \frac{1}{3}R_{\mu \alpha \nu \beta}(y_0)\frac{1}{|x|^4}  \\
&\hskip.2in + \frac{4}{3} R_{\mu \beta \nu k}(y_0) \frac{x^k x^{\alpha}}{|x|^6} + \frac{4}{3}R_{\mu k \nu \beta}(y_0)\frac{x^k x^{\alpha}}{|x|^6} \\
&\hskip.2in + \frac{4}{3} R_{\mu \alpha \nu k}(y_0) \frac{x^k x^{\beta}}{|x|^6} + \frac{4}{3}R_{\mu k \nu \alpha}(y_0)\frac{x^k x^{\beta}}{|x|^6}  \\
&\hskip.2in + \frac{4}{3}R_{\mu k \nu \ell}(y_0)\frac{x^k x^{\ell}}{|x|^6}\delta_{\alpha \beta} - 8 R_{\mu k \nu \ell}(y_0)\frac{x^k x^{\ell} x^{\alpha} x^{\beta}}{|x|^8}  \\
& \hskip.2in - \frac{4A}{|x|^4}\delta_{\alpha \beta} \delta_{\mu \nu} + \frac{16A}{|x|^6}x^{\alpha} x^{\beta} \delta_{\mu \nu} + O(|x|^{-5}).
\end{split}
\end{align}
Consequently, the first term in (\ref{ChrtoRic3}) is
\begin{align} \label{1Ric1} \begin{split}
\partial_m \partial_m g_{ij} &= -\frac{1}{3}R_{i m j m}(y_0)\frac{1}{|x|^4} - \frac{1}{3}R_{i m j m}(y_0)\frac{1}{|x|^4}  \\
&\hskip.2in + \frac{4}{3} R_{i m j k}(y_0) \frac{x^k x^{m}}{|x|^6} + \frac{4}{3}R_{i k j m}(y_0)\frac{x^k x^{m}}{|x|^6} \\
&\hskip.2in + \frac{4}{3} R_{i m j k}(y_0) \frac{x^k x^{m}}{|x|^6} + \frac{4}{3}R_{i k j m}(y_0)\frac{x^k x^{m}}{|x|^6}  \\
&\hskip.2in + \frac{4}{3}R_{i k j \ell}(y_0)\frac{x^k x^{\ell}}{|x|^6}\delta_{m m} - 8 R_{i k j \ell}(y_0)\frac{x^k x^{\ell} x^{m} x^{m}}{|x|^8}  \\
& \hskip.2in - \frac{4A}{|x|^4}\delta_{m m} \delta_{i j} + \frac{16A}{|x|^6}x^{m} x^{m} \delta_{i j} + O(|x|^{-5}).
\end{split}
\end{align}
The first two terms combine to give a Ricci curvature term, while the third through the eighth terms are all the same (though with different
coefficients); adding up we get
\begin{align} \label{1Ric}
\partial_m \partial_m g_{ij} = -\frac{2}{3} R_{ij}(y_0) \frac{1}{|x|^4} + \frac{8}{3}R_{i k j \ell}(y_0) \frac{ x^k x^{\ell}}{|x|^6} + O(|x|^{-5}).
\end{align}
The second term is
\begin{align} \label{2Ric} \begin{split}
\partial_{m} \partial_i g_{ jm} &= -\frac{1}{3}R_{j i m m}(y_0)\frac{1}{|x|^4} - \frac{1}{3}R_{j m m i}(y_0)\frac{1}{|x|^4}  \\
&\hskip.2in + \frac{4}{3} R_{j i m k}(y_0) \frac{x^k x^{m}}{|x|^6} + \frac{4}{3}R_{j k m i}(y_0)\frac{x^k x^{m}}{|x|^6} \\
&\hskip.2in + \frac{4}{3} R_{j m m k}(y_0) \frac{x^k x^{i}}{|x|^6} + \frac{4}{3}R_{j k m m}(y_0)\frac{x^k x^{i}}{|x|^6}  \\
&\hskip.2in + \frac{4}{3}R_{j k m \ell}(y_0)\frac{x^k x^{\ell}}{|x|^6}\delta_{m i} - 8 R_{j k m \ell}(y_0)\frac{x^k x^{\ell} x^{m} x^{i}}{|x|^8}  \\
& \hskip.2in - \frac{4A}{|x|^4}\delta_{m i} \delta_{j m} + \frac{16A}{|x|^6}x^{m} x^{i} \delta_{j m} + O(|x|^{-5}),
\end{split}
\end{align}
notice that the fourth and seventh terms cancel each other, while the first, sixth, and eight terms vanish because of the skew-symmetry of the curvature tensor.  Also,
the second and fifth terms are traces.  Therefore,
\begin{align} \label{3Ric}
\partial_{m} \partial_i g_{ jm} = \frac{1}{3} R_{ij}(y_0) \frac{1}{|x|^4} - \frac{4}{3} R_{jk}(y_0) \frac{x^i x^k}{|x|^6} - \frac{4A}{|x|^4} \delta_{ij} + \frac{16A}{|x|^6}x^i x^j + O(|x|^{-5}),
\end{align}
while the third term is
\begin{align} \label{3bRic}
\partial_{m} \partial_j g_{ im} = \frac{1}{3} R_{ij}(y_0) \frac{1}{|x|^4} - \frac{4}{3} R_{ik}(y_0) \frac{x^j x^k}{|x|^6} - \frac{4A}{|x|^4} \delta_{ij} + \frac{16A}{|x|^6}x^i x^j + O(|x|^{-5}).
\end{align}
The last term in (\ref{ChrtoRic3}) is
\begin{align} \label{4Ric1} \begin{split}
\partial_i \partial_j g_{mm} &= -\frac{1}{3}R_{m j m i}(y_0)\frac{1}{|x|^4} - \frac{1}{3}R_{m i m j}(y_0)\frac{1}{|x|^4}  \\
&\hskip.2in + \frac{4}{3} R_{m j m k}(y_0) \frac{x^k x^{i}}{|x|^6} + \frac{4}{3}R_{m k m j}(y_0)\frac{x^k x^{i}}{|x|^6} \\
&\hskip.2in + \frac{4}{3} R_{m i m k}(y_0) \frac{x^k x^{j}}{|x|^6} + \frac{4}{3}R_{m k m i}(y_0)\frac{x^k x^{j}}{|x|^6}  \\
&\hskip.2in + \frac{4}{3}R_{m k m \ell}(y_0)\frac{x^k x^{\ell}}{|x|^6}\delta_{i j} - 8 R_{m k m \ell}(y_0)\frac{x^k x^{\ell} x^{i} x^{j}}{|x|^8}  \\
& \hskip.2in - \frac{4A}{|x|^4}\delta_{i j} \delta_{m m} + \frac{16A}{|x|^6}x^{i} x^{j} \delta_{m m} + O(|x|^{-5}).
\end{split}
\end{align}
In this case all the curvature terms involve traces, so we get
\begin{align} \label{4Ric} \begin{split}
\partial_i \partial_j g_{mm} &= -\frac{2}{3} R_{ij}(y_0) \frac{1}{|x|^4} + \frac{8}{3} R_{jk}(y_0) \frac{x^i x^k}{|x|^6} + \frac{8}{3} R_{ik}(y_0) \frac{x^j x^k}{|x|^6} \\
& \hskip.2in + \frac{4}{3}R_{k \ell}(y_0) \frac{ x^k x^{\ell}}{|x|^6} \delta_{ij} - 8 R_{k \ell}(y_0) \frac{x^k x^{\ell} x^i x^j }{|x|^8} \\
& \hskip.2in - \frac{16A}{|x|^4} \delta_{ij} + \frac{64A}{|x|^4}x^i x^j + O(|x|^{-5}).
\end{split}
\end{align}
Combining (\ref{1Ric})--(\ref{4Ric}),
\begin{align} \label{sumg1} \begin{split}
\partial_{m} \partial_{m} g_{ij} & - \partial_{m} \partial_i g_{ jm} - \partial_{m} \partial_j g_{i m} + \partial_{i} \partial_j g_{m m} = \\
& \frac{8}{3} R_{ikj\ell}(y_0) \frac{x^k x^{\ell}}{|x|^6}  - 2 R_{ij}(y_0) \frac{1}{|x|^4} + 4 R_{jk}(y_0) \frac{x^i x^k}{|x|^6} +  4 R_{ik}(y_0) \frac{x^j x^k}{|x|^6} \\
& \hskip.2in + \frac{4}{3}R_{k \ell}(y_0) \frac{ x^k x^{\ell}}{|x|^6} \delta_{ij} - 8 R_{k \ell}(y_0) \frac{x^k x^{\ell} x^i x^j }{|x|^8} - \frac{8A}{|x|^4} \delta_{ij} + \frac{32A}{|x|^6}x^i x^j + O(|x|^{-5}).
\end{split}
\end{align}
We now use the fact that $(Y, g_Y)$ is Einstein, and that $\{ y^i \}$ are normal coordinates centered at $y_0$:
\begin{align*}
R_{ij}(y_0) &= \frac{1}{4}R(y_0)\delta_{ij}, \\
R_{ij k \ell}(y_0) &= W_{i j k \ell} (y_0) + \frac{1}{12}R(y_0)\big( \delta_{ik} \delta_{j \ell} - \delta_{i \ell} \delta_{j k} \big).
\end{align*}
Substituting these gives
\begin{align}  \label{sumg2} \begin{split}
\partial_{m} \partial_{m} g_{ij} - \partial_{m} \partial_i g_{ jm} &- \partial_{m} \partial_j g_{i m} + \partial_{i} \partial_j g_{m m} = \\
& \frac{8}{3} W_{ikj\ell}(y_0) \frac{x^k x^{\ell}}{|x|^6} + \frac{1}{18}R(y_0) \frac{1}{|x|^4} \delta_{ij} - \frac{2}{9}R_(y_0)\frac{x^i x^j}{|x|^6}\\
&  - \frac{8A}{|x|^4} \delta_{ij} + \frac{32A}{|x|^6}x^i x^j + O(|x|^{-5}),
\end{split}
\end{align}
and (\ref{Ricex}) follows.

\section{Non-simply-connected examples}
\label{appb}
\begin{table}[h]
\caption{Non-simply-connected examples with more than one bubble}
\centering
\begin{tabular}{lll}
\hline\hline
Topology of connected sum& Value of $t_0$ & Symmetry\\
\hline
$ S^2 \times S^2 \# 2 (S^2 \times S^2 / \ZZ_2)$ &  $- 2(9 m_2)^{-1}$
& bilateral \\
$ S^2 \times S^2 \# 2( \RP^2 \times \RP^2)$  &  $- 2(9 m_3)^{-1}$
& bilateral\\
$\CP^2 \# 3 (S^2 \times S^2 / \ZZ_2)$ &  $- (9 m_2)^{-1}$
& trilateral\\
$\CP^2 \# 3 ( \RP^2 \times \RP^2)$ &  $- (9 m_3)^{-1}$
& trilateral\\
$S^2 \times S^2 \# 4 (S^2 \times S^2 / \ZZ_2)$ & $- 2(9 m_2)^{-1}$
& quadrilateral\\
$S^2 \times S^2 \# 4( \RP^2 \times \RP^2) $ & $- 2(9 m_3)^{-1}$
& quadrilateral\\
$(S^2 \times S^2 / \ZZ_2) \# 2 \overline{\CP}^2$ & $- 1/3$ & bilateral
\\
$(S^2 \times S^2 / \ZZ_2) \# 2 (S^2 \times S^2)$ &  $- 2(9 m_1)^{-1}$  & bilateral
\\
$3 \# (S^2 \times S^2 / \ZZ_2)$ &  $- 2(9 m_2)^{-1}$  & bilateral
\\
$(S^2 \times S^2 / \ZZ_2) \# 2( \RP^2 \times \RP^2)$  &  $- 2(9 m_3)^{-1}$  & bilateral
\\
\hline
\end{tabular}\label{tableapp}
\end{table}
All non-simply-connected possiblities with more than one bubble
are listed in Table~\ref{tableapp}.
The approximate metric in each case is obtained by using
the first factor as the compact manifold, with the
AF space clear from the latter factors.
The first two cases are analogous to Cases (iv) and (v).
The third and fourth are analogous to Cases (vi) and (vii).
The fifth and six are analogous to Cases (vii) and (ix).
The last four cases require a short explanation.
In the case of $S^2 \times S^2 / \ZZ_2$, there are two fixed
points: the equivalence classes of $(n,n)$ and $(n,s)$.
The diagonal symmetry descends to the quotient, and fixes
both of these points. The symmetry of reflection in a horizontal
line descends to the quotient, and this interchanges the fixed
points, and we again call this invariance bilateral symmetry.
We may therefore glue on the same AF space at each fixed point,
and require bilateral symmetry, which yields the last
four cases in Table \ref{tableapp}.

\bibliography{Critical_references}
\end{document}